\documentclass[reqno,11pt]{amsart}
\usepackage{amsthm,amsmath,amsfonts,amssymb}
\usepackage{color}
\usepackage{bold-extra} 

 \usepackage[pagebackref]{hyperref}
\hypersetup{
  colorlinks=true,
  citecolor = blue, 
  linkcolor= blue,
}

\usepackage{enumerate}
\usepackage{enumitem} 
\setlist[enumerate,1]  {label={\rm (\roman*)}, leftmargin=1.5em}   
\setlist{noitemsep} 

\usepackage{mathrsfs} 

\usepackage{etoolbox}
\patchcmd{\section}{\normalfont}{\bfseries}{}{}
\makeatletter
\renewcommand{\@secnumfont}{\bfseries}
\makeatother

\usepackage[all]{xy}
\usepackage{pb-diagram,pb-xy} 

\usepackage{mdframed}
\usepackage{xcolor}
{\vspace{0,5\baselineskip}
  \begin{mdframed}[backgroundcolor=lime]}
  {\end{mdframed}
  \vspace{0,5\baselineskip}
}
{\vspace{0,5\baselineskip}
  \begin{mdframed}[backgroundcolor=pink]}
  {\end{mdframed}
  \vspace{0,5\baselineskip}
}
 \usepackage{version}
\excludeversion{reserve} 
\excludeversion{reserve2e}

 \usepackage{color,soul}

\usepackage{soul}
\usepackage{tikz}
\usetikzlibrary{calc}
\usetikzlibrary{decorations.pathmorphing}

\makeatletter

\newcommand{\defhighlighter}[3][]{%
  \tikzset{every highlighter/.style={color=#2, fill opacity=#3, #1}}%
}

\defhighlighter{yellow}{.5}

\newcommand{\highlight@DoHighlight}{
  \fill [ decoration = {random steps, amplitude=1pt, segment length=15pt}
        , outer sep = -15pt, inner sep = 0pt, decorate
        , every highlighter, this highlighter ]
        ($(begin highlight)+(0,8pt)$) rectangle ($(end highlight)+(0,-3pt)$) ;
}

\newcommand{\highlight@BeginHighlight}{
  \coordinate (begin highlight) at (0,0) ;
}

\newcommand{\highlight@EndHighlight}{
  \coordinate (end highlight) at (0,0) ;
}

\newdimen\highlight@previous
\newdimen\highlight@current

\DeclareRobustCommand*\highlight[1][]{%
  \tikzset{this highlighter/.style={#1}}%
  \SOUL@setup
  \def\SOUL@preamble{%
    \begin{tikzpicture}[overlay, remember picture]
      \highlight@BeginHighlight
      \highlight@EndHighlight
    \end{tikzpicture}%
  }%
  \def\SOUL@postamble{%
    \begin{tikzpicture}[overlay, remember picture]
      \highlight@EndHighlight
      \highlight@DoHighlight
    \end{tikzpicture}%
  }%
  \def\SOUL@everyhyphen{%
    \discretionary{%
      \SOUL@setkern\SOUL@hyphkern
      \SOUL@sethyphenchar
      \tikz[overlay, remember picture] \highlight@EndHighlight ;%
    }{%
    }{%
      \SOUL@setkern\SOUL@charkern
    }%
  }%
  \def\SOUL@everyexhyphen##1{%
    \SOUL@setkern\SOUL@hyphkern
    \hbox{##1}%
    \discretionary{%
      \tikz[overlay, remember picture] \highlight@EndHighlight ;%
    }{%
    }{%
      \SOUL@setkern\SOUL@charkern
    }%
  }%
  \def\SOUL@everysyllable{%
    \begin{tikzpicture}[overlay, remember picture]
      \path let \p0 = (begin highlight), \p1 = (0,0) in \pgfextra
        \global\highlight@previous=\y0
        \global\highlight@current =\y1
      \endpgfextra (0,0) ;
      \ifdim\highlight@current < \highlight@previous
        \highlight@DoHighlight
        \highlight@BeginHighlight
      \fi
    \end{tikzpicture}%
    \the\SOUL@syllable
    \tikz[overlay, remember picture] \highlight@EndHighlight ;%
  }%
  \SOUL@
}
\makeatother

\newcommand{\R}{{\mathbb R}}
\newcommand{\N}{{\mathbb N}}
\newcommand{\C}{{\mathbb C}}

\newcommand{\eps}{\varepsilon}

\newcommand{\myeq}[1]{\ensuremath{\stackrel{\text{#1}}{=}}}
\newcommand{\mydef}{\ensuremath{\stackrel{\text{def}}{:=}}}

\newcommand{\Norm}[1]{|||#1|||}

\newtheorem{theorem}{Theorem}[section]

\newtheorem{remark}[theorem]{Remark}
\newtheorem{lemma}[theorem]{Lemma}
\newtheorem{definition}[theorem]{Definition}
\newtheorem{proposition}[theorem]{Proposition}
\newtheorem{corollary}[theorem]{Corollary}

\numberwithin{equation}{section}

\renewenvironment{proof}[1][Proof]{\noindent \textbf{#1.} }
{\  \rule{0.5em}{0.5em}\par \medskip}

\title{
On linear Schrödinger parabolic  problems in Morrey spaces
}

\thanks{Partially supported by Projects PID2019-103860GB-I00 and PID2022-137074NB-I00, MINECO, Spain.
\newline{$\phantom{ai}$ Key words and phrases: initial value problems for higher-order parabolic
equations, fractional partial differential equations}.
\newline{$\phantom{ai}$ Mathematical Subject Classification
2020:\ 35K30, 35R11}.
}

\thanks{${}^{*}$Partially supported by Severo
Ochoa Grant CEX2019-000904-S funded by MCIN/AEI/10.13039/ 501100011033.}

\begin{document}

\date{\bf \today}
\maketitle

\begin{center}
{\sc Jan W. Cholewa}${}^1\ $  \ {\sc and} \ {\sc Anibal Rodriguez-Bernal}${}^\text{2}$
\end{center}

\makeatletter
\begin{center}
${}^{1}$Institute of
Mathematics\\
University of Silesia in Katowice\\ 40-007 Katowice, Poland\\ {E-mail:
jan.cholewa@us.edu.pl}
\\ \mbox{}
\\
${}^{2}$Departamento de Análisis Matemático y  Matem\'atica Aplicada\\ Universidad
  Complutense de Madrid\\ 28040 Madrid, Spain\\ and \\
  Instituto de Ciencias Matem\'aticas \\
CSIC-UAM-UC3M-UCM${}^{*}$, Spain  \\ {E-mail:
arober@ucm.es}
\end{center}
\makeatother

\setcounter{tocdepth}{3}
\makeatletter
\def\l@subsection{\@tocline{2}{0pt}{2.5pc}{5pc}{}}
\def\l@subsubsection{\@tocline{2}{0pt}{5pc}{7.5pc}{}}
\makeatother

\begin{abstract}
We consider parabolic Schrödinger type equations associated to
fractional powers of uniformly elliptic $2m$-order operators with
constant coefficients. Potentials and initial data are considered in
suitable Morrey spaces.
By means of perturbation techniques we prove that several properties
of the problem with no potential are preserved. We also prove
continuous dependence of solutions with respect to perturbations. To
carry out  the analysis a general abstract perturbation approach is developed,
which broadens the results known in the literature.

\end{abstract}

\section{Introduction}

In this paper we  consider parabolic Schrödinger type   evolution problems of the form
\begin{equation}\label{eq:intro-linear-e-eq-with-potential}
  \begin{cases}
    u_{t} + A_0^\mu u = V(x)u , & t>0, \ x\in \R^N, \\
    u(0,x)=u_0(x),  & x\in \R^N,
  \end{cases}
\end{equation}
where  $0<\mu\leq 1$ gives a fractional power of a uniformly  elliptic $2m$ order
operator of the form
\begin{equation}
\label{eq:operator-A0}
    A_0 = \sum_{|\zeta|= 2m} a_\zeta D^\zeta
    \quad
\text{ with constant real coefficients } \ a_{\zeta} ,
\end{equation}
and  we want to consider potentials $V$ and initial data $u_{0}$ in
suitable Morrey spaces to be introduced below.
This includes the case $A_{0}= (-\Delta)^{m}$ and, in particular when
$m=1$, fractional Schrödinger equations.

In order to solve problems like (\ref{eq:intro-linear-e-eq-with-potential}) in any
given function space there are usually two different, although
related, strategies. One is to prove suitable
\emph{resolvent estimates} on the  elliptic  operator in the
equation, $A_0^\mu -V(x) I$, which allow to prove that
(\ref{eq:intro-linear-e-eq-with-potential}) defines a suitable
semigroup of solutions, see
e.g. \cite{1983_Pazy,1981_Henry,lunardi95:_analy}.
Another one, which we take here, is to exploit the known results for
the unperturbed problem, $V=0$ in
(\ref{eq:intro-linear-e-eq-with-potential}),  and then prove that the
\emph{perturbed} problem
(\ref{eq:intro-linear-e-eq-with-potential}) can be solved for some
class of initial data for which the unperturbed problem can be
solved. This is done by means of Duhamel's principle, or variations of
constants formula
\begin{displaymath}
    u(t) = S_{\mu}(t) u_{0} + \int_{0}^{t}  S_{\mu}(t-s) V u(s) \, ds,
    \quad  t>0,
  \end{displaymath}
  where $ S_{\mu}(t) u_{0}$ represents the solution of the unperturbed
  problem with initial data $u_{0}$. In this way properties of the
  unperturbed problems, e.g. spaces of admissible initial data,
  smoothing properties, exponential growth etc, can be obtained for
  the perturbed problem. Besides the abstract  approach in
\cite{1981_Henry} using the fractional power
  spaces associated to the elliptic operator, this approach has been used in \cite{R-B} for
  second order parabolic problems in Lebesgue, Bessel and uniform
  spaces, \cite{Q-RB} for fourth order problems in the same spaces and
  \cite{C-RB} for general $2m$ order parabolic problems in the same
  scales of spaces. Here we extend this approach to the scale of
  Morrey spaces and fractional operators. It is worth mentioning that
in   all the references  mentioned in the previous paragraph,  the family of spaces one
works in is a one real-parameter scale of spaces a situation that
strongly simplifies the analysis. This will not be the case here as we
 explain below and is one of the main sources of difficulties in our
 analysis.

Morrey spaces, to be described in detail in Section
\ref{sec:morrey_spaces},  are made up of functions, or measures, which
have some more precise mass distribution in space, compared to
functions in standard Lebesgue spaces, see
(\ref{eq:morrey_norm}). So, in a sense they are some sort of
intermediate spaces between $L^{p}(\R^{N})$ and $L^{\infty}(\R^{N})$.
Therefore subtle differences and heavy difficulties appear when
dealing with evolution problems of the type
(\ref{eq:intro-linear-e-eq-with-potential}) in them.

The homogeneous or unperturbed problem, that is $V=0$, has been
studied with initial data in Morrey spaces in several references, see
Section \ref{sec:homogeneous_equation_in_Morrey}. Several of these
results stem from the corresponding problem with initial data in
uniform spaces, which is a setting for
which previous results are also available,  see Section
\ref{sec:homogeneous_equation_in_Uniform}.

Using these results our goal is to solve
(\ref{eq:intro-linear-e-eq-with-potential}) when the potential $V$ is
also in a Morrey space, or is a sum of such potentials. For this we
use perturbation techniques, so we can use in an essential way
properties of the solutions of the unperturbed problem in Morrey
spaces. This technique also requires that the multiplication operator
defined by $V$, transforms some Morrey spaces into some others. This
is the reason to take $V$ in a Morrey spaces itself, see Section
\ref{sec:morrey-potential}.
Previous perturbations
results in \cite{C-RB-linear_morrey} used specific homogeneous
perturbations and the techniques in that reference can not be applied
to the  general potentials we consider here.

Now we describe in some detail the difficulties we face in our
approach. As will be seen in Section \ref{sec:morrey_spaces}, Morrey spaces
$M^{p,\ell}(\R^N)$ depend on two parameters  $0<\ell \leq N$ and  $1\leq
p \leq \infty$, so we have a two parameter scale of spaces. The
unperturbed problem, $V=0$,  defines a semigroup of solutions in this
scale that has suitable smoothing properties between only some of the
spaces of this scale, where both parameters must be chosen in a very
specific way, see (\ref{eq:estimates_Mpl-Mqs}). These estimates are
known to be optimal from \cite{C-RB-scaling1}. For the perturbed problem on the other
hand, if we have the potential in a Morrey
space,  $V\in M^{p_0,\ell_0}(\R^N)$, the corresponding multiplication
operator acts continuously only between quite specific pairs of Morrey
spaces, see (\ref{eq:action-of-PV}). Therefore, to solve
(\ref{eq:intro-linear-e-eq-with-potential}) using Duhamel's principle,
(or  variations of constants formula), which is the main perturbation
tool we use, requires putting all these properties together in a
nontrivial way.

For this, in Section \ref{sec:abstract-approach}, we develop  an
abstract perturbation theory for semigroups defined in general scales
of spaces without any specific assumption in the set of indexes that
label the spaces of the family. These results allow for several simultaneous perturbations
and describe the subset of the scale of spaces for the initial data for which the perturbed
problem can be solved and the spaces of the scale to which the
solutions regularise. These, in turn, determine the part of the scale
in which the perturbed problem defines a semigroup that behaves near
time $t=0$ as the original unperturbed semigroup. The results  also
discuss the exponential growth of the perturbed semigroup in terms of
the sizes of the perturbations, the continuous dependence of solutions with respect to
perturbations and  the analiticity of the
perturbed semigroup.

These results are applied in full detail to the scale of
Morrey spaces in Section \ref{sec:linear-eq-with-Morrey-potential(s)}
but this still requires a nontrivial analysis of this particular
case. Also, the results in Section \ref{sec:abstract-approach} can be
applied to other interesting situations, like two parameter scale of
Sobolev spaces. This will be pursued somewhere else.

For the case of a single perturbation, our main results in Section
\ref{sec:one-perturbation} state that given a
potential $V\in M^{p_0,\ell_0}(\R^N)$, the problem
(\ref{eq:intro-linear-e-eq-with-potential}) can be solved for initial
data $u_{0}$ in Morrey spaces $M^{p,\ell}(\R^N)$ with $1\leq p \leq
\infty$ and $\ell\leq \ell_{0}$ and defines a semigroup of solutions
that satisfy the same smoothing estimates than the unperturbed
semigroup, see Theorem \ref{thm:perturbation-by-a-potential}. That is,
the perturbation preserves part of the scale of spaces and the
smoothing estimates. Also the
perturbed semigroup depends continuously with respect to the
perturbations, see Theorem
\ref{thr:morrey_continuous_dependence_1_perturbation}.  The
corresponding results when the perturbation is the sum of two (or
more)  potentials in different Morrey spaces are stated in Section
\ref{sec:two-perturbations}. This situation adds additional
difficulties to the analysis.

In general in  this paper we denote by $c$ or $C$ generic constants that may
change from line to line, whose value is not important for the
results.

\section{Morrey spaces of functions and measures}
\label{sec:morrey_spaces}

A function $\phi\in
L^p_{loc}(\R^N)$ belongs to the Morrey space $M^{p,\ell}(\R^N)$, $\ell\in[0,N]$, $p\in[1,\infty)$
iff
\begin{equation} \label{eq:morrey_norm}
\| \phi\|_{M^{p,\ell}(\R^N)} = \sup_{x_0\in \R^N\! ,\
  \! R>0} R^{\frac{\ell-N}{p}} \|\phi\|_{L^p(B(x_0,R))}<\infty .
  \end{equation}

If  $\ell=N$ then
$M^{p,N}(\R^N) = L^{p}(\R^N)$ for $p\in[1,\infty)$ (taking $R\to \infty$), whereas if $\ell
=0$ then $M^{p,0}(\R^N)=L^\infty(\R^N)$ (taking $R\to 0$ and using Lebesgue's differentiation theorem). We also set
$M^{\infty,\ell}(\R^N):=L^\infty(\R^N)$, $\ell\in[0,N]$.

Morrey spaces can be characterized in terms of the locally
uniform Lebesgue's spaces $L^p_{U}(\R^N)$, $p\in[1,\infty]$,  which can be
traced back to \cite{K-1975} and
are composed of $\phi \in
L_{loc}^p(\R^N)$ such that
\begin{equation*}
\|\phi\|_{L_U^p(\R^N)}=\sup_{x_0\in\R^N}\|\phi\|_{L^p(B(x_0,1))}<\infty,
\end{equation*}
where $L_U^\infty(\R^{N})=L^\infty(\R^{N})$.
Using dilations defined for functions in $\R^{N}$ by
\begin{displaymath}
\phi_R(x) = \phi(Rx), \quad x\in \R^N, \ R>0
\end{displaymath}
we have that
\begin{displaymath}
\phi \in M^{p,\ell}(\R^N) \ \text{ if and only if }
\sup_{R>0} R^\frac{\ell}p  \|\phi_{R} \|_{L_{U}^p(\R^N)} <\infty
\end{displaymath}
and
$\|\phi \|_{M^{p,\ell}(\R^N)}= \sup_{R>0} R^\frac{\ell}p \| \phi_{R}\|_{L_{U}^p(\R^N)}$
(see \cite[Proposition 2.1]{C-RB-linear_morrey}). Given any $\ell\in[0,N]$ and $p\in[1,\infty)$ we have in particular continuous embedding
\begin{displaymath}
M^{p,\ell}(\R^N)\subset L^p_U(\R^N).
\end{displaymath}

The dotted Morrey spaces $\dot M^{p,\ell}(\R^N)$, $1< p<
\infty$, $\ell\in(0,N]$ denote subspaces of $M^{p,\ell}(\R^N)$ in which translations are continuous, that is
\begin{equation}\label{eq:continuity-of-translations}
\tau_y \phi -\phi \to 0 \ \text{ as } \ y\to 0
\end{equation}
in $M^{p,\ell}(\R^N)$,
where $\tau_y\phi(x)=\phi(x-y)$ for  $x\in\R$.
Given $\ell\in[0,N]$ and $p\in[1,\infty)$, $\dot M^{p,\ell}(\R^N)$ is in particular a subspace of $\dot L^p_U(\R^N)$ consisting of functions from  $L^p_U(\R^N)$ which satisfy (\ref{eq:continuity-of-translations}) in $L^p_U(\R^N)$.

Regarding spaces of Morrey measures (see \cite{1989_Giga_Miyakawa, K}), we consider for $\ell\in[0,N]$
the space $\mathcal{M}^{\ell}(\R^{N})$
which consists of Radon measures
$\mu$ satisfying
\begin{displaymath}
  \|\mu\|_{\mathcal{M}^{\ell}(\R^N)} = \sup_{x_{0} \in \R^N, R>0} \
  R^{\ell-N} \ |\mu|(B(x_{0},R))  <\infty.
\end{displaymath}
Given any $\ell\in (0,N]$,
\begin{displaymath}
M^{1,\ell}(\R^N) \subset
\mathcal{M}^{\ell} (\R^N) \ \text{ isometrically }
\end{displaymath}
where $M^{1,N}(\R^N) = L^{1}(\R^N)$, whereas
$\mathcal{M}^{N} (\R^N)=\mathcal{M}_{BTV} (\R^N)$ is the space of Radon
measures with bounded total variation.

All above mentioned spaces are in particular contained in the space of uniform measures $\mathcal{M}_U(\R^N)$, which consists of Radon measures
$\mu$ satisfying
\begin{displaymath}
  \|\mu\|_{\mathcal{M}_{U}(\R^N)} = \sup_{x_{0} \in \R^N} \
  |\mu|(B(x_{0},1)) < \infty.
\end{displaymath}

\section{The homogeneous linear equation in uniform spaces}
\label{sec:homogeneous_equation_in_Uniform}

In this section, given $A_0$ as in (\ref{eq:operator-A0}),  we consider
the  linear \emph{fractional diffusion} problem
\begin{equation}\label{eq:linear-e-eq-Uniform}
  \begin{cases}
    u_{t} + A_0^{\mu} u =0 ,  & x \in \R^{N}, \quad t>0, \\
    u(0,x)=u_0(x),  & x\in \R^N
  \end{cases}
\end{equation}
where $0<\mu \leq 1$ and $A_0^{\mu}$ is the fractional power of
$A_{0}$. We collect below several known results for
(\ref{eq:linear-e-eq-Uniform}) when the initial data is taken
in  locally uniform spaces. This strategy allows to obtain similar results in
Morrey spaces, see Section \ref{sec:homogeneous_equation_in_Morrey}
and \cite{C-RB-linear_morrey}.

\begin{proposition}
  \label{prop:1st-about-{eq:linear-e-eq-Uniform}}
Let $ 1\leq p \leq \infty$, $0< \mu \leq 1$ and
assume that $A_0$ is as in (\ref{eq:operator-A0}).

\begin{enumerate}
\item
  Then (\ref{eq:linear-e-eq-Uniform}) defines a semigroup of solutions
  $\{S_{\mu}(t)\}_{t\geq0}$ in each one of the spaces
$L^p_U(\R^N)$ and   $\mathcal{M}_U(\R^N)$.

  \item
    The semigroup is analytic and has a sectorial
  generator provided $0<\mu<1$, or $\mu=1$ and  $1< p \leq \infty$.

  \item
  The semigroup has a selfsimilar  kernel
  $k_{\mu}(t,x,y) = \frac{1}{t^\frac{N}{2m\mu}}
  K_{\mu}\left(\frac{x-y}{t^{\frac{1}{2m\mu}}}\right)$, that is,
  \begin{displaymath}
    S_{\mu}(t) u_{0}(x) = \int_{\R^{N}} k_{\mu}(t,x,y) u_{0}(y) \, dy ,
    \quad x\in \R^{N}, \ t>0.
  \end{displaymath}
  Moreover if the semigroup generated by $-A_{0}$, that is,
  $\{S_{1}(t)\}_{t\geq0}$ is order preserving (e.g. if
  $A_{0}= -\Delta$), so is $\{S_{\mu}(t)\}_{t\geq0}$ for $0<\mu <1$ and actually
  $k_{\mu}(t,x,y) \geq 0$ for all $t,x,y$.

\item
  The space $\dot   L^{p}_{U}(\R^N)$ is invariant for the semigroup.

\end{enumerate}
\end{proposition}
\begin{proof}
  We start with (i), (ii) and the case $\mu=1$. The results for $L^{p}_{U}(\R^{N})$
  can be found in \cite{C-RB}. The results for $\mathcal{M}_{U}(\R^{N})$
  and the properties of the kernel in (iii) can be found in
  \cite{C-RB-scaling3-4}.

  Then Proposition B.1 in \cite{C-RB-linear_morrey} gives the analyticity results for $0<\mu <1$ in all the spaces above.

Since kernel is selfsimilar, translations $\tau_z$ given in the line
below (\ref{eq:continuity-of-translations}) commute with $S_\mu(t)$,
that is, $\tau_z S_\mu(t)u_0=S_\mu(t)\tau_z u_0$ for $z\in\R^N$,
$t>0$, $u_0\in L^{p}_U(\R^N)$ and $\mu\in(0,1]$. Hence $\|\tau_z
S_\mu(t)u_0 - S_\mu(t)u_0\|_{L^{p}_U(\R^N)} \leq
\|S_\mu(t)\|_{\mathcal{L}(L^{p}_U(\R^N))} \|\tau_z u_0 -
u_0\|_{L^{p}_U(\R^N)}$ and the invariance of $\dot L^{p}_U(\R^N)$ in
(iv) follows.
\end{proof}

The next result collects several estimates for the semigroup above
between the uniform spaces.

\begin{proposition}
Let $\{S_{\mu}(t)\}_{t\geq0}$, $\mu\in (0,1]$, be as in Proposition \ref{prop:1st-about-{eq:linear-e-eq-Uniform}}.

    Given $1\leq p\leq q \leq \infty$ we have for some constant
$c=c_{\mu,p,q}$ that
\begin{equation}\label{eq:estimates_LpU-LqU}
  \|S_{\mu}(t)\|_{{\mathcal L}(L^{p}_U(\R^N), L^{q}_U (\R^N))} \leq
  c(1+\frac{1}{
    t^{\frac{1}{2m\mu}(\frac{N}{p}-\frac{N}q)}}), \quad t>0,
\end{equation}
which remains true if for $p=1$ we replace $L^{1}_U(\R^N)$ by
$\mathcal{M}_U(\R^N)$.
\end{proposition}
\begin{proof}
With $\mu=1$ this follows from the estimates in \cite[Theorem
3.1]{C-RB-scaling3-4} (since from  \cite[Theorem 6.1]{C-RB-scaling3-4}
we can apply that theorem with the constant  $a=0$).

With $\mu\in(0,1)$ this follows from the estimate for
$\mu=1$ and \cite[Lemma 4.4]{C-RB-scaling1}. This is straightforward
if $\frac{1}{2m}(\frac{N}{p}-\frac{N}{q})<1$. If
$\frac{1}{2m}(\frac{N}{p}-\frac{N}{q})\geq1$, we choose a finite
number of points $q_{j}$, $j=0,\ldots, J$ such that $q_{0}=p <
q_{1}<\ldots < q_{J}=q$ and
$\frac{1}{2m}(\frac{N}{q_{j}}-\frac{N}{q_{j+1}})<1$ to get via
\cite[Lemma 4.4]{C-RB-scaling1} that
$\|S_{\mu}(t)\|_{{\mathcal L}(L^{q_j}_U(\R^N), L^{q_{j+1}}_U (\R^N))}
\leq
c (1+\frac{1}{t^{\frac{1}{2m\mu}(\frac{N}{q_j}-\frac{N}{q_{j+1}})}})$
for $j=1,\ldots, J-1$. From these estimates we get the result using
the semigroup property and Young's inequality.
\end{proof}

The next result states the time continuity properties of the
trajectories of the semigroup.

\begin{proposition}
  Let $\{S_{\mu}(t)\}_{t\geq0}$, $\mu\in (0,1]$, be as in Proposition \ref{prop:1st-about-{eq:linear-e-eq-Uniform}}.

Then for  $1\leq p,q \leq \infty$ we have
\begin{displaymath}
  (0,\infty)\times L^{p}_U(\R^N) \ni (t,u_0) \to S_{\mu}(t) u_0 \in L^{q}_U(\R^{N})
  \ \text{ is continuous},
\end{displaymath}
which remains true if for $p=1$ we replace $L^{1}_U(\R^N)$ by
$\mathcal{M}_U(\R^N)$.
\end{proposition}
\begin{proof}
If $(0,\infty)\times L^{p}_U(\R^N)\ni
  (t_n,u_{0n}) \to (t_0,u_0)\in  (0,\infty)\times L^{p}_U(\R^N)$ as
  $n\to\infty$, then for any small enough $\eps>0$ we can write
\begin{displaymath}
S_{\mu}(t_n)u_{0n}- S_{\mu}(t_0)u_{0}= S_{\mu}(t_n)u_{0n} -
S_{\mu}(t_n)u_{0} + S_{\mu}(t_n-\eps)S(\eps)u_{0} -
S_{\mu}(t_0-\eps)S_{\mu}(\eps)u_{0}.
\end{displaymath}
Then for  $p\leq q\leq \infty$ we have from
(\ref{eq:estimates_LpU-LqU}) \begin{displaymath}
\|S_{\mu}(t_n)u_{0n} - S_{\mu}(t_n)u_{0}\|_{L^{q}_U(\R^N)} \leq
c\left(1+\frac{1}{
  t_n^{\frac{1}{2m\mu}(\frac{N}{p}-\frac{N}q)}} \right) \|u_{0n} -u_{0}\|_{
  L^{p}_U(\R^N)} \to 0 \quad \text{ as } n\to \infty.
\end{displaymath}
Also, since
$S_{\mu}(\eps)u_{0}\in L^{q}_U(\R^N)$ by (\ref{eq:estimates_LpU-LqU}), we
see that
\begin{displaymath}
\|S_{\mu}(t_n-\eps) S_{\mu}(\eps)u_{0} -
S_{\mu}(t_0-\eps)S_{\mu}(\eps)u_{0}\|_{L^{q}_U(\R^N)} \to 0 \quad
\text{ as } n\to \infty,
\end{displaymath}
since by  Proposition
\ref{prop:1st-about-{eq:linear-e-eq}}  the semigroup
$\{S_{\mu}(t)\}_{t\geq0}$ is analytic (thus continuous at each positive
time) in
$L^{q}_U(\R^N)$ for $0<\mu<1$, or $\mu=1$ and $q \not= 1$. For
$\mu=1$ and $q=1$ it is continuous in $L^{q}_U(\R^N)$ for positive times
from \cite[Theorem 4.5]{C-RB}.

Finally, for  $1\leq q < p$,
since we have proven  continuity in $L^p_U(\R^N)$, so we have it in
$L^q_U(\R^N)$ as  $L^p_U(\R^N)\subset L^q_U(\R^N)$.
\end{proof}

About the way the semigroup approaches the initial data, we have the
following result.

\begin{proposition}\label{prop:convergence-as-t-to0}
Let $\{S_{\mu}(t)\}_{t\geq0}$, $\mu\in (0,1]$, be as in Proposition \ref{prop:1st-about-{eq:linear-e-eq}}.

Then for any $u_0\in L^p_U(\R^N)$ with $1\leq p < \infty$ we have
\begin{equation}\label{eq:behavior-at-0-of-S{mu}(t)-in-Prop-2.4}
  S_\mu(t)u_0 \to u_0 \ \text{ as } \ t\to 0^+, \quad \mbox{in
    $L^p_{loc}(\R^N)$},
\end{equation}

The convergence in (\ref{eq:behavior-at-0-of-S{mu}(t)-in-Prop-2.4}) is in $L^p_U(\R^N)$
if $u_0\in \dot L^p_U(\R^N)$.

\end{proposition}
\begin{proof}
  If $\mu=1$ then (\ref{eq:behavior-at-0-of-S{mu}(t)-in-Prop-2.4}) in
  $L^p_{loc}(\R^N)$ is from \cite[Theorems 4.1, 4.5 and Proposition
  4.9]{C-RB}, (\ref{eq:behavior-at-0-of-S{mu}(t)-in-Prop-2.4}) in $L^p_U(\R^N)$
  with $u_0\in \dot L^p_U(\R^N)$ is from \cite[(4.1), (4.5) and
  (4.6)]{C-RB}.

Now, given $\mu\in(0,1)$ observe that (\ref{eq:behavior-at-0-of-S{mu}(t)-in-Prop-2.4}) in $L^p_{loc}(\R^N)$ follows from the convergence properties in the case $\mu=1$ using the expressions in \cite[(20'), p. 264 and (14), p. 262]{1978_Yosida} (see also \cite[Appendix B]{C-RB-linear_morrey}). Indeed, given $u_0\in L^p_U(\R^N)$ and any ball $B\subset \R^N$, with $f_{1,\mu}$ as in \cite[p. 264]{1978_Yosida} we have that
\begin{displaymath}
\|S_\mu(t)u_0 - u_0 \|_{L^p(B)} \leq \int_0^\infty f_{1,\mu} (s) \|S_1(st^\frac{1}{\mu})u_0 - u_0 \|_{L^p(B)} \to 0 \ \text{ as } \ t\to0^+,
\end{displaymath}
where the convergence on the right is due to Lebesgue's dominated
convergence theorem.
\end{proof}

\section{The homogeneous linear equation in Morrey spaces}
\label{sec:homogeneous_equation_in_Morrey}

In this section, given $A_0$ as in (\ref{eq:operator-A0}),  we consider
the  linear \emph{fractional diffusion} problem
\begin{equation}\label{eq:linear-e-eq}
  \begin{cases}
    u_{t} + A_0^{\mu} u =0 ,  & x \in \R^{N}, \quad t>0, \\
    u(0,x)=u_0(x), & x\in \R^N
  \end{cases}
\end{equation}
where $0<\mu \leq 1$ and $A_0^{\mu}$ is the fractional power of
$A_{0}$. We collect below several previous results for
(\ref{eq:linear-e-eq}) when the initial data are taken
in Morrey spaces.

\begin{proposition}
  \label{prop:1st-about-{eq:linear-e-eq}}
Let $ 1\leq p \leq \infty$,  $0< \ell\leq N$,  $0< \mu \leq 1$ and
assume that $A_0$ is as in (\ref{eq:operator-A0}).

\begin{enumerate}
\item
  Then (\ref{eq:linear-e-eq}) defines a semigroup of solutions
  $\{S_{\mu}(t)\}_{t\geq0}$ in each one of the spaces
  $M^{p,\ell}(\R^N)$, $\mathcal{M}^{\ell}(\R^N)$

  \item
    The semigroup is analytic and has a sectorial
  generator provided $0<\mu<1$, or $\mu=1$ and $1< p \leq \infty$.

  \item
     The semigroup has a selfsimilar kernel
  $k_{\mu}(t,x,y) = \frac{1}{t^\frac{N}{2m\mu}}
  K_{\mu}\left(\frac{x-y}{t^{\frac{1}{2m\mu}}}\right)$, that is,
  \begin{displaymath}
    S_{\mu}(t) u_{0}(x) = \int_{\R^{N}} k_{\mu}(t,x,y) u_{0}(y) \, dy ,
    \quad x\in \R^{N}, \ t>0.
  \end{displaymath}
  Moreover if the semigroup generated by $-A_{0}$, that is,
  $\{S_{1}(t)\}_{t\geq0}$ is order preserving (e.g. if
  $A_{0}= -\Delta$), so is $\{S_{\mu}(t)\}_{t\geq0}$ for $0<\mu <1$ and actually
  $k_{\mu}(t,x,y) \geq 0$ for all $t,x,y$.

\item
  The space $\dot M^{p,\ell}(\R^N)$ is invariant for the semigroup.

\end{enumerate}
\end{proposition}
\begin{proof}
  The results for the semigroup in $M^{p,\ell}(\R^N)$,
  $\mathcal{M}^{\ell}(\R^N)$,  can be found in
  \cite{C-RB-linear_morrey}, whereas for the kernel in
  \cite{C-RB-scaling3-4}. As in the proof of Proposition \ref{prop:1st-about-{eq:linear-e-eq-Uniform}}, translations $\tau_z$ commute with $S_\mu(t)$. Hence $\|\tau_z S_\mu(t)u_0 - S_\mu(t)u_0\|_{M^{p,\ell}(\R^N)} \leq \|S_\mu(t)\|_{\mathcal{L}(M^{p,\ell}(\R^N))} \|\tau_z u_0 - u_0\|_{M^{p,\ell}(\R^N)}$ and the invariance of $\dot M^{p,\ell}(\R^N)$ follows.
\end{proof}

\begin{remark}
  \label{rem:related-to-{prop:1st-about-{eq:linear-e-eq}}}
  For $0<\mu\leq1$, $1< p \leq \infty$ and $0<\ell\leq N$ the sectorial generator
  of the semigroup $\{S_{\mu}(t)\}_{t\geq0}$
  in $M^{p,\ell}(\R^N)$ in Proposition \ref{prop:1st-about-{eq:linear-e-eq}} is $-A_0^\mu$, which follows from \cite[Proposition B.1, p. 1604]{C-RB-linear_morrey}.
\end{remark}

The next result collects several estimates for the semigroup above
between the spaces considered before.

\begin{proposition}
\label{prop:2nd-about-{eq:linear-e-eq}}
Let $\{S_{\mu}(t)\}_{t\geq0}$, $\mu\in (0,1]$, be as in Proposition \ref{prop:1st-about-{eq:linear-e-eq}}.

  Given $1\leq p\leq \infty $ and $0<\ell\leq N$, for
  $1\leq q \leq \infty$ and $0\leq s\leq \ell\leq N$ satisfying
  $\frac{s}q\leq \frac{\ell}p$,  we have for some constant
  $c=c_{\mu,p,\ell,q,s}$ that
  \begin{equation} \label{eq:estimates_Mpl-Mqs}
    \|S_{\mu}(t)\|_{{\mathcal L}(M^{p,\ell}(\R^N), M^{q,s} (\R^N))} =
    \frac{c}{ t^{\frac{1}{2m\mu}(\frac\ell{p}-\frac{s}q)}}, \quad t>0,
  \end{equation}
  which remains true if for $p=1$ we replace $M^{1,\ell}(\R^N)$ by
  $\mathcal{M}^{\ell}(\R^N)$.

\end{proposition}
\begin{proof}
This is  from \cite[Theorems 1.4 and 1.5]{C-RB-linear_morrey}.
\end{proof}

The next result states the time continuity properties of the
trajectories of the semigroup.

\begin{proposition}
  \label{prop:about-time-continuity}

 Let $\{S_{\mu}(t)\}_{t\geq0}$, $\mu\in (0,1]$, be as in Proposition \ref{prop:1st-about-{eq:linear-e-eq}}.

Then for $1\leq p,q \leq
\infty$ and $0\leq s\leq \ell\leq N$ satisfying $\frac{s}q\leq
\frac{\ell}p$ we have
\begin{equation}\label{eq:suitable-continuity-in-Morrey}
(0,\infty)\times M^{p,\ell}(\R^N) \ni (t,u_0) \to S_{\mu}(t) u_0 \in M^{q,s}(\R^{N})
\ \text{ is continuous},
 \end{equation}
which remains true if for $p=1$ we replace $M^{1,\ell}(\R^N)$ by $\mathcal{M}^{\ell}(\R^N)$.

\end{proposition}
\begin{proof}
We argue below in three cases.

\medskip

\noindent
\fbox{{\bf Case A}: either $\mu\not=1$, or $\mu=1$ and $q\not=1$.} In this case we first remark that if $(0,\infty)\times M^{p,\ell}(\R^N)\ni
  (t_n,u_{0n}) \to (t_0,u_0)\in  (0,\infty)\times M^{p,\ell}(\R^N)$ as
  $n\to\infty$, then for any small enough $\eps>0$ we can write
\begin{displaymath}
S_{\mu}(t_n)u_{0n}- S_{\mu}(t_0)u_{0}= S_{\mu}(t_n)u_{0n} - S_{\mu}(t_n)u_{0} + S_{\mu}(t_n-\eps)S(\eps)u_{0} - S_{\mu}(t_0-\eps)S_{\mu}(\eps)u_{0}.
\end{displaymath}
Due to Proposition \ref{prop:2nd-about-{eq:linear-e-eq}} we have
\begin{displaymath}
\|S_{\mu}(t_n)u_{0n} - S_{\mu}(t_n)u_{0}\|_{M^{q,s}(\R^N)} \leq \frac{c}{
  t_n^{\frac{1}{2m\mu}(\frac\ell{p}-\frac{s}q)}} \|u_{0n} -u_{0}\|_{ M^{p,\ell}(\R^N)} \to 0 \quad \text{ as } n\to \infty.
\end{displaymath}
Also
\begin{displaymath}
\|S_{\mu}(t_n-\eps) S_{\mu}(\eps)u_{0} -
S_{\mu}(t_0-\eps)S_{\mu}(\eps)u_{0}\|_{M^{q,s}(\R^N)} \to 0 \quad
\text{ as } n\to \infty,
\end{displaymath}
because Proposition \ref{prop:2nd-about-{eq:linear-e-eq}} yields
$S_{\mu}(\eps)u_{0}\in M^{q,s}(\R^N)$ and by Proposition
\ref{prop:1st-about-{eq:linear-e-eq}} the semigroup
$\{S_{\mu}(t)\}_{t\geq0}$ is analytic (thus, in particular, continuous for positive times) in $M^{q,s}(\R^N)$.

\medskip

\noindent
\fbox{{\bf Case B}: $\mu=1$ and $q=1$ and $p\not=1$.} Given $p\not=1$, $0<s\leq \ell\leq N$, $u_0\in M^{p,\ell}(\R^N)$ and $t>0$ we have from
\cite[p. 1587, Theorem 5.1]{C-RB-linear_morrey} that
\begin{displaymath}
\nabla S_1(t) u_0\in M^{1,s}(\R^N),
\end{displaymath}
whereas from
Proposition \ref{prop:2nd-about-{eq:linear-e-eq}}
\begin{displaymath}
S_1(t) u_0\in L^{\infty}(\R^N)\cap M^{1,s}(\R^N),
\end{displaymath}
so via
\cite[p.1571, Proposition 2.2]{C-RB-linear_morrey} we see that
\begin{displaymath}
S_1(t) u_0\in \dot M^{1,s}(\R^N).
\end{displaymath}
This and \cite[formula (1.8), p. 1563]{C-RB-linear_morrey} yield
\begin{equation}\label{eq:continuity-from-the-Right}
\lim_{h\searrow 0} \| S_1(h) S_1(t) u_0 - S_1(t) u_0\|_{M^{1,s}(\R^N)} =0 \ \text{ for each } \ t>0.
\end{equation}

Now for $t>0$ and $-\frac{t}{2}<h<0$, since by Proposition \ref{prop:2nd-about-{eq:linear-e-eq}}
$
\sup_{h\in(-\frac{t}{2},0)}\| S_1\big(\frac{t}{2}+h\big)\|_{\mathcal{L}(M^{1,s}(\R^N))}=c
$,
\begin{equation}\label{eq:to-get-continuity-from-the-Left}
\begin{split}
\| S_1(t+h) u_0 - S_1(t) u_0\|_{M^{1,s}(\R^N)}&
=
\| S_1\left(\frac{t}{2}+h\right)\left(S_1\left(\frac{t}{2}\right) u_0 - S_1(-h)S_1\left(\frac{t}{2}\right) u_0\right)\|_{M^{1,s}(\R^N)}
\\&
\leq  c \|S_1 \left(\frac{t}{2} \right)u_0 - S_1(-h) S_1\left(\frac{t}{2}\right) u_0\|_{M^{1,s}(\R^N)}
\end{split}
\end{equation}
and due to (\ref{eq:continuity-from-the-Right}) the right hand of (\ref{eq:to-get-continuity-from-the-Left}) tends to zero as $h\nearrow 0$. As a consequence,
\begin{displaymath}
\lim_{h\nearrow 0} \| S_1(t+h) u_0 - S_1(t) u_0\|_{M^{1,s}(\R^N)} =0 \ \text{ for each } \ t>0.
\end{displaymath}
Given $u_0\in M^{p,\ell}(\R^N)$ we thus see that
$(0,\infty)\ni t \to S_1(t) u_0 \in M^{1,s}(\R^N)$ is continuous.

This and the estimate
$\| S_1(t) \|_{\mathcal{L}(M^{p,\ell}(\R^N),M^{1,s}(\R^N))}=\frac{c}{t^{\frac{1}{2m}(\frac{\ell}{p}-s)}}$
from Proposition \ref{prop:2nd-about-{eq:linear-e-eq}} yield (\ref{eq:suitable-continuity-in-Morrey}) in the considered case after we use a similar argument as in Case A above.

\medskip

\noindent
\fbox{{\bf Case C}: $\mu=1$ and $q=1$ and $p=1$ and $s<\ell$.} In this case for all sufficiently small $\varepsilon>0$ we have $s<\frac{\ell}{1+\varepsilon} < \ell$ and given $u_0\in M^{1,\ell}(\R^N)$ we observe from Proposition \ref{prop:2nd-about-{eq:linear-e-eq}} that
\begin{displaymath}
\text{
$S_1(t)u_0 \in M^{1+\varepsilon,\ell}(\R^N)$
and $S_1(\tau)S_1(t)u_0 \in M^{1,s}(\R^N)$ whenever $t,\tau>0$.}
\end{displaymath}

Now, if $t>0$ and $t_n \to t$ then choosing small enough $\varepsilon>0$ we have
\begin{equation}\label{eq:to-get-continuity-in-Case-C}
\|S_1(t_n)u_0 -  S_1(t) u_0 \|_{M^{1,s}(\R^N)}
=\|S_1(t_n-\varepsilon)S_1(\varepsilon)  u_0 -  S_1(t-\varepsilon) S_1(\varepsilon) u_0 \|_{M^{1,s}(\R^N)}
\end{equation}
where
\begin{displaymath}
S_1(\varepsilon)  u_0\in M^{1+\varepsilon,\ell}(\R^N) \ \text{ and } s<\frac{\ell}{1+\varepsilon}<\ell,
\end{displaymath}
so due to the continuity proved in Case B above
the right hand side of (\ref{eq:to-get-continuity-in-Case-C}) tends to zero as $n\to\infty$.

Given $u_0\in M^{1,\ell}(\R^N)$ we thus see that
$(0,\infty)\ni t \to S_1(t) u_0 \in M^{1,s}(\R^N)$ is continuous. This together with the estimate
$\| S_1(t) \|_{\mathcal{L}(M^{1,\ell}(\R^N),M^{1,s}(\R^N))}=\frac{c}{t^{\frac{1}{2m}(\ell-s)}}$ from Proposition \ref{prop:2nd-about-{eq:linear-e-eq}} yield (\ref{eq:suitable-continuity-in-Morrey}) in the considered case with a similar argument as in Case A above.

\medskip

\noindent
\fbox{{\bf Case D}: $\mu=1$ and $q=1$ and $p=1$ and $s=\ell$.} Given $u_0\in M^{1,\ell}(\R^N)$ and a ball $B(x_{0},R)\subset \R^{N}$, we denote by $\chi_{B(x_{0}, R)}$ the characteristic function of $B(x_{0},R)$ and observe that for $t>0$, $t+h>0$ we get
\begin{equation}\label{eq:first-for-in-M1ell}
  \begin{split}
  R^{\ell-N}
&\|S(t+h)  u_{0} - S(t)  u_{0}\|_{L^1(B(x_{0},R))}
\\&
  \leq R^{\ell-N}\int_{\R^N}
  \left(
  \chi_{B(x_{0}, R)} (x) \int_{\R^{N}} \Big|  \frac{K_{1}\Big(\frac{z}{(t+h)^{\frac{1}{2m}}}\Big)}{(t+h)^\frac{N}{2m}} - \frac{K_{1}\Big(\frac{z}{t^{\frac{1}{2m}}}\Big)}{t^\frac{N}{2m}} \Big|   |  u_{0}(x-z)|   \, dz
  \right)
  \, dx
\\ &
=
R^{\ell-N}
\int_{\R^N}
\left(
\Big|  \frac{K_{1}\Big(\frac{z}{(t+h)^{\frac{1}{2m}}}\Big)}{(t+h)^\frac{N}{2m}} - \frac{K_{1}\Big(\frac{z}{t^{\frac{1}{2m}}}\Big)}{t^\frac{N}{2m}} \Big|   \int_{\R^{N}}    |  u_{0}(x-z)| \ \chi_{B(x_{0}, R)} (x)  \, dx
\right)
\, dz
\\
&
=
\int_{\R^N}
\Big|  \frac{K_{1}\Big(\frac{z}{(t+h)^{\frac{1}{2m}}}\Big)}{(t+h)^\frac{N}{2m}} - \frac{K_{1}\Big(\frac{z}{t^{\frac{1}{2m}}}\Big)}{t^\frac{N}{2m}} \Big|
\, R^{\ell-N} \|u_{0}\|_{L^1(_{B(x_{0}-z, R)})}
\, dz
\\
&
\leq
\|  u_{0}\|_{M^{1,\ell}(\R^N)} \int_{\R^N}
\Big|  \frac{K_{1}\Big(\frac{z}{(t+h)^{\frac{1}{2m}}}\Big)}{(t+h)^\frac{N}{2m}} - \frac{K_{1}\Big(\frac{z}{t^{\frac{1}{2m}}}\Big)}{t^\frac{N}{2m}} \Big|
\, dz
 .
\end{split}
\end{equation}
Taking $\delta \in \mathcal{M}^{N} (\R^N)=\mathcal{M}_{BTV} (\R^N)$ we have as in \cite[Proposition 6.1(i)]{C-RB-scaling3-4} that $K_1=S_1(1)\delta$. Since due to \cite[Proposition 3.2]{C-RB-scaling3-4} $S_1(t)\delta$ immediately enters $L^p(\R^N)$ for any $p\geq1$, we see
from \cite[Proposition 2.3 and Remark 2.6]{C-RB} that $S_1(t)\delta$ also enters immediately $H^{2m}_q(\R^N)$ for arbitrarily large $q$. Thus, via Sobolev embedding, $K_1$ is in particular a bounded uniformly continuous function in $\R^N$, which in turn implies that
\begin{displaymath}
  \Big|
  \frac{K_{1}\Big(\frac{z}{(t+h)^{\frac{1}{2m}}}\Big)}{(t+h)^\frac{N}{2m}}
   -
   \frac{K_{1}\Big(\frac{z}{t^{\frac{1}{2m}}}\Big)}{t^\frac{N}{2m}}
   \Big|
   \to 0 \ \text{ as } \ h\to 0 \ \text{ for each } \ t>0, \ z\in\R^N.
\end{displaymath}
Using pointwise Gaussian bound
(see \cite[formula (2.3) in Theorem 2.2]{C-RB-scaling3-4}) we also have
\begin{displaymath}
\Big|  K_{1}(\cdot) \Big| \leq \exp\big( {-c |\cdot|^\frac{2m}{2m-1}} \big)
\end{displaymath}
for some positive constant $c$.
Hence due to Lebesgue's dominated convergence theorem
\begin{equation}\label{eq:second-for-in-M1ell}
\int_{\R^N} \Big|  \frac{K_{1}\Big(\frac{z}{(t+h)^{\frac{1}{2m}}}\Big)}{(t+h)^\frac{N}{2m}} - \frac{K_{1}\Big(\frac{z}{t^{\frac{1}{2m}}}\Big)}{t^\frac{N}{2m}} \Big|
\, dz \to 0 \ \text{ as } \ h\to 0.
\end{equation}
From (\ref{eq:first-for-in-M1ell}) and (\ref{eq:second-for-in-M1ell}) we conclude that
$\|S(t+h)  u_{0} - S(t)  u_{0}\|_{M^{1,\ell}(\R^N)}\to 0$ as $h\to0$.

This and the estimate
$\|S_{\mu}(t)\|_{{\mathcal L}(M^{1,\ell}(\R^N))} = c$
from Proposition \ref{prop:2nd-about-{eq:linear-e-eq}}
yield (\ref{eq:suitable-continuity-in-Morrey}) in the considered case with a similar argument as in Case A above.
\end{proof}

About the way the semigroup approaches the initial data, we have the
following result.

\begin{proposition}
  \label{prop:initial_data}

Let $\{S_{\mu}(t)\}_{t\geq0}$, $\mu\in (0,1]$, be as in Proposition \ref{prop:1st-about-{eq:linear-e-eq}}.

Then for any $u_0\in M^{p,\ell}(\R^N)$ with $ 1\leq p \leq \infty$ and $0< \ell\leq N$ we have
\begin{equation}\label{eq:behavior-at-0-of-S{mu}(t)}
  S_\mu(t)u_0 \to u_0 \ \text{ as } \ t\to 0^+, \quad \mbox{in
    $L^p_{loc}(\R^N)$}.
\end{equation}

The convergence in (\ref{eq:behavior-at-0-of-S{mu}(t)}) is in $M^{p,\ell}(\R^N)$ if $u_0\in \dot M^{p,\ell}(\R^N)$.

\end{proposition}
\begin{proof}
  If $\mu=1$ then the result is from
  \cite[Theorem 1.1]{C-RB-linear_morrey}, whereas for $\mu\in(0,1)$ it follows analogously as in the proof of Proposition \ref{prop:convergence-as-t-to0}.
\end{proof}

\section{Morrey potentials}
\label{sec:morrey-potential}

As mentioned in the Introduction, our goal is now to perturb the
fractional diffusion equation (\ref{eq:linear-e-eq}) with some
potential terms. Previous results in this direction can be found in
\cite[Section 7]{C-RB-linear_morrey} for specific type of homogeneous
potentials of the form  $\frac{c}{|x|^{\beta}}$ and suitable
$\beta>0$. In that reference the results rely on  suitable
Gagliardo-Nirenberg-Hardy type inequalities that can not be applied
here for more general potentials.

Here our goal is to include a general potentials in Morrey spaces,
using different techniques. Of course, our results here apply to these
type of potentials as well.
Hence, we will assume that
\begin{equation}
  \label{eq:V-in-Morrey}
 V\in M^{p_0,\ell_0}(\R^N) \quad \text{ for  $1\leq p_0\leq \infty$,
   $\ell_0\in(0,N]$. }
\end{equation}

Then we consider the  multiplication operator $P_V$,  defined for functions $\phi$ in $\R^N$ by
\begin{equation}\label{eq:operator-V}
P_V \phi (x) = V(x) \phi(x), \quad x\in \R^N.
\end{equation}

The following result states how the multiplication operator
(\ref{eq:operator-V}) acts between Morrey spaces.

\begin{lemma}\label{lem:crucial-inequality}
  Assume $ V\in M^{p_0,\ell_0}(\R^N)$  for some $1\leq p_0\leq
  \infty$, $\ell_0\in(0,N]$ and let $p_0'$ be H\"older's conjugate of
  $p_0$.

Given any
$w\in[p_0',\infty]$
and $0< \kappa \leq N$ we have that if
$z$ and $\nu$ satisfy
\begin{displaymath}
\frac{1}{z}=\frac{1}{w}+\frac{1}{p_0}, \quad \frac{\nu}{z}=\frac{\kappa}{w}+\frac{\ell_0}{p_0}
\end{displaymath}
then for any $\phi\in M^{w,\kappa}(\R^N)$ and $P_V\phi$ as in
(\ref{eq:operator-V}), we have $P_V\phi\in M^{z,\nu}(\R^N) $  and
\begin{displaymath}
  \| P_V\phi\|_{M^{z,\nu}(\R^N)} \leq \| \phi\|_{M^{w,\kappa}(\R^N)}
  \| V\|_{M^{p_0,\ell_0}(\R^N)}.
\end{displaymath}
In particular,
\begin{equation}
  \label{eq:action-of-PV}
  P_V \in \mathcal{L}( M^{w,\kappa}(\R^{N}), M^{z,\nu}(\R^N)) \quad \text{ and } \quad
\|P_V \|_{\mathcal{L}( M^{w,\kappa}(\R^{N}), M^{z,\nu}(\R^N))}\leq \| V\|_{M^{p_0,\ell_0}(\R^N)}.
\end{equation}

\end{lemma}

\begin{proof}
  The result follows applying the following consequence of Hölder's
  inequality
  \begin{displaymath}
\| fg \|_{M^{z,\nu}(\R^N)} \leq \| f \|_{M^{w,\kappa}(\R^N)} \| g
\|_{M^{p_{0},\ell_{0}}(\R^N)} ,
\end{displaymath}
 see \cite[formula (2.5)]{C-RB-linear_morrey}.
\end{proof}

\section{Linear perturbations in the scale: an abstract approach}
\label{sec:abstract-approach}

As a consequence of the results in Sections
\ref{sec:homogeneous_equation_in_Morrey} and
\ref{sec:morrey-potential},  we have a semigroup
$\{S_{\mu}(t)\}_{t\geq0}$, $0< \mu\leq 1$, in the scale of spaces
$\{M^{p,\ell}(\R^{N})\}_{p,\ell}$, $1\leq p\leq \infty$, $0<\ell \leq
N$ (where, for $p=1$,  we can even replace  $M^{1,\ell}(\R^N)$ by
$\mathcal{M}^{\ell}(\R^N)$). This semigroup has continuous curves as in
Proposition \ref{prop:about-time-continuity}, attains the initial
data as in Proposition \ref{prop:initial_data} and  acts within these
spaces as in (\ref{eq:estimates_Mpl-Mqs}). On the other hand,  we have a
potential that acts within these spaces as in
(\ref{eq:action-of-PV}). Notice the corresponding restrictions on the
indexes of the spaces of the scale in these two latter  equations.

We can accomodate this situation in an abstract setting that will
allow further applications to other situations. This will be done
elsewhere.

First, assume we have a family of Banach spaces,
$\{X^\gamma\}_{\gamma\in \mathbb{J}}$ which
we call \emph{the scale}, where $\mathbb{J}$ is a certain set of
indexes. The norm  in $X^\gamma$ is denoted by
$\|\cdot\|_\gamma$. The spaces of the scale are assumed to be
topologically consistent, that is, if $\{u_{0n}\}\subset
X^\gamma\cap X^{\tilde\gamma}$  and $\{u_{0n}\}$ converges both in
$X^\gamma$ and in $X^{\tilde\gamma}$ then the limit is the same.

Each space in the scale, has an associated \emph{regularity index}
given by a mapping $\mathtt{r}:\mathbb{J} \to \R$.

As several of the results below do not depend on the semigroup
property, we consider a slightly more general situation for a family
of linear mappings $\{S(t)\}_{t\geq0}$ acting in the scale as we now
define.

\begin{definition}
\label{def:linear_mappings_on_scale}
Given  the scale $\{X^\gamma\}_{\gamma\in \mathbb{J}}$  and a family
of linear operators $\{S(t)\}_{t\geq0}$ with $S(0)=I$ defined in a
consistent way on the spaces of the scale, that is, if $u_{0} \in
X^{\gamma} \cap X^{\tilde\gamma}$ then the value of $S(t)u_{0}$ as
operators in $X^{\gamma}$ and $X^{\tilde\gamma}$ coincide.

\begin{enumerate}
\item
   We say that $\{S(t)\}_{t\geq0}$ smooths from
  $X^\gamma$ to $X^{\tilde{\gamma}}$ for positive times, which we
  denote as
  \begin{displaymath}
    \gamma \stackrel{_{S(t)}}{\dashrightarrow} \tilde\gamma,
  \end{displaymath}
  iff for any $T>0$ there is a constant $M= M(\gamma,\tilde{\gamma}, T)$
  (that can always assume to be nondecreasing with respect to $T$),
  such that
  \begin{equation} \label{eq:estimate_scale} \|S(t) \|_{
      \mathcal{L}(X^\gamma,X^{\tilde{\gamma}})} \leq
    \frac{M}{ t^{d(\tilde{\gamma},\gamma)}}
    \ \text{ for } \ 0<t\leq T
  \end{equation}
  where
  $d(\tilde\gamma, \gamma)\mydef
  \mathtt{r}(\tilde\gamma)-\mathtt{r}(\gamma)\geq0$.

  \item
      We say that $\{S(t)\}_{t\geq0}$ continuously smooths
  from $X^\gamma$ to $X^{\tilde{\gamma}}$ for positive times, which we
  denote as
  \begin{displaymath}
    \gamma \stackrel{_{S(t)}}{\leadsto} \tilde\gamma,
  \end{displaymath}
  iff additionally to (i), we have that
  $(0,\infty) \ni t\to S(t)u_0\in X^{\tilde{\gamma}}$ is continuous
  for each $u_0\in X^\gamma$.

  \item
     If for each $\gamma\in\mathbb{J}$,
  $\{S(t)\}_{t\geq0}$ is a (not necessarily $C^0$)
  semigroup in $X^\gamma$ and
  $\gamma \stackrel{_{S(t)}}{\leadsto}\gamma$ then we say that
  we have a semigroup $\{S(t)\}_{t\geq0}$ in the scale.
\end{enumerate}
\end{definition}

Therefore the results in Sections
\ref{sec:homogeneous_equation_in_Morrey} and \ref{sec:morrey-potential} for
  Morrey spaces $ X^{\gamma}=M^{p,\ell}(\R^{N})$, correspond to
\begin{equation}\label{eq:Morrey_scale-2m}
\gamma  = (p,\ell)\in \mathbb{J} = [1,\infty]\times (0,N], \qquad  \mathtt{r}(\gamma)=-\frac{\ell}{2m\mu p}
\end{equation}
(for $p=1$,  $M^{1,\ell}(\R^N)$ can be replaced by $\mathcal{M}^{\ell}(\R^N)$).
The operators $\{S_{\mu}(t)\}_{t\geq0}$, $0< \mu\leq 1$, are a
semigroup in the Morrey scale in the sense above. Notice that
(\ref{eq:estimate_scale}) holds in this case with
$M$ independent of $T$, see
(\ref{eq:estimates_Mpl-Mqs}).

As the case of semigroups in a scale is  specially relevant for applications, we
make the following important remark.

\begin{lemma} \label{lem:estimates4semigroups_in_scales}

Assume  $\{S(t)\}_{t\geq0}$ is a semigroup in the scale
$\{X^{\gamma}\}_{\gamma \in \mathbb{J}}$.
\begin{enumerate}
\item
  For each $\gamma \in \mathbb{J}$, there exists
  constants $M_{0}= M_{0}(\gamma)$, $a_\gamma$, such that
  \begin{equation}
    \label{eq:exponential-bound-for-S(t)-in-Xgamma}
    \|S(t)\|_{\mathcal{L}(X^\gamma)} \leq M_{0}  e^{a_\gamma t},
    \quad  t\geq0,
  \end{equation}

\item
   If $\gamma\stackrel{_{S(t)}}{\dashrightarrow} \gamma'$
  denote $\omega \mydef  \min\{a_{\gamma}, a_{\gamma'}\}$. Then for each
  $T>0$ there exists a constant $M=M(\gamma,\gamma', T)$ such that
  \begin{displaymath}
    \|S(t)\|_{\mathcal{L}(X^\gamma, X^{\gamma'})} \leq
    \begin{cases}
      \frac{ M }{t^{d(\gamma',\gamma)}}, &  0<t\leq T\\
      M e^{\omega t}, & T<t .
    \end{cases}
  \end{displaymath}

  In particular, for each $a>\omega$ there is a constant
  $M_{1}=M_{1}(\gamma, \gamma')$  such that
  \begin{equation} \label{eq:another-exponential-for-S(t)}
    \|S(t)\|_{\mathcal{L}(X^\gamma, X^{\gamma'})} \leq \frac{ M_{1}
    }{t^{d(\gamma',\gamma)}} e^{a t}, \quad  t>0.
  \end{equation}
\end{enumerate}

\end{lemma}
\begin{proof}
(i) Since $\|S(\tau)\|_{\mathcal{L}(X^\gamma)} \leq C$ for
$\tau\in[0,1]$, then, by the semigroup property, for $k\in \N$ and $\tau
\in [0,1]$,  $\| S(k+\tau)\|_{X^\gamma} \leq C^{k+1}$
and, letting $t=k +\tau$, we get
(\ref{eq:exponential-bound-for-S(t)-in-Xgamma}).

\noindent (ii)
The estimate for $0<t\leq T$ comes from (\ref{eq:estimate_scale}). Now
for  $u_0\in X^\gamma$ and  $t>T$, by the semigroup property,
we get
$\|S(t) u_0 \|_{X^{\gamma'}}
        \leq
    \|S(t-T)\|_{\mathcal{L}(X^{\gamma'}, X^{\gamma'})} \|S(T) u_0 \|_{X^{\gamma'}}
    \leq        \frac{C}{T^{d(\gamma',\gamma)} e^{a_{\gamma' T}}}
    e^{a_{\gamma'} t} \| u_0\|_\gamma$.

Also, for $t>T$,
$\|S(t) u_0 \|_{X^{\gamma'}}    \leq
    \|S(T)\|_{\mathcal{L}(X^{\gamma},  X^{\gamma'})} \|S(t-T) u_0 \|_{X^{\gamma}}
    \leq        \frac{C}{T^{d(\gamma',\gamma)} e^{a_{\gamma T}}}
    e^{a_{\gamma} t} \| u_0\|_\gamma$.     So we    get the result.

   Also, these two estimates  together yield
(\ref{eq:another-exponential-for-S(t)}).
\end{proof}

Now we want to  consider an equation of the form
\begin{equation}\label{eq:abstract-VCF_linearly_perturbed_semigroup}
u(t) = S(t) u_0 +  \sum_{i=1}^n \int_0^t S(t-\tau)  P_i u(\tau) \,
d\tau, \quad t>0,
\end{equation}
for suitable  $u_0$ in some space in the scale, and   suitable linear
perturbations $P_{i}$ acting within the scale, as we now define.

\begin{definition}
  \label{def:perturbations_in_scale}

  Given $\alpha \in \mathbb{J}$ and $R>0$,

  \begin{enumerate}
  \item
     Denote $\mathscr{P}_{\beta,R}$, with
    $\beta \in \mathbb{J}$, the set of linear bounded mappings
    $P\in \mathcal{L}(X^\alpha,X^{\beta})$ with
    $0\leq d(\alpha,\beta)=\mathtt{r}(\alpha)-\mathtt{r}(\beta)<1$ and
    $\beta \stackrel{_{S(t)}}{\leadsto} \alpha$ and
    $\|P\|_{ \mathcal{L}(X^\alpha,X^{\beta})} \leq R$.

\item
  For $\beta_{1}, \ldots, \beta_{n} \in \mathbb{J}$
    consider sets of perturbations $P=\{P_1,\ldots,P_n\}$ such that
    $P_{i}\in \mathscr{P}_{\beta_{i},R}$. Then we say that
    \begin{equation} \label{eq:assumption-for-linear-perturbations}
      P=\{P_1,\ldots,P_n\} \in \mathscr{P}_{\beta_{1}, \ldots,
        \beta_{n},R} .
    \end{equation}
  \end{enumerate}
  \end{definition}

  Notice that if all $\beta_{i}$ are the same, then we can add the
perturbations and (\ref{eq:abstract-VCF_linearly_perturbed_semigroup})
would be equivalent to
\begin{equation}\label{eq:suming_all_perturbations}
  u(t) = S(t) u_0 +  \int_0^t S(t-\tau) \Big(\sum_{i=1}^n P_i \Big) u(\tau) \, d\tau,
\end{equation}
so, a single perturbation would be considered. Analogously, if in
(\ref{eq:abstract-VCF_linearly_perturbed_semigroup}) some
$\beta_{i}=\beta_{j}$ then $P_{i}$ and $P_{j}$ can be added into a
single perturbation. Hence we can,  without loss of generality, assume
that in (\ref{eq:abstract-VCF_linearly_perturbed_semigroup}) all
$\beta_{i}$ are different.

Hence if $P$ is as in (\ref{eq:assumption-for-linear-perturbations}),  notice that for
(\ref{eq:abstract-VCF_linearly_perturbed_semigroup}) to make sense we
need $u:(0,T) \to X^{\alpha}$ and then $\tau \mapsto S(t-\tau)  P_i
u(\tau) \in X^{\alpha}$, but it must be integrable, so we need a
precise control on how $u$ enters in $X^{\alpha}$ and use (\ref{eq:estimate_scale}). Also notice that we
can allow $u_{0}\in X^{\gamma}$ as long as $\gamma
\stackrel{_{S(t)}}{\leadsto} \alpha$.

This motivates to consider the following set of functions.

\begin{definition}
For $\alpha\in \mathbb{J}$, $T>0$ and $\varepsilon\geq0$ we define
\begin{displaymath}
{\mathcal  L}^\infty_{\alpha,\varepsilon} ((0,T])=\{\varphi\in L^{\infty}_{loc}((0,T], X^\alpha)\colon
\Norm{\varphi}_{\alpha,\varepsilon,T}= \sup_{t\in(0,T]}  t^{\varepsilon} \|\varphi(t) \|_\alpha<\infty \},
\end{displaymath}
and ${\mathcal  L}^\infty_{\alpha,\varepsilon}=\bigcap_{T>0} {\mathcal  L}^\infty_{\alpha,\varepsilon} ((0,T])$.
\end{definition}

\subsection{Perturbations in the scale. Existence, uniqueness and regularity}

Then we have the following  existence and uniqueness result for
(\ref{eq:abstract-VCF_linearly_perturbed_semigroup}), for $u_{0}$ in a
set of spaces in the scale determined by  $\mathcal{E}_{\alpha}$
below; the set of existence and uniqueness for
(\ref{eq:abstract-VCF_linearly_perturbed_semigroup}). Notice in
particular  that if
$\alpha \stackrel{_{S(t)}}{\leadsto} \alpha$ then $\alpha \in
\mathcal{E}_{\alpha}$.

\begin{theorem} [{\bf Existence of solutions}]
  \label{thm:existence-linear}

  Assume  $\alpha \in \mathbb{J}$ and $P=\{P_1,\ldots,P_n\}$ satisfies
  (\ref{eq:assumption-for-linear-perturbations}) and let
\begin{displaymath}
\gamma \in \mathcal{E}_\alpha\mydef \{\gamma \in \mathbb{J} \colon
\mathtt{r}(\gamma) \in(\mathtt{r}(\alpha)-1, \mathtt{r}(\alpha)]
\text{ and } \gamma \stackrel{_{S(t)}}{\leadsto} \alpha \}.
\end{displaymath}

Then for  $u_0\in X^\gamma$ there is a  unique
$u=u(\cdot,u_0)$ in ${\mathcal
  L}^\infty_{\alpha,d(\alpha,\gamma)}$ satisfying
(\ref{eq:abstract-VCF_linearly_perturbed_semigroup}) for each $t>0$.

Therefore we have a  family of linear operators $\{S_P(t)\}_{t\geq0}$
in $X^{\gamma}$ given by
\begin{equation}
  \label{eq:SP(t)}
  S_P(0)u_0=u_0 \ \text{ and } \
  S_P(t)u_0 = u(t,u_0)
 \end{equation}
and  $\{S_P(t)\}_{t\geq0}$ continuously smooths from $X^\gamma$ to
$X^\alpha$ for positive times, that is,
\begin{equation}
  \label{eq:crucial-property-of-SP(t)}
\gamma \stackrel{_{S_P(t)}}{\leadsto} \alpha.
\end{equation}

Finally,  $\beta_{i}\in
\mathcal{E}_{\alpha}$ and in particular,
\begin{displaymath}
\beta_i \stackrel{_{S_P(t)}}{\leadsto} \alpha, \quad i=1,\ldots n.
\end{displaymath}

Also, if $\alpha \stackrel{_{S(t)}}{\leadsto} \alpha$ then $\alpha \in
\mathcal{E}_{\alpha}$ and $\alpha \stackrel{_{S_{P}(t)}}{\leadsto}
\alpha$.

\end{theorem}
\begin{proof}
{\bf Step 1. Existence.}
  First, if $\gamma \in \mathcal{E}_{\alpha}$ then $0\leq d(\alpha, \gamma)
  = \mathtt{r}(\alpha) - \mathtt{r}(\gamma) <1$. Then with $T>0$ and $\theta>0$  we  consider in ${\mathcal
  L}^\infty_{\alpha,d(\alpha,\gamma)} ((0,T])$ an equivalent norm given by
\begin{displaymath}
  \Norm{\varphi}_{T,\theta}=
  \sup_{t\in(0,T]} e^{-\theta t}  t^{d(\alpha,\gamma)} \|\varphi(t) \|_\alpha.
\end{displaymath}

Then take $K_{0}>0$ to be chosen below and letting
\begin{displaymath}
\mathcal{K}_{T,K_0,\theta}=\{\varphi\in {\mathcal
  L}^\infty_{\alpha,d(\alpha,\gamma)} ((0,T]):
\Norm{\varphi}_{T,\theta}\leq K_0  \}
\end{displaymath}
and for $\varphi \in {\mathcal   L}^\infty_{\alpha,d(\alpha,\gamma)} ((0,T])$
 \begin{equation} \label{eq:VCF}
{\mathcal F}(\varphi,u_0) (t) = S(t) u_0 +  \sum_{i=1}^n\int_0^t
S(t-\tau)  P_i \varphi (\tau) \, d\tau, \quad t\in(0,T], \
\end{equation}
we look for a fixed point of $\varphi\mapsto {\mathcal F}(\varphi,u_0)$ in $\mathcal{K}_{T,K_0,\theta}$.

We remark that if $\varphi\in \mathcal{K}_{T,K_0,\theta}$ then
$\varphi$ is measurable with respect to a Lebesgue's $\sigma$-field  as
in \cite[Definition \S5, p. 4]{Dinculeanu} and, by assumptions on
$P_i$ and $\{S(t)\}_{t\geq0}$, we see via \cite[Proposition \S13,
p. 7]{Dinculeanu} that $S(t-\cdot)P_i \varphi(\cdot)$ is then
measurable for every $t\in(0,T]$ and $i=1,\ldots,n$.
Hence $S(t-\cdot)  P_i \varphi(\cdot)$ with values in $X^\alpha$ is integrable on $(0,t)$ whenever $\|S(t-\cdot)  P_i \varphi(\cdot)\|_\alpha$ is integrable on $(0,t)$.
Also, the technical  Lemma \ref{lem:continuity} proved below ensures that ${\mathcal
  F}(\varphi,u_0)$ is actually continuous in $(0,T]$ with values in
$X^\alpha$ and thus measurable (see \cite[Corollary
1.4.8]{Cazenave-Haraux}).

Given $\varphi,\psi\in
\mathcal{K}_{T,K_0,\theta}$ and $t\in (0,T]$
we have
\begin{equation} \label{eq:estimate_VCF}
\begin{split}
e^{-\theta t} &t^{d(\alpha,\gamma)} \|{\mathcal
  F}(\varphi,u_0)(t)\|_\alpha
  \leq
C \|u_0\|_{\gamma}
+ e^{-\theta t}t^{d(\alpha,\gamma)}  \sum_{i=1}^n \int_0^t \frac{C}{(t-s)^{d(\alpha,\beta_i)}}R\|\varphi(s)\|_\alpha ds
\\&
\leq  C \|u_0\|_{\gamma} +  C\sum_{i=1}^n R t^{1-d(\alpha,\beta_i)} e^{-\theta t}
 \int_0^1 \frac{e^{\theta t \zeta}}{(1-\zeta)^{d(\alpha,\beta_i)}
  \zeta^{d(\alpha,\gamma)}} \, d\zeta\ \Norm{\varphi}_{T,\theta}
\end{split}
\end{equation}
and
\begin{displaymath}
\begin{split}
e^{-\theta t} &t^{d(\alpha,\gamma)} \| ({\mathcal
  F}(\varphi,u_0)(t)-{\mathcal F}(\psi,u_0)(t))\|_\alpha
  \leq
e^{-\theta t}t^{d(\alpha,\gamma)} \sum_{i=1}^n \int_0^t
\frac{C}{(t-s)^{d(\alpha,\beta_i)}}R\|\varphi(s)-\psi(s)\|_\alpha
ds
\\&
\leq
C\sum_{i=1}^n R  t^{1-d(\alpha,\beta_i)} e^{-\theta t} \int_0^1
\frac{e^{\theta t\zeta}}{(1-\zeta)^{d(\alpha,\beta_i)}
  \zeta^{d(\alpha,\gamma)}} \, d\zeta\
\Norm{\varphi-\psi}_{T,\theta}.
\end{split}
\end{displaymath}

Now choose
$K_{0} > 2 C  \|u_0\|_{\gamma}$ and denote
\begin{displaymath}
c_{i}(\theta)= \sup_{t\in(0,T]} t^{1-d(\alpha,\beta_i)}e^{-\theta t}
\int_0^1  \frac{e^{\theta t \zeta}}{ (1-\zeta)^{d(\alpha,\beta_i)}
\zeta^{d(\alpha,\gamma)}} \, d\zeta
\end{displaymath}
so we obtain for $\varphi,\psi\in \mathcal{K}_{T,K_0,\theta}$ that
\begin{equation}\label{eq:mathcalF-in-the-sublinear-case}
  \begin{split}
      &\Norm{{\mathcal
          F}(\varphi,u_0)}_{T,\theta} \leq
      \Bigl(\frac12+ CR\sum_{i=1}^n c_{i}(\theta)   \Bigr) K_0,
      \\&
      \Norm{({\mathcal F}(\varphi,u_0)-{\mathcal F}(\psi,u_0)}_{T,\theta}\leq
CR      \sum_{i=1}^n c_{i}(\theta) \Norm{\varphi-\psi}_{T,\theta}.
\end{split}
\end{equation}

Given $i\in\{1,\ldots, n\}$ we have via  H\"older's inequality that for $\frac{1}{q}+\frac{1}{q'}=1$
\begin{displaymath}
\begin{split}
c_{i}(\theta)
&\leq  \sup_{t\in(0,T]} t^{1-d(\alpha,\beta_i)}e^{-\theta t} (\int_0^1 e^{\theta t\zeta q'} \, d\zeta)^\frac1{q'} (\int_0^1 (1-\zeta)^{-qd(\alpha,\beta_i)} \zeta^{-qd(\alpha,\gamma)} \, d\zeta)^\frac1q
\\&
\leq \sup_{t\in(0,T]} t^{\frac1{q}-d(\alpha,\beta_i)} (\theta q')^{-\frac1{q'}}( 1 - e^{-\theta t q'})^\frac1{q'} B^\frac1{q}(1-q d(\alpha,\beta_i), 1-q d(\alpha,\gamma))
\\&
\leq
\theta^{-\frac1{q'}} T^{\frac1{q}-d(\alpha,\beta_i)}   {q'}^{-\frac1{q'}}  B^\frac1{q}(1-q d(\alpha,\beta_i), 1-q d(\alpha,\gamma))
\end{split}
\end{displaymath}
where $B(\cdot,\cdot)$ is Euler's beta function and $1-q
d(\alpha,\beta_i)$, $1-q d(\alpha,\gamma)>0$ for $q>1$ close
enough to $1$, because $1-d(\alpha,\beta_i)>0$ and
$1-d(\alpha,\gamma)>0$. Therefore, for such $q$,  we  see that
$c_{i}(\theta)$ is bounded from above  by a
multiple of
$\theta^{-\frac1{q'}}T^{\frac1{q}-d(\alpha,\beta_i)} $ and then
\begin{equation}\label{eq:useful-property}
\text{
given $T>0$ we have $\lim_{\theta\to\infty} c_{i}(\theta)=0$ for every $i=1,\ldots n$}.
\end{equation}

Therefore, from this and   (\ref{eq:mathcalF-in-the-sublinear-case})
for a given $T>0$,
 we  can  choose $\theta$ large such that  ${\mathcal
  F}(\cdot,u_0)$ is a  contraction in
$\mathcal{K}_{T,K_0,\theta}$. Hence,   ${\mathcal
  F}(\cdot,u_0)$ has a unique fixed point $u$  in
$\mathcal{K}_{T,K_0,\theta}$ and then  $u\in {\mathcal
  L}^\infty_{\alpha,d(\alpha,\gamma)} ((0,T])$ satisfies
(\ref{eq:abstract-VCF_linearly_perturbed_semigroup}) for $t\in
(0,T]$.

\noindent {\bf Step 2. Uniqueness.}
For fixed $T$ notice the sets  $\mathcal{K}_{T,K_0,\theta}$ are
increasing in $K_{0}$ and in $\theta$. Hence if $v\in {\mathcal
  L}^\infty_{\alpha,d(\alpha,\gamma)} ((0,T])$ satisfies
(\ref{eq:abstract-VCF_linearly_perturbed_semigroup}) in $(0,T]$, with
$\theta$ and $u$  as in Step 1 above, denote $K_{1}= \max\{K_{0},
\Norm{v}_{T,\theta}\}$. Then $u,v\in \mathcal{K}_{T,K_1,\theta}$. Now
choose $\theta_{1}$ larger than $\theta$ and such $\mathcal{F}$ has a
unique fixed point in  $\mathcal{K}_{T,K_1,\theta_{1}} \supset
\mathcal{K}_{T,K_1,\theta} \supset
\mathcal{K}_{T,K_0,\theta}$. Therefore $u,v \in
\mathcal{K}_{T,K_1,\theta_{1}}$ and are fixed points, whence $u=v$ in
$(0,T]$. Therefore, there is a unique element in  ${\mathcal
  L}^\infty_{\alpha,d(\alpha,\gamma)} ((0,T])$ that satisfies
(\ref{eq:abstract-VCF_linearly_perturbed_semigroup}).

In particular, if  $T_{1}< T_{2}$ and $u_{i} \in {\mathcal
  L}^\infty_{\alpha,d(\alpha,\gamma)} ((0,T_{i}])$ satisfy
(\ref{eq:abstract-VCF_linearly_perturbed_semigroup}) in $(0,T_{i}]$
then $u_{1}=u_{2}$ in $(0,T_{1}]$.

As we can construct, as above, for each $T>0$ an $u\in {\mathcal
  L}^\infty_{\alpha,d(\alpha,\gamma)} ((0,T])$ that satisfies
(\ref{eq:abstract-VCF_linearly_perturbed_semigroup}) in $(0,T]$, we
have therefore a unique $u\in {\mathcal
  L}^\infty_{\alpha,d(\alpha,\gamma)}$ that satisfies
(\ref{eq:abstract-VCF_linearly_perturbed_semigroup}) for $t>0$.

\noindent {\bf Step 3. Linearity.} Now for $t>0$, $S_P(t)$ in  (\ref{eq:SP(t)})
is a well defined map  from $X^\gamma$ into $X^\alpha$. The linearity
of $S_P(t)$  is now a consequence of the uniqueness in ${\mathcal
  L}^\infty_{\alpha,d(\alpha,\gamma)}$ and the linearity in $u_{0}$
and in $\varphi$ in (\ref{eq:VCF}).

\noindent {\bf Step 4. Estimates.}
Now for $S_{P}(t) u_{0}= u(t;u_{0})$  and $T>0$, from
(\ref{eq:estimate_VCF}), for $0<t\leq T$ we have
\begin{displaymath}
\begin{split}
e^{-\theta t} &t^{d(\alpha,\gamma)} \|u(t)\|_\alpha
\leq  C \|u_0\|_{\gamma} +  CR\sum_{i=1}^n c_{i}(\theta)  \Norm{u}_{T,\theta},
\end{split}
\end{displaymath}
and therefore
\begin{displaymath}
\Norm{u}_{T,\theta} \leq
C \|u_0\|_{\gamma}+  CR\sum_{i=1}^n c_{i}(\theta)
      \Norm{u}_{T,\theta} .
\end{displaymath}
From  (\ref{eq:useful-property}) we can choose  $\theta$ large enough
such that
$ CR\sum_{i=1}^n c_{i}(\theta) \leq \frac12$, and then
$\Norm{u}_{T,\theta}\leq 2
C \|u_0\|_{\gamma}$ which in turn leads to the
estimate
\begin{equation}\label{eq:Xgamma-Xalpha-estimate}
\|S_{P}(t)u_{0}\|_{\alpha} \leq \frac{2C e^{\theta T}}{
  t^{d(\alpha,\gamma)}}   \|u_0\|_{\gamma}, \quad t\in (0,T).
\end{equation}
This and the continuity in Lemma \ref{lem:continuity} below, completes the
proof of (\ref{eq:crucial-property-of-SP(t)}).

Finally, that  $\beta_{i}\in
\mathcal{E}_{\alpha}$ and $\beta_i \stackrel{_{S_P(t)}}{\leadsto}
\alpha$, for $ i=1,\ldots n$ and that if $\alpha \stackrel{_{S(t)}}{\leadsto} \alpha$ then $\alpha \in
\mathcal{E}_{\alpha}$ and $\alpha \stackrel{_{S_{P}(t)}}{\leadsto}
\alpha$, follows by the definitions.
\end{proof}

We now prove the technical  lemma used above.

\begin{lemma}\label{lem:continuity}
Assume $P=\{P_1,\ldots,P_n\} $ satisfies (\ref{eq:assumption-for-linear-perturbations}) and let $\gamma\in \mathcal{E}_\alpha$.

If $u_0\in X^\gamma$ and $u\in {\mathcal L}^\infty_{\alpha,d(\alpha,\gamma)} ((0,T])$
then
\begin{displaymath}
(0,T]\ni t \to {\mathcal F}(u,u_0) (t) =S(t) u_0 +  \sum_{i=1}^n \int_0^t S(t-\tau)  P_i u(\tau) \, d\tau \in X^\alpha \ \text{ is continuous}.
\end{displaymath}

\end{lemma}

\begin{proof}
We consider $0<t\leq T$ and $h\in\R$ satisfying
\begin{equation}\label{eq:to-recall-below}
\frac{t}2\leq t+h\leq T .
\end{equation}
We see that
\begin{equation}\label{eq:to-cite-in-the-proof-continuity-in-Xgamma'}
\begin{split}
\|{\mathcal F}(u,u_0) (t+h)  &- {\mathcal F}(u,u_0) (t)  \|_{\alpha} \leq
\|S(t +h)u_0 - S(t) u_0 \|_{\alpha}
  \\&
  +
  \sum_{i=1}^n\|\int_0^{t+h} S(t+h-s)P_iu(s) ds - \int_0^{t} S(t-s)P_iu(s) \, ds
  \|_{\alpha}
  \\&
  \phantom{aaaaaaaaa.aaaa}=: I_{1,\alpha}(h) + I_{2,\alpha}(h).
\end{split}
\end{equation}
Since
$\gamma \stackrel{_{S(t)}}{\leadsto} \alpha$,  we get
 $\lim_{h\to 0} I_{1,\alpha}(h)=0$. The proof
that $\lim_{h\to 0} I_{2,\alpha}(h)= 0$ follows in two cases.

\smallskip

\noindent
\fbox{Case $h>0$.} If $h>0$ then  $I_{2,\alpha}(h)\leq j_{h,\alpha}^+ + k_{h,\alpha}^+$ where
\begin{equation}\label{eq:mention-to-replace-alpha-by-gamma'}
\begin{split}
&j_{h,\alpha}^+= \sum_{i=1}^n
  \int_0^{t} \| (S(t+h-s) - S(t-s)) P_i u(s)  \|_{\alpha} \, ds,
  \\&
  k_{h,\alpha}^+=\sum_{i=1}^n
  \int_{t}^{t+h} \|S(t+h-s)P_i u(s)  \|_{\alpha} \, ds.
  \end{split}
\end{equation}
We see that
$\| (S(t+h-s) - S(t-s)) P_iu(s) \|_{\alpha}$ is bounded by
$\| S(t+h-s)P_iu(s)\|_{\alpha} + \| S(t-s) P_iu(s) \|_{\alpha}$, which for $s\in(0,t)$ is estimated by
$
G_i(s)=  \frac{R C}{(t-s)^{d(\alpha,\beta_i)}s^{d(\alpha,\gamma)}}
  \Norm{u}_{\alpha,d(\alpha,\gamma),T}$. Since $d(\alpha,\beta_i)<1$ and $d(\alpha,\gamma)<1$,
function $G_i(s)$ is integrable for $s\in(0,t)$. For such $s$
we also have
\begin{equation}\label{eq:about-conv-to-zero}
\lim_{h\to 0^+}\| (S(t+h-s) - S(t-s)) P_iu(s) \|_{\alpha}=0,
\end{equation}
because $P_iu(s)\in X^{\beta_i}$ and by assumption $(0,\infty)\ni t \to S(t)\phi\in X^\alpha$ is continuous when $\phi \in X^{\beta_i}$.
Thus, via Lebesgue's theorem $\lim_{h\to0^+} j_{h,\alpha}^+=0$,

For $s\in(t,t+h)$
we see in turn that $\|S(t+h-s)P_i u(s)  \|_{\alpha}$ is bounded from
above by $ \frac{R  C \Norm{u}_{\alpha,d(\alpha,\gamma),T}}{(t+h-s)^{d(\alpha,\beta_i)}t^{d(\alpha,\gamma)}}
  $. Then $k_{h,\alpha}^+\leq \sum_{i=1}^n \frac{RC}{(1-d(\alpha,\beta_i))t^{d(\alpha,\gamma)}}
  \Norm{u}_{\alpha,d(\alpha,\gamma),T} h^{1-d(\alpha,\beta_i)}$, which implies that $\lim_{h\to0^+} k_{h,\alpha}^+ =0$.
As a consequence $\lim_{h\to0^+} I_{2,\alpha}(h) =0$.

\smallskip

\noindent
\fbox{Case $h<0$.} If $h<0$, we have  $I_{2,\alpha}(h)\leq j_{h,\alpha}^- + k_{h,\alpha}^-$ where
\begin{displaymath}
\begin{split}
&j_{h,\alpha}^-=\sum_{i=1}^n
  \int_0^{t+h} \| (S(t+h-s) - S(t-s)) P_i u(s) \|_{\alpha} ds,
  \\&
  k_{h,\alpha}^-=\sum_{i=1}^n\int_{t+h}^{t} \|S(t-s) P_i u(s) \|_{\alpha} ds.
  \end{split}
\end{displaymath}

For $s\in (t+h,t)$, recalling (\ref{eq:to-recall-below}), we see that $\|S(t-s) P_i u(s) \|_{\alpha}$ is estimated from above by $\frac{R C}{(t-s)^{d(\alpha,\beta_i)}(\frac{t}{2})^{d(\alpha,\gamma)}}
  \Norm{u}_{\alpha,d(\alpha,\gamma),T}$ and
$k_{h,\alpha}^-\leq \sum_{i=1}^n\frac{R C}{(1-d(\alpha,\beta_i))(\frac{t}{2})^{d(\alpha,\gamma)}}
  \Norm{u}_{\alpha,d(\alpha,\gamma),T} (-h)^{1-d(\alpha,\beta_i)}$, which yields
$\lim_{h\to0^-}k_{h,\alpha}^- =0$.

Given any $\xi>0$ such that
$\frac{t}4\leq t-\xi$ we now write for $h\in(-\xi,0)$ (thus $t-\xi \leq t+h$)
\begin{displaymath}
\begin{split}
&j_{h,\alpha}^-\leq
\sum_{i=1}^n \int_0^{t-\xi} \| (S(t+h-s) - S(t-s)) P_i u(s) \|_{\alpha} ds
\\&
+ \sum_{i=1}^n\int_{t-\xi}^{t+h} (\| S(t+h-s) P_i u(s) \|_{\alpha} +\| S(t-s) P_i u(s) \|_{\alpha}) ds
=:l_\alpha(h,\xi) + m_\alpha(h,\xi).
\end{split}
\end{displaymath}
Observe that, since $h\in(-\xi,0)$ and $\frac{t}{4}\leq t-\xi$, $\|
S(t+h-s) P_i u(s) \|_{\alpha} +\| S(t-s) P_i u(s) \|_{\alpha}$ is for
$s\in(t-\xi, t+h)$ bounded from above by $\frac{ 2R
  C}{(t+h-s)^{d(\alpha,\beta_i)}(\frac{t}{4})^{d(\alpha,\gamma)}}
  \Norm{u}_{\alpha,d(\alpha,\gamma),T}$, whereas $m_\alpha(h,\xi)\leq  \frac{2 nRC}{(1-d(\alpha,\beta_i))(\frac{t}{4})^{d(\alpha,\gamma)}}
  \Norm{u}_{\alpha,d(\alpha,\gamma),T}\xi^{d(\alpha,\beta_i)}$.
Hence, given $\eta>0$, there exists $\xi>0$ such that $m_\alpha(h,\xi)<\eta$
for all $h\in(-\xi,0)$.  Having fixed such $\xi$, note that
$(0,t-\xi)\subset (0,t+h)$ and
$\| (S(t+h-s) - S(t-s)) P_i u(s) \|_{\alpha}$ is for $s\in(0,t-\xi)$ bounded from above by $
H_i(s)=RC
\Norm{u}_{\alpha,d(\alpha,\gamma),T} \bigl(
(t-\xi-s)^{-d(\alpha,\beta_i)} s^{-d(\alpha,\gamma)} +
(t-s)^{-d(\alpha,\beta_i)} s^{-d(\alpha,\gamma)}\bigr)$ and that
$H_i(s)$ is integrable for $s\in(0,t-\xi)$, because $d(\alpha,\gamma)<1$ and $d(\alpha,\beta_i)<1$.
From (\ref{eq:about-conv-to-zero}) we also have $\lim_{h\to 0^+}\| (S(t+h-s) - S(t-s)) P_iu(s) \|_{\alpha}
=0$.
Therefore, via Lebesgue's dominated convergence theorem
$\lim_{h\to0^-} l_\alpha(h,\xi)=0$ and we conclude that $\lim_{h\to 0^-} I_{2,\alpha}(h)= 0$.
\end{proof}

Now we prove that the family of operators
$\{S_P(t)\}_{t\geq0}$ constructed in Theorem
\ref{thm:existence-linear} is consistent in the spaces $X^{\gamma}$
with $\gamma \in \mathcal{E}_\alpha$.

  \begin{lemma}
Assume $\alpha \in \mathbb{J}$ and $P=\{P_1,\ldots,P_n\}$ satisfies
(\ref{eq:assumption-for-linear-perturbations}) and let
    $\gamma,\tilde{\gamma}\in \mathcal{E}_\alpha$.

    Given $u_0\in X^\gamma\cap X^{\tilde{\gamma}}$, if
    $u\in {\mathcal L}^\infty_{\alpha,d(\alpha,\gamma)}$ and
    $\tilde{u}\in {\mathcal
      L}^\infty_{\alpha,d(\alpha,\tilde{\gamma})}$ are the unique
    functions satisfying
    (\ref{eq:abstract-VCF_linearly_perturbed_semigroup}) for $t>0$
    then $u(t)=\tilde{u}(t)$ for every $t>0$.

    Consequently, the family
    $\{S_P(t)\}_{t\geq0}$ defined in (\ref{eq:SP(t)}) is the family of
    consistent operators in the spaces $X^\gamma$,
    $\gamma \in \mathcal{E}_\alpha$.
  \end{lemma}
\begin{proof}
Without loss of  generality we can assume that $d(\alpha,\tilde{\gamma})\geq d(\alpha,\gamma)$.
Then, since
$t^{d(\alpha,\tilde{\gamma})}\|u(t)\|_\alpha=t^{d(\alpha,\tilde{\gamma})-d(\alpha,\gamma)}t^{d(\alpha,\gamma}\|u(t)\|_\alpha$
and $u\in  {\mathcal  L}^\infty_{\alpha,d(\alpha,\gamma)}$, we see
that
$u\in  {\mathcal  L}^\infty_{\alpha,d(\alpha,\tilde{\gamma})}$.
Then both $u$ and $\tilde{u}$ belong to ${\mathcal  L}^\infty_{\alpha,d(\alpha,\tilde{\gamma})}$ and satisfy
(\ref{eq:abstract-VCF_linearly_perturbed_semigroup})  for $t>0$, so that by the uniqueness in Theorem \ref{thm:existence-linear} we get the result.
\end{proof}

Assume $P=\{P_1,\ldots,P_n\}$ satisfies
(\ref{eq:assumption-for-linear-perturbations}). Since the set
$\mathcal{E}_{\alpha}$ does not depend on $P$, we can perturb the
original family $\{S(t)\}_{t\geq 0}$ in the scale sequentially by
first considering $\{S_{P_{1}}(t)\}_{t\geq 0}$ defined in $X^{\gamma}$
for  $\gamma\in \mathcal{E}_\alpha$ as the unique solutions in ${\mathcal L}^\infty_{\alpha,d(\alpha,\gamma)}$ of
\begin{displaymath}
u(t) = S(t) u_0 +  \int_0^t S(t-\tau)  P_1 u(\tau) \, d\tau
\end{displaymath}
 for $u_0\in X^\gamma$. Now, by Theorem \ref{thm:existence-linear}, since $\beta_{2}
\stackrel{_{S_{P_1}(t)}}{\leadsto} \alpha$ and
(\ref{eq:assumption-for-linear-perturbations}) we can perturb
$\{S_{P_{1}}(t)\}_{t\geq 0}$ with $P_{2}$ to get
$\{(S_{P_1})_{P_2}(t)\}_{t\geq0}$ as the unique solutions in $
{\mathcal L}^\infty_{\alpha,d(\alpha,\gamma)} $ of
\begin{displaymath}
 u(t) = S_{P_1}(t) u_0 +  \int_0^t S_{P_1}(t-\tau)  P_2 u(\tau) \,  d\tau,
\end{displaymath}
 to get $\{(S_{P_1})_{P_2}(t)\}_{t\geq0}$, and so on. Our next
 result shows that this sequential perturbation leads to the same
 family $\{S_{P}(t)\}_{t\geq 0}$. In particular, the order in which
 the perturbations are applied is irrelevant.

\begin{proposition} [{\bf Iterated perturbations}]
  \label{prop:iterated-perturbations}

  Assume $P \!=\! \{P_1,\ldots,\! P_n\}$ satisfies  (\ref{eq:assumption-for-linear-perturbations}) and
$\{S_P(t)\}_{t\geq0}$ is given by   (\ref{eq:SP(t)}).

 For $\gamma\in \mathcal{E}_\alpha$ and  $u_0\in
X^\gamma$ we have
\begin{displaymath}
S_P(t)u_0= ((S_{P_1})_{P_2}\ldots)_{P_n})(t)u_0, \qquad t>0 .
\end{displaymath}

\end{proposition}
\begin{proof}
 We use induction in $n$.
For $n=1$ there is nothing to prove.  Assume that the result  holds
for $n$ perturbations  we will now prove that it holds for $n+1$ ones.
Thus consider a set $P=\{P_1, \ldots,P_{n+1}\}   \in     \mathscr{P}_{\beta_{1}, \ldots,
  \beta_{n+1},R}$ as in
(\ref{eq:assumption-for-linear-perturbations}).  Denoting $\tilde P  =\{P_1, \ldots,P_{n}\}$,  by the  induction
assumption we have $S_{\tilde P}(t) =
((S_{P_1})_{P_2}\ldots)_{P_n})(t)$, $t>0$.

Given  $\gamma\in \mathcal{E}_\alpha$ and  $u_0\in X^\gamma$, let  $u,v\in
{\mathcal  L}^\infty_{\alpha,d(\alpha,\gamma)}$ be $u(t) =
S_{P}u_{0}(t)$ and $v(t)= ((S_{P_1})_{P_2}\ldots)_{P_{n+1}})(t) u_0$ for
$t>0$. Therefore, they are, respectively,  the unique solutions in ${\mathcal
  L}^\infty_{\alpha,d(\alpha,\gamma)}$ of
\begin{displaymath}
\begin{split}
  &u(t)=S(t)u_0 + \sum_{i=1}^{n+1}\int_0^t S(t-s) P_iu(s)  \, ds,
  \\&
  v(t)=S_{\tilde P} (t) u_0 + \int_0^t
 S_{\tilde P}(t-s) P_{n+1} v(s) \, ds .
\end{split}
\end{displaymath}

The induction assumption gives that
\begin{displaymath}
S_{\tilde P} (t)u_0=S(t)u_0+
\sum_{i=1}^{n} \int_0^t S(t-s) P_i S_{\tilde P} (s)u_0\, ds,
\end{displaymath}
and also
\begin{displaymath}
S_{\tilde P} (t-s)  P_{n+1} v(s) =S(t-s) P_{n+1} v(s)  +
\sum_{i=1}^{n} \int_0^{t-s} S(t-s-\xi ) P_i S_{\tilde P} (\xi)
P_{n+1} v(s) \, d\xi .
\end{displaymath}
Therefore
\begin{displaymath}
    v(t)= S(t)u_0 + \int_0^t S(t-s) P_{n+1} v(s) \, ds + \mathscr{R}(t)
\end{displaymath}
where
\begin{displaymath}
\begin{split}
\mathscr{R}(t) &= \sum_{i=1}^{n} \int_0^t S(t-s) P_i S_{\tilde P} (s)u_0\, ds
+ \int_0^t \left(\sum_{i=1}^{n} \int_0^{t-s} S(t-s-\xi) P_i S_{\tilde  P}  (\xi) P_{n+1}v(s) \, d\xi \right) ds
\\&
= \sum_{i=1}^{n} \int_0^t S(t-s) P_i S_{\tilde P} (s)u_0\, ds
+ \sum_{i=1}^{n} \int_0^t \left( \int_s^{t} S(t-\sigma) P_i S_{\tilde P} (\sigma-s) P_{n+1}v(s) \, d\sigma \right) ds .
\end{split}
\end{displaymath}
Thus, using the uniqueness  in Theorem \ref{thm:existence-linear}, we
will get $u(t) = v(t)$ for $t>0$  if we show that
\begin{equation}\label{eq:we-must-show-this}
\sum_{i=1}^{n} \int_0^t S(t-s) P_iv(s) \, ds=\mathscr{R}(t)
\end{equation}
for which we compute below the term in the left hand side.

For this, using again that $v(s)=S_{\tilde P}(s) u_0 + \int_0^s
S_{\tilde P}(s-\xi) P_{n+1} v(\xi) \, d\xi$ and that $S(t-s) P_i\in
\mathcal{L}(X^\alpha)$ for $i=1,\ldots,n$, we see that
\begin{displaymath}
\sum_{i=1}^{n} S(t-s) P_i v(s)
=
\sum_{i=1}^{n} S(t-s) P_i S_{\tilde P}(s) u_0
+ \sum_{i=1}^{n} \int_0^s S(t-s) P_i S_{\tilde P}(s-\xi) P_{n+1} v(\xi) \, d\xi,
\end{displaymath}
which after integration with respect to $s\in(0,t)$ yields
\begin{displaymath}
\begin{split}
\sum_{i=1}^{n} \int_0^t S(t-s) P_iv(s) \, ds=
&\sum_{i=1}^{n} \int_0^t S(t-s) P_i S_{\tilde P}(s) u_0 \, ds
\\&
\
+ \sum_{i=1}^{n} \int_0^t \left(\int_0^s S(t-s) P_i S_{\tilde P}(s-\xi) P_{n+1} v(\xi) \, d\xi \right) ds.
\end{split}
\end{displaymath}
Notice that the first term above is the same as the first term in
$\mathscr{R}(t)$ above.
After  changing the order of integration, the second  term above equals
\begin{displaymath}
  \sum_{i=1}^{n} \int_0^t \left(\int_\xi^t S(t-s) P_i S_{\tilde
      P}(s-\xi) P_{n+1} v(\xi) \, ds  \right) d\xi
\end{displaymath}
and after relabelling  $s \mapsto \sigma$ and $\xi \mapsto s$,  this   is
precisely the second term $\mathscr{R}(t)$ above. Hence
(\ref{eq:we-must-show-this}) holds true and the result is proved.
\end{proof}

Now we analyse the regularisation properties of the family
$\{S_P(t)\}_{t\geq0}$  in the scale. That is, the set of spaces to
which $S_{P}(t)u_{0}$ belongs; the regularity set for
(\ref{eq:abstract-VCF_linearly_perturbed_semigroup}). For
this, for each $\beta \in \mathbb{J}$ define
\begin{equation} \label{eq:regularity_set_beta}
\mathcal{R}_{\beta} \mydef \{\gamma' \in \mathbb{J} \colon
\mathtt{r}(\gamma') \in[\mathtt{r}(\beta), \mathtt{r}(\beta)+1)
\text{ and } \beta\stackrel{_{S(t)}}{\leadsto} \gamma' \}.
\end{equation}
In particular for such $\gamma'$ we have $0\leq d(\gamma', \beta) < 1$.

\begin{theorem} [{\bf Smoothing of solutions}]
  \label{thm:Xgamma-Xgamma'-estimates}
  Assume  $P=\{P_1,\ldots,P_n\}$ satisfies
  (\ref{eq:assumption-for-linear-perturbations}) and
  $\mathcal{R}_{\beta_{i}} $ as in (\ref{eq:regularity_set_beta}) and consider
  \begin{displaymath}
\mathcal{R}_{\beta_1,\ldots,\beta_n} \mydef \bigcap_{i=1}^{n}
\mathcal{R}_{\beta_{i}} .
\end{displaymath}
Then $\alpha \in \mathcal{R}_{\beta_1,\ldots,\beta_n}$. Moreover,
\begin{enumerate}
\item
    Assume $\gamma\in \mathcal{E}_\alpha$,
  $\gamma'\in \mathcal{R}_{\beta_1,\ldots,\beta_n}$ and
  $\gamma\stackrel{_{S(t)}}{\dashrightarrow} \gamma'$. Then
  $\gamma \stackrel{_{S_{P}(t)}}{\dashrightarrow} \gamma'$, that is,
  for any $T>0$ there exists $M=M(\gamma,\gamma', T)$  (non decreasing
  in $T$) such that
  \begin{equation}\label{eq:Xgamma-Xgamma'-estimates}
    \|S_P(t)\|_{\mathcal{L}(X^\gamma, X^{\gamma'})} \leq
    \frac{ M}{t^{d(\gamma',\gamma)}}, \quad
    0<t\leq T  .
  \end{equation}

\item
    Assume $\gamma\in \mathcal{E}_\alpha$,
$\gamma'\in \mathcal{R}_{\beta_1,\ldots,\beta_n}$ and
$\gamma\stackrel{_{S(t)}}{\leadsto} \gamma'$. Then
$\gamma \stackrel{_{S_{P}(t)}}{\leadsto} \gamma'$, that is,
(\ref{eq:Xgamma-Xgamma'-estimates}) holds and
\begin{equation}
  \label{eq:smoothing-action-of-u}
  (0,t) \ni t \to S_P(t)u_0  \in X^{\gamma'} \ \text{ is continuous for every } u_0\in X^{\gamma}.
\end{equation}
\end{enumerate}
\end{theorem}
\begin{proof}
Notice that  (\ref{eq:assumption-for-linear-perturbations})  implies
$\alpha \in \mathcal{R}_{\beta_1,\ldots,\beta_n}$.

Now, given  $u_0\in X^\gamma$ and $T>0$ and using (\ref{eq:abstract-VCF_linearly_perturbed_semigroup}) with
$u(t)=S_P(t)u_0$ we have for $0<t\leq T<\infty$
\begin{equation}\label{eq:Xgamma-estimate}
\begin{split}
 \| &u(t) \|_{\gamma'} \leq \frac{C}{t^{d(\gamma',\gamma)}} \|u_0\|_{\gamma} +
 \sum_{i=1}^n  \int_0^t
 \frac{C}{(t-s)^{d(\gamma',\beta_i)}}
 \|P_i\|_{{\mathcal L}(X^\alpha,X^{\beta_i})}\|u(s)\|_\alpha \, ds
\\&
\leq
\frac{C}{t^{d(\gamma',\gamma)}} \|u_0\|_{\gamma} +  CR  \sum_{i=1}^n
 \int_0^1
\frac{t^{1-d(\alpha,\gamma)-d(\gamma',\beta_i)}}{(1-\zeta)^{d(\gamma',\beta_i)}
  \zeta^{d(\alpha,\gamma)}} \, d\zeta \,
  \Norm{u}_{T}
  \end{split}
\end{equation}
where we have set $\Norm{u}_{T}=
\Norm{u}_{\alpha,d(\alpha,\gamma),T}$.
Since $d(\alpha,\gamma)<1$ and $d(\gamma',\beta_i)<1$ for all $i$, we conclude that
\begin{displaymath}
u(t)= S_P(t)u_0\in X^{\gamma'} \quad \text{ for } t>0.
\end{displaymath}

From   (\ref{eq:Xgamma-Xalpha-estimate}) we see that
$\Norm{u}_{T}\leq  C\| u_0\|_\gamma$.
Hence from  (\ref{eq:Xgamma-estimate}) we obtain
\begin{displaymath}
 \| u(t) \|_{\gamma'}
  \leq
  \!
  \left(
\frac{C}{t^{d(\gamma',\gamma)}} +  \! RC \sum_{i=1}^n
 t^{1-d(\alpha,\gamma)-d(\gamma',\beta_i)} \! \! \int_0^1 \! \! \!
\frac{d\zeta}{(1-\zeta)^{d(\gamma',\beta_i)}
  \zeta^{d(\alpha,\gamma)}}  \right)
 \| u_0\|_\gamma.
\end{displaymath}
Now, we multiply both sides of the above inequality by
$t^{d(\gamma',\gamma)}$ and observing  that
$1-d(\alpha,\gamma)-d(\gamma',\beta_i) + d(\gamma',\gamma) =
1-d(\alpha,\beta_i)>0$, because of
(\ref{eq:assumption-for-linear-perturbations}), we get
\begin{displaymath}
t^{d(\gamma',\gamma)} \| S_P(t)u_0\|_{\gamma'}
  \leq
C \| u_0\|_\gamma, \quad u_0\in X^\gamma, \ 0<t\leq T .
\end{displaymath}
This completes the proof of part (i).

Concerning part (ii) we only need to prove
(\ref{eq:smoothing-action-of-u}). Not being to exhaustive we now
observe that continuity in (\ref{eq:smoothing-action-of-u}) follows
analogously as in the proof of Lemma \ref{lem:continuity}. Namely, we
replace in (\ref{eq:to-cite-in-the-proof-continuity-in-Xgamma'}) and
(\ref{eq:mention-to-replace-alpha-by-gamma'}) $\alpha$ by $\gamma'$
and see that $\| (S(t+h-s) - S(t-s)) P_iu(s) \|_{\gamma'}$ is
estimated by function
$\frac{2RC}{(t-s)^{d(\gamma',\beta_i)}s^{d(\alpha,\gamma)}}
\Norm{u}_{\alpha,d(\alpha,\gamma),T}$, which is integrable for
$s\in(0,t)$ and for such $s$ we also see that
(\ref{eq:about-conv-to-zero}) holds true with $\alpha$ replaced by
$\gamma'$, because $P_iu(s)\in X^{\beta_i}$ and by assumption
$(0,\infty)\ni t \to S(t)\phi\in X^{\gamma'}$ is continuous when
$\phi \in X^{\beta_i}$. Hence we get, via Lebesgue's theorem
$\lim_{h\to0^+} j_{h,\gamma'}^+=0$. After similar modifications we get
$\lim_{h\to0^+} I_{2,\gamma'}(h) =0$ and
$\lim_{h\to 0^-} I_{2,\gamma'}(h)= 0$, which leads to
(\ref{eq:smoothing-action-of-u}).
\end{proof}

Now we can prove the joint continuity of $\{S_{P}(t)\}_{t\geq0}$ with
respect to time and the initial data.

\begin{corollary}
  \label{cor:continuity-of-SP(t)u0-with-respect-to-a-pair-of-argument}

  Assume  $P=\{P_1,\ldots,P_n\}$ satisfies (\ref{eq:assumption-for-linear-perturbations}).

If $\gamma\in \mathcal{E}_\alpha$, $\gamma'\in \mathcal{R}_{\beta_1,\ldots,\beta_n}$ and $\gamma\stackrel{_{S(t)}}{\leadsto} \gamma'$ then
\begin{displaymath}
(0,\infty)\times X^\gamma \ni (t,u_0) \to S_P(t) u_0 \in X^{\gamma'}
\ \text{ is continuous}.
 \end{displaymath}
\end{corollary}
\begin{proof}
Consider $(t_0,u_0)\in  (0,\infty)\times X^\gamma$ and a sequence
$\{(t_n,u_{0n})\}\subset(0,\infty)\times X^\gamma$ such that
$(t_n,u_{0n}) \to (t_0,u_0)$ in  $(0,\infty)\times X^\gamma$ as
$n\to\infty$.
Write
\begin{displaymath}
S_P(t_n)u_{0n}- S_P(t_0)u_{0}= S_P(t_n)u_{0n} - S_P(t_n)u_{0} + S_P(t_n)u_{0} - S_P(t_0)u_{0}.
\end{displaymath}
From  (\ref{eq:Xgamma-Xgamma'-estimates})
$\|S_P(t_n)u_{0n} - S(t_n)u_{0}\|_{X^{\gamma'}}
\leq \frac{c}{t_n^{d(\gamma',\gamma)}}
\|u_{0n} -u_{0}\|_{X^\gamma} \to 0$ as $n\to \infty$. On the other
hand, from  (\ref{eq:smoothing-action-of-u}),
$\|S_P(t_n)u_{0} - S_P(t_0)u_{0}\|_{X^{\gamma'}} \to 0$ as $n\to \infty$.
Hence, $\lim_{n\to \infty}\|S_P(t_n)u_{0n}- S_P(t_0)u_{0}\|_{X^{\gamma'}} \to 0$.
\end{proof}

The next result describes the  behavior of $S_P(t)u_0$ at $t=0$.

\begin{theorem}
Assume $P=\{P_1,\ldots,P_n\}$ satisfies
(\ref{eq:assumption-for-linear-perturbations}) and let $\gamma\in
\mathcal{E}_\alpha$  and
$\gamma'\in \mathcal{R}_{\beta_1,\ldots,\beta_n}$ such that $0\leq
d(\gamma',\gamma) < 1-d(\alpha,\beta_i)$ for all $i=1,\ldots, n$.

Then for  $u_0\in X^\gamma$,
\begin{equation}\label{eq:lim-of=int0t-ato+}
  \lim_{t\to0^+}\| S_P(t) u_0 - S(t) u_0 \|_{\gamma'} =0 .
\end{equation}
\end{theorem}
\begin{proof}
Observe that $v(t) =  S_P(t) u_0 - S(t) u_0$ satisfies, as the second
term in (\ref{eq:Xgamma-estimate}),
\begin{displaymath}
  \|v(t)\|_{\gamma'} \leq  R C \sum_{i=1}^n
t^{1-d(\alpha,\gamma)-d(\gamma',\beta_i)}  \int_0^1 \frac{1}{(1-\zeta)^{d(\gamma',\beta_i)}
  \zeta^{d(\alpha,\gamma)}} \, d\zeta \,
  \Norm{S_P(\cdot)u_0}_{T}
\end{displaymath}
and $1-d(\alpha,\gamma)-d(\gamma',\beta_i) = 1-d(\alpha,
\beta_{i})-d(\gamma',\gamma)>0$
which leads to (\ref{eq:lim-of=int0t-ato+}).
\end{proof}

Now we show Lipschitz continuous dependence of $S_P(t)u_0$ with
respect to $P$ and $u_{0}$.

\begin{theorem} [{\bf Continuous dependence on perturbations}]
  \label{thm:Lipschitz-properties-of-S_P(t)}

  Assume $P=\{P_1,\ldots,P_n\}$ and
$\tilde{P}=\{\tilde{P}_1,\ldots,\tilde{P}_n\}$ satisfy
(\ref{eq:assumption-for-linear-perturbations}) and define
$|P-\tilde
P|_{\alpha,\beta_1,\ldots,\beta_n}=\max_{i=1,\ldots,n}\|P_i-\tilde{P}_i\|_{{\mathcal
    L}(X^\alpha,X^{\beta_i})}$. Also, assume $\gamma\in
\mathcal{E}_\alpha$ and $u_0,\tilde{u}_0\in X^\gamma$ are such that
$\|u_0\|_{\gamma}, \|\tilde{u}_0\|_{\gamma}\leq \mathscr{R}$.

Then for  $\gamma'\in \mathcal{R}_{\beta_1,\ldots,\beta_n}$  such that
$\gamma\stackrel{_{S(t)}}{\dashrightarrow} \gamma'$, and $T>0$, we
have, for $0<t\leq T$,
\begin{equation}\label{eq:estimate-for-P-tildeP-as-well-as-u1-u2}
\|S_{P}(t)u_{0} - S_{\tilde{P}}(t)\tilde{u}_0\|_{\gamma'}
\leq \frac{M_{0}}{t^{d(\gamma',\gamma)}} \left( \|u_0-\tilde{u}_0 \|_{\gamma} + |P-\tilde P|_{\alpha,\beta_1,\ldots,\beta_n} \right),
\end{equation}
and
\begin{equation}\label{eq:estimate-for-P1-P2}
\|S_{P}(t) - S_{\tilde{P}}(t)\|_{{\mathcal L}(X^\gamma,X^{\gamma'})}
\leq \frac{M_{1}}{t^{d(\gamma',\gamma)}} |P-\tilde P|_{\alpha,\beta_1,\ldots,\beta_n} ,
\end{equation}
where $M_{0}$ and $M_{1}$ depend on $\alpha, \beta_{i}, \gamma,
\gamma', R$ and
$T$. Additionally $M_{0}$ depends on $\mathscr{R}$.
\end{theorem}
\begin{proof}
We see that $U(\cdot):= S_{P}(\cdot) u_{0}-S_{\tilde{P}}(\cdot) \tilde{u}_{0}$ satisfies
\begin{equation*}
  \begin{split}
U(t) = S(t)(u_{0}- \tilde{u}_{0}) &+ \sum_{i=1}^n\int_0^t S(t-\tau)  \left( P_i U(\tau) + (P_i-\tilde{P}_i) S_{\tilde{P}}(\tau) \tilde{u}_{0}\right) d\tau
,\ t>0
  \end{split}
\end{equation*}
and that given  $\theta>0$ we have, for $0<t\leq T$,
\begin{displaymath}
  \begin{split}
  e^{-\theta t} &t^{d(\gamma',\gamma)}
 \|U(t)\|_{\gamma'} \leq C \|u_{0}- \tilde{u}_{0}\|_{\gamma} +
C \sum_{i=1}^n \int_0^t \frac{e^{-\theta t}
   t^{d(\gamma',\gamma)}}{(t-\tau)^{d(\gamma', \beta_i)}}  R \|U(\tau
 )\|_{\alpha}\, d\tau
\\&
\phantom{aaaaaaaaaaaaa}+ C \sum_{i=1}^n  \int_0^t
\frac{t^{d(\gamma',\gamma)} }{(t-\tau)^{d(\gamma',\beta_i)}} |P-\tilde
P|_{\alpha,\beta_1,\ldots,\beta_n} \|S_{\tilde{P}}(\tau)
\tilde{u}_0\|_\alpha \, d\tau
\\&\leq C \|u_{0}- \tilde{u}_{0}\|_{\gamma} + R C \sum_{i=1}^n  \int_0^1
\frac{t^{1-d(\alpha,\beta_i)}e^{-\theta t(1-\zeta)}}{(1-\zeta)^{d(\gamma',\beta_i)}
  \zeta^{d(\alpha,\gamma)}} \, d\zeta\ \Norm{U}_{T,\theta}
\\&
+ |P-\tilde P|_{\alpha,\beta_1,\ldots,\beta_n} C \sum_{i=1}^n \int_0^1
\frac{t^{1-d(\alpha,\beta_i)}e^{-\theta t(1-\zeta)}}{(1-\zeta)^{d(\gamma',\beta_i)}
  \zeta^{d(\alpha,\gamma)}} \, d\zeta\ \Norm{S_{\tilde{P}}(s)
  \tilde{u}_0}_{T,\theta},
\end{split}
\end{displaymath}
where we have set $\Norm{\cdot}_{T,\theta}= \Norm{\cdot}_{\alpha,d(\alpha,\gamma),T,\theta}$.
From  (\ref{eq:Xgamma-Xalpha-estimate}) we have
$\Norm{S_{\tilde{P}}(s) \tilde{u}_0}_{T,\theta}\leq
C  \|\tilde{u}_0\|_\gamma$,
and then
\begin{equation}\label{eq:and-again-to-use-just-below}
\begin{split}
&\Norm{U}_{\gamma', d(\gamma',\gamma),T,\theta}
\leq RC \sum_{i=1}^n c_{i}(\theta) \Norm{U}_{T,\theta}
\\&
+C \left( 1 +  \mathscr{R} \sum_{i=1}^n c_{i}(\theta) \right)\left( \|u_{0}-
\tilde{u}_{0}\|_{\gamma} + |P-\tilde
P|_{\alpha,\beta_1,\ldots,\beta_n}\right)
.
\end{split}
\end{equation}
with
\begin{displaymath}
c_{i}(\theta)= \sup_{t\in(0,T]} t^{1-d(\alpha,\beta_i)}e^{-\theta t}
\int_0^1  \frac{e^{\theta t \zeta}}{ (1-\zeta)^{d(\gamma',\beta_i)}
\zeta^{d(\alpha,\gamma)}} \, d\zeta .
\end{displaymath}

Now, observe that $\alpha \in \mathcal{R}_{\beta_1,\ldots,\beta_n} $,
see Theorem \ref{thm:Xgamma-Xgamma'-estimates}, and since $\gamma\in
\mathcal{E}_\alpha$ then $\gamma\stackrel{_{S(t)}}{\leadsto}
\alpha$. Hence, we can take first  $\gamma'=\alpha$,
and as in (\ref{eq:useful-property}),
for $\theta$ large enough $RC \sum_{i=1}^n c_{i}(\theta) \leq
\frac12$ and hence we get
\begin{displaymath}
\Norm{U}_{T,\theta}
\leq 2  C\left( 1+\frac{ \mathscr{R}}{R}\right)\left( \|u_{0}- \tilde{u}_{0}\|_{\gamma} + |P-\tilde P|_{\alpha,\beta_1,\ldots,\beta_n}\right)
.
\end{displaymath}
Plugging this in the right side of
(\ref{eq:and-again-to-use-just-below}) we obtain
(\ref{eq:estimate-for-P-tildeP-as-well-as-u1-u2}).

Finally,  (\ref{eq:estimate-for-P-tildeP-as-well-as-u1-u2}) with
$u_0=\tilde{u}_0$ and $\mathscr{R}=1$ gives
(\ref{eq:estimate-for-P1-P2}).
\end{proof}

\subsection{Perturbations of a   semigroup in the scale}

We now assume that we have a semigroup in the scale as in Definition
\ref{def:linear_mappings_on_scale}. As noticed above Theorem
\ref{thm:existence-linear} in this case we have $\alpha \in
\mathcal{E}_{\alpha}$.
Our goal is to show that
$\{S_{P}(t)\}_{t\geq0}$ is still a semigroup at least in some spaces
of the scale. The first basic result is the following.

\begin{proposition}\label{prop:linear-perturbed-semigroup}
Assume  $\{S(t)\}_{t\geq0}$ is  a semigroup in the scale and
$P=\{P_1,\ldots,P_n\}$ satisfies
(\ref{eq:assumption-for-linear-perturbations}).

If $\gamma\in \mathcal{E}_\alpha$ and $u_0\in X^\gamma$ then
\begin{equation}\label{eq:smgp-property}
S_P(t_1+t_2)u_0=S_P(t_1)S_P(t_2)u_0
\end{equation}
holds as the equality in $X^\alpha$ for all $t_{1}, t_{2}>0$.
\end{proposition}
\begin{proof}
  For $u_0\in X^\gamma$, using
  (\ref{eq:abstract-VCF_linearly_perturbed_semigroup})  and that
  $\{S(t)\}_{t\geq0}$ is a semigroup in the scale, we have
  \begin{displaymath}
S_P(t+t_1)u_0= S(t)S(t_1)u_0 +  S(t)\sum_{i=1}^n\int_0^{t_1}S(t_1-\tau)  P_i S_P(\tau) \, d\tau
+\sum_{i=1}^n\int_{t_1}^{t+t_1}S(t+t_1-\tau)  P_i S_P(\tau) \, d\tau.
\end{displaymath}
After the change of variable $\tau-t_1=s$ in the last integral above, we conclude that
\begin{displaymath}
S_P(t+t_1)u_0 = S(t) S_P(t_1)u_0 +  \sum_{i=1}^n\int_0^t S(t-s)  P_i S_P(s+t_1)u_0 \, ds,
\end{displaymath}
that is, $v(\cdot)= S_P(\cdot+t_1)u_0$ satisfies
\begin{displaymath}
v(t) = S(t) S_P(t_1)u_0 +  \sum_{i=1}^n\int_0^t S(t-s)  P_i v(s) \, ds
\end{displaymath}
for $t>0$. By the uniqueness in  ${\mathcal
  L}^\infty_{\alpha,d(\alpha,\gamma)}$, we get $v(t) = S_{P}(t)
S_P(t_1)u_0$ for $t>0$ as claimed.
\end{proof}

\begin{corollary}
Assume  $\{S(t)\}_{t\geq0}$ is  a semigroup in the scale and
$P=\{P_1,\ldots,P_n\}$ satisfies
(\ref{eq:assumption-for-linear-perturbations}).

Then for $c\in \R$,  $\{P, cI\}$ satisfies
(\ref{eq:assumption-for-linear-perturbations}) with
$\beta_{n+1}=\alpha$   and for $\gamma \in\mathcal{E}_\alpha$ and  $u_0\in X^\gamma$
\begin{displaymath}
 S_{\{P,cI\}}(t) u_{0}=  (S_P)_{cI}(t) u_0=  (S_{cI})_{P}(t) u_0=e^{ct}
 S_P(t)u_0, \quad t>0.
\end{displaymath}

\end{corollary}
\begin{proof}
Since $\alpha\stackrel{_{S(t)}}{\leadsto}\alpha$ then clearly $\{P, cI\}$ satisfies
(\ref{eq:assumption-for-linear-perturbations}) with
$\beta_{n+1}=\alpha$.  Also, from Proposition \ref{prop:iterated-perturbations} it is enough to show that
$(S_P)_{cI}(t) u_0 =e^{ct}  S_P(t)u_0$ for $\gamma
\in\mathcal{E}_\alpha$ and  $ u_0\in X^\gamma$ and $t>0$.

Let $v \in {\mathcal  L}^\infty_{\alpha,d(\alpha,\gamma)}$ be $v(t)=
e^{c t} S_{P}(t)u_{0}$ for $t>0$. Then
\begin{displaymath}
  \begin{split}
    S_{P}(t)& u_{0} + \int_{0}^{t} S_{P}(t-s) cv(s) \, ds = S_{P}(t)
    u_{0} + \int_{0}^{t} S_{P}(t-s) c e^{cs} S_{P}(s) u_{0} \, ds
    \\
  & \myeq{~(\ref{eq:smgp-property})} \big( 1 + \int_{0}^{t} c e^{cs}\, ds \big)  S_{P}(t) u_{0} = e^{ct}
    S_{P}(t) u_{0}= v(t)
  \end{split}
\end{displaymath}
which is precisely the integral  equation for  $u(t) = (S_P)_{cI}(t) u_0$.
\end{proof}

In the next theorem we specify spaces in the scale in which $\{S_P(t)\}_{t\geq 0}$
is a semigroup.

\begin{theorem} [{\bf Perturbed semigroup in the scale}]
  \label{thm:linear-perturbed-semigroup}
Assume $\{S(t)\}_{t\geq0}$ is a semigroup in the scale and
$P=\{P_1,\ldots,P_n\}$ satisfies
(\ref{eq:assumption-for-linear-perturbations}). Define for each
$P_{i}$, $\Sigma_{\alpha,\beta_{i}}  \mydef
\mathcal{E}_\alpha\cap\mathcal{R}_{\beta_{i}}$ and
\begin{displaymath}
  \Sigma_{\alpha,\beta_1,\ldots,\beta_n}  \mydef
  \mathcal{E}_\alpha\cap\mathcal{R}_{\beta_1,\ldots,\beta_n}
  =\bigcap_{i=1}^n \Sigma_{\alpha,\beta_{i}}  .
\end{displaymath}

Then $\alpha \in  \Sigma_{\alpha,\beta_1,\ldots,\beta_n}$ and for
\begin{displaymath}
  \gamma\in \Sigma_{\alpha,\beta_1,\ldots,\beta_n}  =\bigcap_{i=1}^n
  \{\gamma\in \mathbb{J} \colon
  \beta_i\stackrel{_{S(t)}}{\leadsto}\gamma, \
  \gamma\stackrel{_{S(t)}}{\leadsto} \alpha \ \text{ and } \
  \mathtt{r}(\gamma)\in[\mathtt{r}(\beta_i),\mathtt{r}(\alpha)]\}
\end{displaymath}
 then
$\{S_P(t)\}_{t\geq0}$ is a semigroup in $X^\gamma$ satisfying for some
constants $M_{0}= M_{0}(\gamma)$, $\omega_\gamma$
\begin{equation}
  \label{eq:exponential-bound-for-SV(t)-n-Xgamma}
  \|S_P(t)\|_{\mathcal{L}(X^\gamma)} \leq M_{0}  e^{\omega_\gamma t},
  \quad  t\geq0
\end{equation}
and for  $u_0\in X^\gamma$
\begin{equation}\label{eq:SV(t) u0 - S(t)u0-as-t-to-0}
\lim_{t\to0^+}
\| S_P(t) u_0 - S(t)u_0\|_{X^\gamma} =0 .
\end{equation}

For  $\gamma'\in \mathcal{R}_{\beta_1,\ldots,\beta_n}$ such that
$\gamma\stackrel{_{S(t)}}{\dashrightarrow} \gamma'$ and  $T>0$ there exists a
constant $M=M(\gamma,\gamma', T)$
such that
\begin{displaymath}
    \|S_{P}(t)\|_{\mathcal{L}(X^\gamma, X^{\gamma'})} \leq
  \begin{cases}
    \frac{ M }{t^{d(\gamma',\gamma)}}, &  0<t\leq T\\
    M e^{\omega_{\gamma} t},  & T<t .
  \end{cases}
\end{displaymath}
In particular, for  any $\omega>\omega_{\gamma}$, there exists a constant
$M_{1}=M_{1}(\gamma, \gamma')$ independent of $t>0$ such that
\begin{equation}\label{eq:another-exponential-Xgamma-Xgamma'-estimates}
  \|S_P(t)\|_{\mathcal{L}(X^\gamma, X^{\gamma'})} \leq  \frac{M_{1}
  }{t^{d(\gamma',\gamma)}} e^{\omega t}, \quad  t>0 .
\end{equation}

\end{theorem}
\begin{proof}
This is a consequence of Theorem \ref{thm:Xgamma-Xgamma'-estimates},
Proposition \ref{prop:linear-perturbed-semigroup} and Lemma
\ref{lem:estimates4semigroups_in_scales}. Finally, using
    (\ref{eq:lim-of=int0t-ato+}) with $\gamma'=\gamma$ we obtain
    (\ref{eq:SV(t) u0 - S(t)u0-as-t-to-0}).
\end{proof}

  Finally we  obtain some result in which we have a more
  precise estimate on the exponential type of $\{S_{P}(t)\}_{t\geq0}$,
  that is, of the exponents in the exponentials in
  (\ref{eq:exponential-bound-for-SV(t)-n-Xgamma}) for $\gamma=\alpha$. This estimate will
  be obtained in terms of the corresponding exponentials for
  $\{S(t)\}_{t\geq0}$ and the size of the perturbations.

  \begin{proposition} [{\bf Exponential bounds for the perturbed semigroup}]
    \label{prop:exponential-bound-for-SV(t)-in-Xalpha}
Assume $\{S(t)\}_{t\geq0}$ is a semigroup in the scale and
$P=\{P_1,\ldots,P_n\}$ satisfies
(\ref{eq:assumption-for-linear-perturbations})
  and assume
    \begin{displaymath}
      \|S(t)\|_{\mathcal{L}(X^\alpha)} \leq   M_{0}
      e^{\omega_\alpha t}, \quad  t\geq0 .
    \end{displaymath}

Then for any $a>\omega_{\alpha}$ there exists a constant $M_{1}$ such
that for $u_{0}\in X^{\alpha}$
    \begin{displaymath}
      \| S_{P}(t) u_0\|_{\alpha} \leq M_{1}    \|u_0\|_{\alpha}
      e^{(a +\theta_{P}) t}, \qquad t>0
    \end{displaymath}
    with
\begin{displaymath}
      \theta_{P} = \sum_{i=1}^{n}  c_{i} \|P_{i}\|_{{\mathcal
            L}(X^\alpha,X^{\beta_{i}})}^{\frac{1}{1-d(\alpha,\beta_{i})}}.
    \end{displaymath}
for some constants $c_{i}=c_{i}(\alpha, \beta_{i})$.

    Also for any $\gamma\in \mathcal{E}_\alpha$ and $T>0$ there is a
    certain constant $M=M(T)$ such that
    \begin{equation}\label{eq:exponential-bound-for-SV(t)-n-Xgamma-Xalpha}
      \|S_P(t)\|_{\mathcal{L}(X^\gamma, X^{\alpha})} \leq
      \begin{cases}
        \frac{ M }{t^{d(\alpha,\gamma)}}, &  0<t\leq T\\
        M e^{(a + \theta_{P}) t}, & T<t.
      \end{cases}
    \end{equation}

  \end{proposition}

  \begin{proof}
    Using Proposition \ref{prop:iterated-perturbations} we add one
    perturbation at a time.
By the assumption we have
$\beta_{1}\stackrel{_{S(t)}}{\leadsto} \alpha$ and by  Lemma
\ref{lem:estimates4semigroups_in_scales}
we have, for any  $a > \omega_{\alpha} $,
    \begin{equation} \label{eq:estimates_to_perform_perturbation}
  \|S(t)\|_{\mathcal{L}(X^\alpha)} \leq   C       e^{at}, \quad       \|S(t)\|_{\mathcal{L}(X^{\beta_{1}}, X^{\alpha})} \leq  \frac{
        C }{t^{d(\alpha,\beta_{1}) }}
      e^{a t}, \quad  t>0.
    \end{equation}
        Then
    \begin{displaymath}
      \| S_{P_{1}}(t) u_0\|_{\alpha} \leq C  e^{a t}
      \|u_0\|_{\alpha} + \int_0^t \frac{C  e^{a (t-s)}}{(t-s)^{d(\alpha,\beta_{1})}}
      \|P_{1}\|_{{\mathcal  L}(X^\alpha, X^{\beta_{1}})}\|S_{P_{1}}(s)u_0\|_\alpha \, ds, \quad t>0 .
    \end{displaymath}

    Hence,
    $u(t) = e^{-a t} \| S_{P_{1}}(t) u_0\|_{\alpha}$
    satisfies
    \begin{displaymath}
      u(t)  \leq C   \|u_0\|_{\alpha} + \int_0^t
      \frac{C}{(t-s)^{d(\alpha,\beta_{1})}}   \|P_{1}\|_{{\mathcal
          L}(X^\alpha, X^{\beta_{1}})} u(s) \, ds, \quad t>0 .
    \end{displaymath}
    Then from \cite[Lemma 7.1.1]{1981_Henry} we get
    \begin{displaymath}
      u(t) \leq C   \|u_0\|_{\alpha}  e^{ \theta_{1} t}, \qquad t\geq0
    \end{displaymath}
    with
    \begin{displaymath}
      \theta_{1}=\left(C\Gamma(1-d(\alpha,\beta_{1}))
        \|P_{1}\|_{{\mathcal
            L}(X^\alpha,X^{\beta_{1}})}\right)^\frac{1}{1-d(\alpha,\beta_{1})}.
    \end{displaymath}
    Hence
    \begin{equation}\label{eq:esimate_after_perturbation}
      \| S_{P_{1}}(t) u_0\|_{\alpha} \leq C   \|u_0\|_{\alpha}
      e^{(a +\theta_{1}) t}  \quad t>0 .
    \end{equation}

 Now by assumption we have $\beta_{2}\in
 \mathcal{E}_{\alpha}$ and then,  by  Theorem
 \ref{thm:existence-linear},  $\beta_2
    \stackrel{_{S_{P_{1}}(t)}}{\leadsto} \alpha$ so
    we use again Lemma \ref{lem:estimates4semigroups_in_scales} for
    $S_{P_{1}}(t)$ so we get
    \begin{displaymath}
       \|S_{P_{1}}(t)\|_{\mathcal{L}(X^{\beta_{2}}, X^{\alpha})} \leq  \frac{
C }{t^{d(\alpha,\beta_{2}) }}      e^{(a +\theta_{1}) t}, \quad  t>0
    \end{displaymath}
which together to (\ref{eq:esimate_after_perturbation}) is like (\ref{eq:estimates_to_perform_perturbation}) but for
$S_{P_{1}} (t)$.

Now we  perturb this semigroup  with $P_{2}$ and denote
    $P=\{P_{1}, P_{2}\}$ then from the argument above, for any $a >
    \omega_{\alpha}$,
    \begin{displaymath}
      \| S_{P}(t) u_0\|_{\alpha} \leq C   \|u_0\|_{\alpha}
      e^{(a +\theta_{1} +\theta_{2}) t} , \quad t>0
    \end{displaymath}
    with $\theta_{2}=\left(C\Gamma(1-d(\alpha,\beta_{2}))        \|P_{2}\|_{{\mathcal
            L}(X^\alpha,X^{\beta_{2}})}\right)^{\frac{1}{1-d(\alpha,\beta_{2})}}$.

Reiterating the perturbations  for $P=\{P_1,\ldots,P_n\}$ we get for any $a >
    \omega_{\alpha}$,
    \begin{displaymath}
      \| S_{P}(t) u_0\|_{\alpha} \leq C   \|u_0\|_{\alpha}
      e^{(a +\theta_{P}) t} , \quad t>0
    \end{displaymath}
    with $\theta_{P} = \sum_{i=1}^{n} \theta_{i}$ as in the
    statement.

    Finally, (\ref{eq:exponential-bound-for-SV(t)-n-Xgamma-Xalpha}) is
    a consequence of Lemma \ref{lem:estimates4semigroups_in_scales}.
  \end{proof}

\subsection{Perturbations of an  analytic semigroup in the scale}

We now consider the case when the unperturbed  semigroup is analytic
with sectorial generator in some space of the scale  as we now define, see e.g.
\cite[Definition 2.0.1]{lunardi95:_analy} although notice that we
changed a bit the notations in this reference. The goal is to show that
the perturbed semigroup is also analytic and to identify its
generator.

\begin{definition}\label{def:analytic-Semigroup-in-the-scale}
If  $\{S(t)\}_{t\geq0}$ is a semigroup in a Banach space  $X$,  we say that $\{S(t)\}_{t\geq0}$ is
analytic in $X$ with sectorial generator iff  there is a
linear operator $L$ defined on domain $D(L)$ (which we do not assume
to be  dense in $X$) such that,
denoting in general, for $a \in \R$, $\theta\in \left(0,  \pi\right)$,
\begin{displaymath}
  S_{a,\theta} = \C \setminus \overline{C_{a,\theta}} \quad \text{
    where } \   C_{a,\theta} = \{z\in \C\setminus\{a\}, \  |Arg
  (z-a)|< \theta \},
\end{displaymath}
then for some $a_{0} \in \R$, $\theta_{0}\in \left(0,
  \frac{\pi}{2}\right)$ we have $S_{a_{0},\theta_{0}}\subset \rho(L)$
(the resolvent set of $L$) and
\begin{equation}\label{eq:L-is-sectorial}
\sup_{\lambda\in S_{a_{0},\theta_{0}}} {|\lambda-a_{0}|}
\|(L-\lambda)^{-1} \|_{\mathcal{L} (X)} <\infty
\end{equation}
and
\begin{displaymath}
  S(t)= \frac1{2\pi i}\int_{a_{0}+\Gamma_{r,\eta}} e^{-\lambda t}
  (L - \lambda)^{-1} \, d\lambda \quad  \ \text{ for } t>0
\end{displaymath}
where, for  fixed $0 < \theta_{0} < \eta<\frac{\pi}{2}$ and $r>0$,
$\Gamma_{r,\eta}$ denotes the clockwise oriented  curve
$\{\lambda\in\C\colon |Arg(\lambda)|=\eta, |\lambda|\geq r\}\cup
\{\lambda\in\C\colon |Arg(\lambda)|\leq \eta, |\lambda|= r\}$.

In such a case, we write $S(t)=e^{-L t}$ in $X$ for $t>0$ and $-L$ is
the sectorial  generator of the semigroup.

\end{definition}

\begin{remark}

  \begin{enumerate}
  \item
    When a semigroup is analytic with sectorial generator as above, then the resolvent
    operator of $L$ can be computed with the semigroup. More
    precisely, if $\{S(t)\}_{t\geq0}$ is analytic in $X$ with
    sectorial generator as above, then in particular
    $\{\lambda \in \C \colon Re(\lambda)<a_{0}\}\subset \rho(L)$,
    \begin{equation}\label{eq:domain-of-L}
      D(L)= (L-\lambda)^{-1} (X),  \quad Re(\lambda) <a_{0}
    \end{equation}
    and from \cite[(2.1.1)(a)]{lunardi95:_analy} for some constant $c$
    \begin{displaymath}
      \| S(t)\|_{\mathcal{L}(X)}\leq c e^{-a_{0} t}, \quad  t>0.
    \end{displaymath}

    Then, from \cite[Lemma 2.1.6]{lunardi95:_analy}, for
    $Re(\lambda) <a_{0}$ and $ u_0 \in X$, we have
    \begin{equation} \label{eq:resolvent-of-L}
      (L-\lambda)^{-1} u_0 = G(\lambda) u_0 \mydef \int_0^\infty e^{\lambda t} S(t)
      u_0\, dt .
          \end{equation}

  \item
To identify the sectorial  generator of the perturbed semigroup
$\{S_P(t)\}_{t\geq0}$ and to prove is analytic, we will consider the
candidate for resolvent, as in  the right hand side of
(\ref{eq:resolvent-of-L}), see  (\ref{eq:Fgammagamma'(lambda)}) below,
which are denoted \emph{pseudoresolvents},
and prove, using a result in \cite{lunardi95:_analy}, that  (\ref{eq:Fgammagamma'(lambda)})  is
actually the resolvent of some operator, see Lemma
\ref{lem:operator-Lambda}.

  \end{enumerate}
\end{remark}

The next result establishes a relationship between the
\emph{pseudoresolvents} of the semigroups $\{S(t)\}_{t\geq0}$ and
$\{S_P(t)\}_{t\geq0}$.  Observe that we do not use yet that
$\{S(t)\}_{t\geq0}$ is analytic. Also notice that the subscripts
in the operators $ G_{\gamma,\gamma'}$ and $F_{\gamma,\gamma'}$ below  are
used to indicate in which spaces these operators act.

\begin{proposition}
  \label{prop:pseudoresolvents}

Assume $\{S(t)\}_{t\geq0}$ is a semigroup in the scale and
$P=\{P_1,\ldots,P_n\}$ satisfies
(\ref{eq:assumption-for-linear-perturbations}).

\begin{enumerate}
\item
   If $\gamma\stackrel{_{S(t)}}{\leadsto} \gamma'$
  with $d(\gamma',\gamma)<1$ then let $a_{\gamma'}$ be
  the exponent in (\ref{eq:exponential-bound-for-S(t)-in-Xgamma}) for
  the space $X^{\gamma'}$. Then for  $a>a_{\gamma'}$ and
  $Re(\lambda) <-a$, the pseudoresolvent maps
  \begin{equation}\label{eq:Gbetaigamma0(lambda)}
    X^{\gamma}\ni u_0 \to
    G_{\gamma,\gamma'}(\lambda)u_0=\int_0^\infty e^{\lambda t} S(t)
    u_0\, dt \in X^{\gamma'}
  \end{equation}
  satisfy,  for some constant $M$,
  \begin{equation}\label{eq:crucial-bound-for-analyticity-result}
    \|G_{\gamma,\gamma'}(\lambda) \|_{{\mathcal L}(X^{\gamma},X^{\gamma'})}
    \leq
    \frac{M\Gamma(1-d(\gamma',\gamma))}{|Re(\lambda + a)|^{1-d(\gamma',\gamma)}},
    \qquad Re(\lambda) <-a,
  \end{equation}
  where $\Gamma(\cdot)$ is Euler's gamma function.

  \item
    For $\gamma\in \mathcal{E}_\alpha$ and
    $\gamma'\in \Sigma_{\alpha,\beta_1,\ldots,\beta_n}$ let
    $\omega_{\gamma'}$ be  the exponent
  in (\ref{eq:exponential-bound-for-SV(t)-n-Xgamma}) in the space
  $X^{\gamma'}$. Then for
  $\omega>\omega_{\gamma'}$ and   $Re(\lambda) <- \omega$, the pseudoresolvent maps
  \begin{equation}\label{eq:Fgammagamma'(lambda)}
    X^{\gamma}\ni u_0 \mapsto
    F_{\gamma,\gamma'}(\lambda)u_0 \mydef \int_0^\infty e^{\lambda t} S_P(t)
    u_0\, dt \in X^{\gamma'}
  \end{equation}
  satisfy, for some constant $C$,
  \begin{displaymath}
    \|F_{\gamma,\gamma'}(\lambda) \|_{{\mathcal L}(X^{\gamma},X^{\gamma'})}
    \leq \frac{C\,
      \Gamma(1-d(\gamma',\gamma))}{|Re(\lambda+\omega)|^{1-d(\gamma',\gamma)}}
    ,     \qquad  Re(\lambda) < -\omega .
  \end{displaymath}

  \item
       For $\gamma\in\Sigma_{\alpha,\beta_1,\ldots,\beta_n}$
  and $\gamma\stackrel{_{S(t)}}{\leadsto} \gamma$, there exists  $k>0$ such
  that for $Re(\lambda) <-k$, the pseudoresolvents
  $G_{\gamma,\gamma}(\lambda)$, $G_{\beta_i,\gamma}(\lambda)$ with
  $i=1,\ldots,n$ and $F_{\gamma,\gamma}(\lambda)$,
  $F_{\gamma,\alpha}(\lambda)$ as above are well defined  and for
  $u_0\in X^{\gamma}$
  \begin{equation}
    \label{eq:crucial-equality-to-prove-analiticity}
    F_{\gamma,\gamma}(\lambda) u_0= G_{\gamma,\gamma}(\lambda) u_0 +
    \sum_{i=1}^n G_{\beta_i,\gamma}(\lambda) P_i
    F_{\gamma,\alpha}(\lambda)u_0,  \qquad     Re(\lambda) < -k.
  \end{equation}
\end{enumerate}

\end{proposition}
\begin{proof}
  (i)
Since  for $u_0\in X^{\gamma}$,  $(0,\infty)\ni t \to S(t) u_0 \in
X^{\gamma'}$ is continuous  and     $Re(\lambda) < -a$, then using
    (\ref{eq:another-exponential-for-S(t)}) we have
    $\int_{0}^{\infty} \| e^{\lambda t} S(t) u_0 \,
    dt\|_{X^{\gamma'}} \leq M \int_0^\infty \frac{e^{Re(\lambda+a)t}
    }{t^{d(\gamma',\gamma)}}\, dt \, \|u_0\|_{X^{\gamma}}$ and we get
    the estimate.

\noindent   (ii)
From Theorem \ref{thm:Xgamma-Xgamma'-estimates},
$(0,\infty)\ni t  \to S_P(t) u_0 \in X^{\gamma'}$ is continuous and
since $Re(\lambda) < -\omega$ and
(\ref{eq:another-exponential-Xgamma-Xgamma'-estimates}), we see that
$\int_0^\infty \|e^{\lambda t} S_P(t) u_0\, dt\|_{X^{\gamma'}} \leq
C \int_0^\infty   \frac{e^{Re(\lambda+\omega)t}
}{t^{d(\gamma',\gamma)}}\, dt \, \|u_0\|_{X^{\gamma}}$ and we get the
estimate.

\noindent (iii)
For  $u_0\in X^\gamma$ and negative  enough $Re(\lambda)$, since
$u(t)= S_P(t)u_0$ satisfies
(\ref{eq:abstract-VCF_linearly_perturbed_semigroup}), we multiply this
expression  by $e^{\lambda t}$ and then we integrate with respect to
$t\in(0,\infty)$  to get
\begin{equation}\label{eq:to-transform-below}
F_{\gamma,\gamma}(\lambda) u_0= G_{\gamma,\gamma}(\lambda) u_0 +
\sum_{i=1}^n \int_0^\infty \left(e^{\lambda t}  \int_0^t S(t-s) P_i
  S_P(s) u_0 \, ds \right) dt.
\end{equation}
After changing  the order of integration and the  change of variable $\tau=t-s$ we see that
\begin{displaymath}
\begin{split}
\sum_{i=1}^n
\int_0^\infty \bigl(\int_0^t e^{\lambda t} S(t-s) P_i S_P(s) u_0 \, ds \bigr) dt&=  \sum_{i=1}^n\int_0^\infty \bigl(\int_s^\infty e^{\lambda t} S(t-s) P_i S_P(s) u_0 \, dt \bigr) ds
\\
&
= \int_0^\infty \bigl(\sum_{i=1}^n \int_0^\infty e^{\lambda \tau} S(\tau) P_i e^{\lambda s} S_P(s) u_0 \, d\tau\bigr) ds
\\&
=\int_0^\infty \sum_{i=1}^n G_{\beta_i,\gamma} (\lambda) P_i e^{\lambda s} S_P(s) u_0 ds.
\end{split}
\end{displaymath}
Using (\ref{eq:crucial-bound-for-analyticity-result}) with $\gamma=\beta_i$, $\gamma'=\gamma$ and  (\ref{eq:assumption-for-linear-perturbations}) we see that
$\sum_{i=1}^n G_{\beta_i,\gamma}(\lambda) P_i\in
\mathcal{L}(X^\alpha,X^{\gamma})$, whereas using
(\ref{eq:crucial-bound-for-analyticity-result}) with $\gamma'=\alpha$ we see that
$e^{\lambda \cdot} S_P(\cdot) u_0\in L^1((0,\infty),X^\alpha)$.
Therefore, via \cite[Proposition 1.4.22]{Cazenave-Haraux} we obtain
\begin{displaymath}
\int_0^\infty \sum_{i=1}^n G_{\beta_i,\gamma} P_i e^{\lambda s} S_P(s) u_0 ds=\sum_{i=1}^n G_{\beta_i,\gamma} (\lambda)P_i \int_0^\infty e^{\lambda s} S_P(s) u_0 ds
= \sum_{i=1}^n G_{\beta_i,\gamma}(\lambda)P_i F_{\gamma,\alpha}(\lambda)u_0.
\end{displaymath}
As a consequence, for all negative enough $Re(\lambda)$ the right hand
side in (\ref{eq:to-transform-below}) is equal to
$G_{\gamma,\gamma}(\lambda) u_0 +\sum_{i=1}^n
G_{\beta_i,\alpha}(\lambda) P_i F_{\gamma,\alpha}(\lambda)u_0$ and
thus (\ref{eq:crucial-equality-to-prove-analiticity}) is proven.
\end{proof}

The next result shows that if  $\{S(t)\}_{t\geq0}$ is analytic  in
$X^{\gamma}$  with sectorial generator  then the pseudoresolvents above
are actually the resolvent of some operator.

\begin{lemma}\label{lem:operator-Lambda}
Assume $\{S(t)\}_{t\geq0}$ is a semigroup in the scale and
$P=\{P_1,\ldots,P_n\}$ satisfies
(\ref{eq:assumption-for-linear-perturbations}).

Assume  $\gamma\in\Sigma_{\alpha,\beta_1,\ldots,\beta_n}$
and $\{S(t)\}_{t\geq0}$ is analytic  in $X^{\gamma}$ with sectorial generator,
then there exists $\omega_\gamma$  such that for  $Re(\lambda) <-\omega_\gamma$
the maps $F_{\gamma,\gamma}(\lambda)$ in Proposition
\ref{prop:pseudoresolvents} are well defined  and there exists a linear
operator $\Lambda$   in $X^{\gamma}$
such that $\{\lambda\in \C \colon Re(\lambda)<-\omega_{\gamma}\}\subset
\rho(\Lambda)$ and for $ u_0\in X^{\gamma}$
\begin{displaymath}
 (\Lambda - \lambda)^{-1}u_0 =
 F_{\gamma,\gamma}(\lambda) u_0 , \qquad  Re(\lambda) <
 -\omega_{\gamma} .
\end{displaymath}
Also, the domain $D(\Lambda)$ of $\Lambda$ is given by
\begin{equation}\label{eq:domain-of-Lambda}
  D(\Lambda)= (\Lambda - \lambda)^{-1} (X^{\gamma}),  \quad
  Re(\lambda) < -\omega_{\gamma}.
\end{equation}

\end{lemma}

\begin{proof}
From  Proposition \ref{prop:pseudoresolvents} with $\gamma'=\gamma$,
the  maps $F_{\gamma,\gamma}(\lambda)\in \mathcal{L}(X^{\gamma})$  are
well defined  for $Re(\lambda) < -\omega_{\gamma}$.
We also have, see \cite[p. 42]{lunardi95:_analy},
\begin{equation}\label{eq:to-refer-to-Lunardi}
F_{\gamma,\gamma}(\lambda_1)-F_{\gamma,\gamma}(\lambda_2)
=
(\lambda_1-\lambda_2)F_{\gamma,\gamma}(\lambda_1)F_{\gamma,\gamma}(\lambda_2), \quad
Re(\lambda_1),\ Re(\lambda_2) < -\omega_{\gamma} .
\end{equation}

Now we prove that  $F_{\gamma,\gamma}(\lambda)$ is injective  for  $Re(\lambda) < -\omega_{\gamma}$.
Due to (\ref{eq:to-refer-to-Lunardi}), if $u_0\in X^{\gamma}$ and
$F_{\gamma,\gamma}(\lambda)u_0=0$ for some  $Re(\lambda) < -\omega_{\gamma}$ then
$F_{\gamma,\gamma}(\lambda)u_0=0$ for all  $Re(\lambda) <
-\omega_{\gamma}$.  Then for any  functional $l\in (X^{\gamma})'$ the
function  $\chi_l(\cdot)=l(S_P(\cdot) u_0)$
satisfies
\begin{equation}\label{eq:null_laplace_transform}
\int_0^\infty e^{\lambda t} \chi_l(t) \, dt =0 \quad \text{ for all } \
Re(\lambda) < -\omega_{\gamma}.
\end{equation}

Then we claim that  $\chi_l(\cdot)$ must be zero in $(0,\infty)$ for every
$l\in (X^{\gamma})'$, which implies that $S_P(\cdot)u_0=0$ in
$(0,\infty)$. Since $u(t)=S_P(t)u_0$ satisfies
(\ref{eq:abstract-VCF_linearly_perturbed_semigroup}) then we get that  $S(\cdot)u_0=0$
in $(0,\infty)$. Then \cite[Corollary 2.1.7]{lunardi95:_analy} yields
$u_0=0$. Thus $F_{\gamma,\gamma}(\lambda)$ is injective for
$Re(\lambda) < -\omega_{\gamma}$ and using
\cite[Proposition~A.0.2]{lunardi95:_analy} we get the result.

To prove the claim above, observe that
  $\chi_l$ grows like $e^{\omega_{\gamma t}}$.    So take $\lambda = b
  -n -1 $ with $b< -\omega_{\gamma}$ and $n=0,1,\ldots$ to get
  \begin{displaymath}
    \int_{0}^{\infty} e^{-nt} \big( e^{bt} e^{-t} \chi_l(t) \big)   =0
    .
  \end{displaymath}
Then   take $s= e^{-t}$ to get
  \begin{displaymath}
    \int_{0}^{1} s^{n} \big( s^{-b} \chi_{l}(-\ln(s)) \big)  =0
  \end{displaymath}
  and
  \begin{displaymath}
   | g(s)|: = | s^{-b} \chi_{l}(-\ln(s)) | \leq s^{-b}
   e^{- \omega_{\gamma} \ln(s)} = s^{-b -\omega_{\gamma}}
  \end{displaymath}
  and $b <-\omega_{\gamma}$ implies $g(s)$ is bounded in $(0,1)$.
  So $g$ is orthogonal in $(0,1)$ to all polynomials and then to all continuous
  functions and  to all $L^{1}(0,1)$ functions. So $g=0$ and the
  claim is proved.
\end{proof}

We are now ready to prove analyticity results for the perturbed
semigroup  $\{S_P(t)\}_{t\geq0}$. The first result is about
analyticity in the common domain of the perturbations.

\begin{theorem}\label{thm:Xalpha-analiticity-of-linearly-perturbed-semigroup}

Assume $\{S(t)\}_{t\geq0}$ is a semigroup in the scale and
$P=\{P_1,\ldots,P_n\}$ satisfies
(\ref{eq:assumption-for-linear-perturbations}).

If
$\{S(t)\}_{t\geq0}$ is analytic in $X^{\alpha}$ with sectorial
generator then $\{S_P(t)\}_{t\geq0}$ is analytic in $X^\alpha$ with
sectorial generator, $-L_{P}$.

\end{theorem}
\begin{proof}
By Definition \ref{def:analytic-Semigroup-in-the-scale} there exists
an operator $L$ satisfying  (\ref{eq:L-is-sectorial}) in $X^{\alpha}$. In particular
$\{\lambda\in \C \colon Re(\lambda) <a_0\}\subset \rho(L)$ and for
some constant $C_0$
\begin{equation}\label{eq:recall-resolvent-estimate-for-Lalpha}
\|(L - \lambda)^{-1} \|_{\mathcal L(X^\alpha)} \leq \frac{C_0}{|\lambda-a_0| },
\quad
Re(\lambda) < a_{0}
\end{equation}
and from (\ref{eq:resolvent-of-L}) and (\ref{eq:Gbetaigamma0(lambda)}), for $u_0 \in  X^{\alpha}$,
  \begin{displaymath}
(L-\lambda)^{-1} u_0 = G_{\alpha,\alpha}(\lambda) u_0,
\qquad Re(\lambda) < a_{0} .
\end{displaymath}
Now from   Lemma \ref{lem:operator-Lambda} with $\gamma=\alpha$ we
have the operator $\Lambda$  in $X^\alpha$ such that
$\{\lambda\in \C \colon Re(\lambda) < -\omega_{\alpha}\}\subset \rho(\Lambda)$
and for $u_0\in X^{\alpha}$,
\begin{equation}\label{eq:resolvent-of-Lambda-in-Xalpha}
(\Lambda - \lambda)^{-1}u_0=F_{\alpha,\alpha}(\lambda)u_0,
\qquad  Re(\lambda) < -\omega_{\alpha} .
\end{equation}

Using these and  part (iii) in Proposition
\ref{prop:pseudoresolvents} with $\gamma=\alpha$, we get for large
enough $k$ and $u_0\in X^{\alpha}$,
\begin{equation}
  \label{eq:crucial-equality-to-prove-analiticity''}
(\Lambda - \lambda)^{-1} u_0= (L - \lambda)^{-1} u_0 + \sum_{i=1}^n
G_{\beta_i,\alpha}(\lambda) P_i (\Lambda -\lambda)^{-1}u_0, \qquad
 Re(\lambda) < -k.
\end{equation}

Now from  (\ref{eq:assumption-for-linear-perturbations}) and
(\ref{eq:crucial-bound-for-analyticity-result}),  we can choose $k$ so
large that
\begin{displaymath}
\sup_{Re(\lambda)\leq -k}\|\sum_{i=1}^n G_{\beta_i,\alpha} (\lambda)P_i\|_{{\mathcal L}(X^\alpha)}
< \frac{1}{2}.
\end{displaymath}
Then we get
\begin{displaymath}
\|(\Lambda -\lambda)^{-1} u_0\|_{\alpha}\leq  2 \|(L - \lambda)^{-1} u_0\|_\alpha , \qquad
 Re(\lambda) < -k.
\end{displaymath}
and from  (\ref{eq:recall-resolvent-estimate-for-Lalpha}) we
obtain
\begin{displaymath}
  \|(\Lambda -\lambda)^{-1} u_0\|_{\alpha}\leq  \frac{2
    C_0}{|\lambda - a_0| } \|u_0\|_\alpha, \qquad
 Re(\lambda) < -k.
\end{displaymath}
Then $\|\lambda (\Lambda -\lambda)^{-1} \|_{{\mathcal L}(X^\alpha)}\leq
\frac{2C_0}{|1+\frac{a_0}{\lambda}|}$ for $Re(\lambda) < -k$
and, increasing $k$ if needed, we conclude that
\begin{equation}\label{eq:needed-in-remark-below'}
\sup_{Re(\lambda)\leq -k} |\lambda|\|(\Lambda -\lambda)^{-1} \|_{{\mathcal L}(X^\alpha)}\leq 4 C_0,
\end{equation}
as  $|\frac{a_0}{\lambda}|\leq \frac{|a_0|}{k}<\frac12$ for
$Re(\lambda) \leq -k$.

This implies, using  \cite[Proposition 2.1.11]{lunardi95:_analy} and
Definition \ref{def:analytic-Semigroup-in-the-scale}, that
for some $0<\theta < \frac{\pi}{2}$
\begin{displaymath}
S_{-k,\theta}\subset \rho(\Lambda), \quad
\sup_{\lambda\in S_{-k,\theta}} {|\lambda + k|} \|(\Lambda-\lambda)^{-1} \|_{\mathcal{L} (X^{\alpha})} <\infty.
\end{displaymath}
Hence $\Lambda$ is sectorial in $X^{\alpha}$
and the corresponding  analytic semigroup
$\{e^{-\Lambda t}\}_{t\geq0}$ is given for $\theta < \eta < \frac\pi2$ and
$r>0$ by the formula
  \begin{displaymath}
  e^{-\Lambda t}= \frac1{2\pi i}\int_{-k+\Gamma_{r,\eta}} e^{-\lambda t}
  (\Lambda -\lambda )^{-1} \, d\lambda, \quad  t>0.
\end{displaymath}
Also,  \cite[Lemma 2.1.6]{lunardi95:_analy} implies for $u_0\in
X^{\alpha}$,
\begin{equation}\label{eq:resolvent-of-e{-LPt}}
  (\Lambda-\lambda )^{-1} u_0= \int_{0}^{\infty} e^{\lambda t} e^{-\Lambda t} u_0 \, dt, \qquad
 Re(\lambda) < -k.
\end{equation}

Then combining
(\ref{eq:resolvent-of-Lambda-in-Xalpha}),
(\ref{eq:Fgammagamma'(lambda)})  and (\ref{eq:resolvent-of-e{-LPt}}) we have for $u_0\in X^\alpha$ and $Re(\lambda) < -k$,  $\int_0^\infty e^{\lambda t} S_P(t)u_0 \, dt=\int_0^\infty
e^{\lambda t}e^{-\Lambda t} u_0 \, dt$.  Thus, arguing as in
(\ref{eq:null_laplace_transform}) we get
\begin{equation}\label{eq:to-go-next-to-corollary}
S_P(t)u_0=e^{-\Lambda t}u_0, \quad u_0\in X^\alpha, \ t>0
\end{equation}
and therefore $\Lambda= L_{P}$  and $-L_{P}$ is the sectorial generator of the semigroup
$\{S_{p}(t)\}_{t\geq 0}$ in $X^{\alpha}$.
\end{proof}

In a similar way, the  next result is about analyticity in the range of a single
perturbation.

\begin{theorem}\label{thm:Xbeta-analiticity-of-linearly-perturbed-semigroup}
Assume (\ref{eq:assumption-for-linear-perturbations}) holds with
$\beta_1=\ldots=\beta_n=:\beta$. Then we can assume that we have only
one perturbation as in (\ref{eq:suming_all_perturbations}).

If
$\{S(t)\}_{t\geq0}$ is a semigroup in the scale and $\{S(t)\}_{t\geq0}$
is analytic in $X^{\beta}$ with sectorial generator then
$\{S_P(t)\}_{t\geq0}$ is analytic in $X^\beta$ with sectorial
generator, $-L_{P}$.

\end{theorem}
\begin{proof}
By Definition \ref{def:analytic-Semigroup-in-the-scale} there exists
an operator $L$ satisfying  (\ref{eq:L-is-sectorial}) in $X^{\beta}$.
In particular $\{\lambda\in \C \colon Re(\lambda) < a_0\}\subset \rho(L)$ and for
some constant $C_0$
\begin{equation}\label{eq:recall-resolvent-estimate-for-Lbeta}
\|(L - \lambda)^{-1} \|_{\mathcal L(X^\beta)} \leq \frac{C_0}{|\lambda
  - a_0| },
\qquad Re(\lambda) < a_{0}
\end{equation}
and from (\ref{eq:resolvent-of-L}) and
(\ref{eq:Gbetaigamma0(lambda)}), for $u_0 \in  X^{\beta}$,
  \begin{displaymath}
(L-\lambda)^{-1} u_0 = G_{\beta,\beta}(\lambda) u_0,
\qquad Re(\lambda) < a_{0} .
\end{displaymath}

Now from   Lemma \ref{lem:operator-Lambda} with $\gamma=\beta$ we
have the operator $\Lambda$  in $X^\beta$ such that
$\{\lambda\in \C \colon Re(\lambda) < -\omega_{\beta}\} \subset  \rho(\Lambda)$
and for $u_0\in X^{\beta}$,
\begin{displaymath}
(\Lambda  - \lambda)^{-1}u_0=F_{\beta,\beta}(\lambda)u_0 \quad \text{
  for } \quad  Re(\lambda) < -\omega_{\beta} .
\end{displaymath}

Part (ii) in Proposition  \ref{prop:pseudoresolvents} with
$\gamma=\beta$ and $\gamma'=\alpha$ gives that for a suitably large
$\omega$
\begin{equation}\label{eq:to-mention-in-the-proof-below'}
\|F_{\beta,\alpha}(\lambda) \|_{{\mathcal L}(X^{\beta},X^\alpha)}
\leq \frac{C\,
  \Gamma(1-d(\alpha,\beta))}{|Re(\lambda  + \omega)|^{1-d(\alpha,\beta)}},
\quad Re(\lambda) < - \omega.
\end{equation}
while part (iii) in Proposition
\ref{prop:pseudoresolvents} with $\gamma=\beta$ gives,  for large
enough $k$ and $u_0\in X^{\beta}$,
\begin{equation}
  \label{eq:crucial-equality-to-prove-analiticity''''}
(\Lambda -\lambda)^{-1} u_0= (L - \lambda)^{-1} u_0 + (L
- \lambda)^{-1} P F_{\beta,\alpha}(\lambda) u_0, \qquad  Re(\lambda) < -k.
\end{equation}

Now from  (\ref{eq:assumption-for-linear-perturbations}) and
(\ref{eq:to-mention-in-the-proof-below'}) we can choose $k$ so large that
\begin{displaymath}
\sup_{Re(\lambda)\leq -k}\| P  F_{\beta,\alpha}(\lambda)\|_{{\mathcal L}(X^\beta)}
< 1
\end{displaymath}
and then
\begin{displaymath}
  \|(\Lambda -\lambda)^{-1} u_0\|_{\beta}\leq  2 \|(L -
  \lambda)^{-1}\|_{\mathcal{L}(X^\beta)} \|u_0\|_\beta, \qquad Re(\lambda) < -k
\end{displaymath}
and from  (\ref{eq:recall-resolvent-estimate-for-Lbeta}) we obtain
\begin{displaymath}
  \|(\Lambda -\lambda)^{-1} \|_{\mathcal{L}(X^\beta)} \leq  \frac{2 C_0}{|\lambda -
    a_0| }, \qquad Re(\lambda) < -k.
\end{displaymath}
From this, arguing as in (\ref{eq:needed-in-remark-below'})--(\ref{eq:to-go-next-to-corollary}) we get
that for some $0 < \theta < \frac{\pi}{2}$
\begin{displaymath}
S_{-k,\theta}\subset \rho(\Lambda), \quad
\sup_{\lambda\in S_{-k,\theta}} {|\lambda + k|} \|(\Lambda-\lambda)^{-1} \|_{\mathcal{L} (X^{\beta})} <\infty,
\end{displaymath}
and
\begin{displaymath}
S_P(t)u_0=e^{-\Lambda t}u_0, \quad u_0\in X^\beta, \ t>0
\end{displaymath}
and therefore $\Lambda= L_{P}$  and $-L_{P}$ is the sectorial generator of the semigroup
$\{S_{p}(t)\}_{t\geq 0}$ in $X^{\beta}$.
\end{proof}

In the last results in this section, we  characterize the operator $L_P$.

\begin{proposition}\label{prop:LP-in-Xbeta}
In the case of Theorem
\ref{thm:Xbeta-analiticity-of-linearly-perturbed-semigroup}, let $-L$
be the sectorial  generator of $\{S(t)\}_{t\geq0}$ in $X^{\beta}$.

Then
\begin{displaymath}
  D(L_{P}) = D(L) \subset X^{\alpha}
\end{displaymath}
and for $v\in D(L_{P})$
\begin{displaymath}
  L_{P} v = Lv -P v \in X^{\beta }
\end{displaymath}

Finally for $u_{0}\in X^{\beta}$, $u(t) =S_P(t) u_0$, $t>0$, satisfies
in $X^{\beta}$
\begin{displaymath}
  u_{t} + Lu = P u, \quad t>0.
\end{displaymath}

\end{proposition}
\begin{proof}
We follow the notations in the proof of Theorem
\ref{thm:Xbeta-analiticity-of-linearly-perturbed-semigroup}. Using
(\ref{eq:domain-of-L}) and (\ref{eq:domain-of-Lambda})
with $\gamma=\beta$
we have for negative  enough $Re(\lambda)$
\begin{displaymath}
D(L)= (L-\lambda)^{-1}(X^\beta) , \qquad D(\Lambda)=  (\Lambda-\lambda)^{-1}(X^\beta) =D(L_P).
\end{displaymath}

Now  (\ref{eq:to-mention-in-the-proof-below'}) ensures  that for $Re(\lambda)$ negative enough the norm  of
$P F_{\beta,\alpha}(\lambda)$ in $X^\beta$ is strictly
less than $1$ and then $I+ P F_{\beta,\alpha}(\lambda)$ is bijective from $X^\beta$ in
$X^\beta$. But then, from
(\ref{eq:crucial-equality-to-prove-analiticity''''}) we have
\begin{displaymath}
(\Lambda - \lambda)^{-1} =   (L - \lambda)^{-1} \left(I + P
  F_{\beta,\alpha}(\lambda)\right)
\end{displaymath}
and then $D(L) =  D(\Lambda)$.

From  the consistency of these operators in the spaces of the
scale, if $u_{0}\in X^{\beta}$ then
$ v_{0}= (\Lambda-\lambda)^{-1} u_0 = F_{\beta,\alpha}(\lambda)u_{0} \in
X^\alpha$ and therefore $D(L) \subset X^{\alpha}$.

Also,  from (\ref{eq:crucial-equality-to-prove-analiticity''''}) we
have
\begin{displaymath}
  v_0=(L - \lambda)^{-1} \left(I+ P F_{\beta,\alpha}(\lambda)\right)
  (\Lambda -\lambda)v_0 .
\end{displaymath}
 Applying in both sides  $L - \lambda$ we get
\begin{displaymath}
Lv_0 = \Lambda v_0 + P F_{\beta,\alpha}(\lambda)(\Lambda -\lambda)v_0
= \Lambda v_0 + P F_{\beta,\alpha}(\lambda) u_0 = Lv _{0} + P v_{0} .
\end{displaymath}

The last statement about the equation $u_{t} + L_{P}u=0$
follows now  from \cite[Proposition  2.1.1]{lunardi95:_analy}  since the
perturbed semigroup is analytic in $X^{\beta}$.
\end{proof}

In the case the perturbations have different ranges we have the
following result on the operator $L_{P}$ in Theorem
\ref{thm:Xalpha-analiticity-of-linearly-perturbed-semigroup}  that requires that the unperturbed semigroup is well
defined and analytic in a \emph{superspace} that contains all target
spaces of the perturbations.

\begin{proposition} \label{prop:LP-in-Xalpha}
In the case of Theorem
\ref{thm:Xalpha-analiticity-of-linearly-perturbed-semigroup} assume there is a Banach space $Z$ such that
$X^\alpha\subset Z$, $X^{\beta_1}\subset Z,\ldots,X^{\beta_n}\subset
Z$ continuously and $\{S(t)\}_{t\geq0}$ is an analytic semigroup in
$Z$ with sectorial generator so that $S(t)=e^{-\mathscr{L} t}$ in
$Z$.

Then
\begin{displaymath}
  D_{X^{\alpha}}(L_{P}) \subset  D_{Z}(\mathscr{L})
\end{displaymath}
and for $v\in D_{X^{\alpha}}(L_{P})$
\begin{displaymath}
  L_{P} v = \mathscr{L} v -P v .
\end{displaymath}

Finally for $u_{0}\in X^{\alpha}$, $u(t) =S_P(t) u_0$, $t>0$,
satisfies in $X^{\alpha}$
\begin{displaymath}
  u_{t} + (\mathscr{L} u -\sum_{i=1}^n P_iu) =0 , \quad t>0.
\end{displaymath}

\end{proposition}
\begin{proof}
By assumption,  for all negative enough $Re(\lambda)$ and $z_0 \in  Z$
we have
\begin{displaymath}
(\mathscr{L} - \lambda)^{-1} z_0 = \int_{0}^{\infty} e^{\lambda t}
S(t)z_0 \, dt.
\end{displaymath}

By consistency, for $Re(\lambda)$  sufficiently negative and  $u_0\in
X^{\beta_i} \subset Z$, we have $G_{\beta_i,\alpha}(\lambda) u_0 =
(L - \lambda)^{-1} u_0 = (\mathscr{L} - \lambda)^{-1} u_0$,
see (\ref{eq:Gbetaigamma0(lambda)}). Then
(\ref{eq:crucial-equality-to-prove-analiticity''}) reads,  for $u_0\in
X^{\alpha}$,
\begin{equation}
  \label{eq:from-{eq:crucial-equality-to-prove-analiticity''}}
(\Lambda -\lambda)^{-1} u_0= (\mathscr{L} - \lambda)^{-1} u_0 +
\sum_{i=1}^n (\mathscr{L} - \lambda)^{-1} P_i (\Lambda
-\lambda)^{-1}u_0 .
\end{equation}
Since $D(\Lambda) = (\Lambda-\lambda)^{-1}(X^\alpha) = D_{X^{\alpha}}(L_{P})$ in
$X^\alpha$  and  $(\mathscr{L} - \lambda)^{-1}(X^\alpha), (\mathscr{L}
- \lambda)^{-1}(X^\beta_i) \subset D_{Z}(\mathscr{L})$ for $i=1,\ldots,n$,
we get $ D_{X^{\alpha}}(L_{P}) \subset  D_{Z}(\mathscr{L})$.

Now for  $u_0\in X^\alpha$ and  $v_0 = (\Lambda -\lambda)^{-1}
u_0 \in D_{X^{\alpha}}(L_{P}) $, from
(\ref{eq:from-{eq:crucial-equality-to-prove-analiticity''}})  we get
\begin{displaymath}
v_0= (\mathscr{L} - \lambda)^{-1} (\Lambda -\lambda) v_0 +
\sum_{i=1}^n (\mathscr{L} - \lambda)^{-1} P_i v_0.
\end{displaymath}
After applying $\mathscr{L} - \lambda$ to both sides of the last
equality above and using that $L_P=\Lambda$, we get $L_P
v_0=\mathscr{L} v_0 -\sum_{i=1}^n P_i v_0$.

The last statement about the equation $u_{t} + L_{P}u=0$
follows now  from \cite[Proposition  2.1.1]{lunardi95:_analy}  since the
perturbed semigroup is analytic in $X^{\alpha}$.
\end{proof}

\section{Linear equation with Morrey potential}
\label{sec:linear-eq-with-Morrey-potential(s)}

In this section we will use the approach from Section \ref{sec:abstract-approach},
to perturb the   semigroup $\{S_{\mu}(t)\}_{t\geq0}$, $0<\mu\leq 1$
associated with the homogeneous problem (\ref{eq:linear-e-eq}) in the Morrey scale.

For this, we first represent the Morrey spaces $\{M^{p,\ell}(\R^{N})\}_{p,\ell}$, $1\leq p\leq \infty$, $0<\ell \leq
N$ (where, for $p=1$,  we can even replace  $M^{1,\ell}(\R^N)$ by
$\mathcal{M}^{\ell}(\R^N)$) in a more convenient way than in
(\ref{eq:Morrey_scale-2m}) as follows: we write
$X^{\gamma}=M^{p,\ell}(\R^{N})$ (or $\mathcal{M}^{\ell}(\R^N)$ if $p=1$) with
\begin{equation}\label{eq:parameters_4_Morrey}
(p,\ell) \longmapsto \gamma= \gamma (p,\ell)  = \left(\frac{1}{p},\frac{\ell}{2m\mu p}\right) \in \mathbb{J}
\end{equation}
with
\begin{equation}  \label{eq:interpretation-of-Morrey_scale-2m}
     \mathbb{J} =  \mathbb{J}_{*}  \cup \{(0,0)\},  \qquad   \mathbb{J}_{*} =\{(\gamma_1,\gamma_2)\in (0,1]\times
    \Big(0,\frac{N}{2m\mu}\Big] \colon \ \frac{\gamma_2}{\gamma_1}
    \leq \frac{N}{2m\mu}\}
\end{equation}
which is a planar triangle with vertices $(0,0)$, $(0,1)$ and
$\left(1,\frac{N}{2m\mu}\right)$, so all points in
$\gamma= (\gamma_1,\gamma_2) \in \mathbb{J}_{*}$, have slopes $0\leq
\frac{\gamma_2}{\gamma_1} =\frac{\ell}{2m\mu p}\leq
  \frac{N}{2m\mu}$. Also notice that for $p=\infty$ all Morrey spaces $M^{\infty,\ell}(\R^{N})$
  are equal to $L^{\infty}(\R^{N})$ and they are mapped by
  (\ref{eq:parameters_4_Morrey}) into
  $(0,0)$. For any $(\gamma_1,\gamma_2) \in
  \mathbb{J}_{*}$ there exist a unique $1\leq p<\infty$
  and $0<\ell\leq N$
  such that $X^{\gamma}=M^{p,\ell}(\R^{N})$.

\subsection{One perturbation}
\label{sec:one-perturbation}

For simplicity in the exposition we will first consider below only one
perturbation given by the multiplication operator by a given potential
as in Section \ref{sec:morrey-potential}, see (\ref{eq:V-in-Morrey}),
\begin{equation}  \label{eq:Morrey_potential}
 V\in M^{p_0,\ell_0}(\R^N) \quad \text{ for  $1\leq p_0\leq \infty$,
   \   $\ell_0\in(0,N]$}
\end{equation}
and so we will show that
\begin{equation}\label{eq:in-Sec-6-linear-e-eq-with-potential}
  \begin{cases}
    u_{t} + A_0^\mu u = V(x)u , &  t>0, \ x\in \R^N , \\
    u(0,x)=u_0(x), & x\in \R^N
  \end{cases}
\end{equation}
defines a semigroup in Morrey spaces possessing suitable smoothing and
analyticity properties.

Abusing of the notations we will denote by $V$ itself the
multiplication operator by $V$ in Morrey spaces.

The next result translates the smoothing of
the semigroup $\{S_{\mu}(t)\}_{t\geq0}$, $0<\mu\leq 1$, in Section
\ref{sec:homogeneous_equation_in_Morrey} and the properties of the multiplication
operator in Lemma \ref{lem:crucial-inequality}  into the parameters in
(\ref{eq:parameters_4_Morrey}), (\ref{eq:interpretation-of-Morrey_scale-2m}).

\begin{lemma}
  \label{lem:abstract_setup_4_Morrey}

  \begin{enumerate}

\item
  The semigroup $\{S_\mu(t)\}_{t\geq0}$ is a
    semigroup in the scale $\{X^{\gamma}\}_{\gamma \in \mathbb{J}}$ as
    in (\ref{eq:interpretation-of-Morrey_scale-2m}) as in Definition
    \ref{def:linear_mappings_on_scale} and moreover
    $\gamma \stackrel{S_\mu(t)}{\leadsto} \tilde{\gamma}$ provided
    that
    \begin{equation} \label{eq:smoothing_morrey_gamma_gamma'}
      \tilde{\gamma}_2\leq \gamma_2 \quad \text{ and } \quad
      \frac{\tilde{\gamma}_2}{\tilde{\gamma}_1} \leq
      \frac{\gamma_2}{\gamma_1}.
    \end{equation}

    Moreover,
    \begin{equation} \label{eq:estimates_Mpl-Mqs_abstract}
      \|S_{\mu}(t) \|_{ \mathcal{L}(X^\gamma,X^{\tilde{\gamma}})} =
      \frac{c}{t^{d(\tilde{\gamma},\gamma)}} \quad t>0 ,
    \end{equation}
    with
    $d(\tilde{\gamma},\gamma) =
    \mathtt{r}(\gamma')-\mathtt{r}(\gamma)\geq 0$ and regularity
    mapping
    \begin{displaymath}
\mathtt{r}(\gamma) =-\gamma_2= -\frac{\ell}{2m\mu p} .
\end{displaymath}

Finally for $\gamma \in
\mathbb{J}$
the semigroup is analytic in $X^{\gamma}$ with the additional restriction that $\gamma_{1}<1$ if $\mu=1$.

  \item
    The assumption (\ref{eq:Morrey_potential}) on $V$
    reads
    \begin{displaymath}
      \text{
        $V\in X^{\gamma^0}$,  \quad  $\gamma^0
        =\left(\frac{1}{p_{0}},\frac{\ell_{0}}{2m\mu p_{0}}\right) \in
        \mathbb{J}$}.
          \end{displaymath}

    \item
      For $\alpha \in \mathbb{J}$, the multiplication
    operator defined by $V$ in Morrey spaces is linear and bounded
    \begin{displaymath}
    V: X^{\alpha} \to
      X^{\beta}, \quad \beta \in \mathbb{J}, \quad \beta = \alpha +
      \gamma^{0}
    \end{displaymath}
    provided
    \begin{displaymath}
      \alpha_{1}+ \gamma_{1}^{0} \leq 1.
    \end{displaymath}

  \end{enumerate}
\end{lemma}
\begin{proof}
Part  (i)  follows from Propositions
\ref{prop:2nd-about-{eq:linear-e-eq}} and
\ref{prop:about-time-continuity}, where analyticity is from Proposition \ref{prop:1st-about-{eq:linear-e-eq}}.

  (ii) This  is by (\ref{eq:parameters_4_Morrey}). Part (iii) is by Lemma
  \ref{lem:crucial-inequality} and since
  $\frac{\alpha_{2}}{\alpha_{1}},
  \frac{\gamma^{0}_{2}}{\gamma^{0}_{1}}\leq \frac{N}{2m\mu}$ we have
  \begin{displaymath}
    \frac{\beta_{2}}{\beta_{1}} =
    \frac{\alpha_{2}+\gamma^{0}_{2}}{\alpha_{1}+ \gamma^{0}_{1}} \leq
    \frac{\frac{N}{2m\mu}(\alpha_{1}+\gamma^{0}_{1})}{\alpha_{1}+
      \gamma^{0}_{1}}  = \frac{N}{2m\mu}.
  \end{displaymath}
  Therefore $\beta = \alpha +\gamma^{0} \in \mathbb{J}$ if and only if
  $\alpha_{1}+ \gamma^{0}_{1} \leq 1$.
\end{proof}

Using the results in Section \ref{sec:abstract-approach}, specifically
Theorem \ref{thm:linear-perturbed-semigroup}, leads us to
the following result.

\begin{theorem}\label{thm:perturbation-by-a-potential}

Let $A_0$ be as in (\ref{eq:operator-A0}), $\mu\in(0,1]$  and
assume $V$ is as in (\ref{eq:V-in-Morrey}), that is, $V\in
M^{p_0,\ell_0}(\R^N)$  for  $1\leq p_0\leq     \infty$, $0<\ell_0\leq
N$ and moreover assume
\begin{displaymath}
\kappa_{0}\mydef     \frac{\ell_{0}}{2m\mu p_{0}}<1.
\end{displaymath}

\begin{enumerate}
\item
  {\bf (The perturbed semigroup)} For
  $1\leq p \leq \infty$ and $0<\ell \leq \ell_0$,
  (\ref{eq:in-Sec-6-linear-e-eq-with-potential}) defines a semigroup
  $\{S_{\mu,V}(t)\}_{t\geq0}$ in $M^{p,\ell}(\R^N)$ such that for
  $u_0\in M^{p,\ell}(\R^N)$, $u(t) \mydef   S_{\mu,V}(t)u_0 $
  satisfies
  \begin{displaymath}
    u(t) = S_{\mu}(t) u_{0} + \int_{0}^{t}  S_{\mu}(t-s) V u(s) \, ds,
    \quad  t>0,
  \end{displaymath}
 \begin{displaymath}
    \lim_{t\to 0^+} \| u(t)- S_{\mu}(t)
    u_0\|_{M^{p,\ell}(\R^N)}= 0 .
  \end{displaymath}
  Also,
  \begin{equation}\label{eq:Mp,ell-estimate-of-S{mu,V}(t)}
    \|S_{\mu,V}(t) \|_{\mathcal{L}(M^{p,\ell}(\R^N))} \leq C e^{\omega
      t}, \quad t\geq 0
  \end{equation}
  for some constants $C$, $\omega$.
  For $p=1$ all the above  remains true if we replace $M^{1,\ell}(\R^N)$ by
  $\mathcal{M}^{\ell}(\R^N)$.

  Moreover, when $p\in[p_0',\infty]$ the exponent in
  (\ref{eq:Mp,ell-estimate-of-S{mu,V}(t)}) can be taken as
  \begin{equation}\label{eq:specifying-omega}
    \omega =c \| V\|_{M^{p_0,\ell_0}(\R^N)}^{\frac{1}{1-\kappa_{0}}},
  \end{equation}
  for some positive constant $c=c(p,\ell)$.

\item
  {\bf (Smoothing properties)} For
$1\leq p \leq \infty $, $0<\ell \leq \ell_0$, if   $1\leq q\leq \infty$ and
$0< s\leq \ell$ satisfy  $\frac{s}{q}\leq \frac{\ell}{p}$ we have,
for some constants $a$, $b$,
\begin{equation}\label{eq:estimate-for-SV(t)}
  \| S_{\mu,V}(t)\|_{\mathcal{L}(M^{p,\ell}(\R^N), M^{q,s}(\R^N))}
  \leq
  \frac{be^{at}}{t^{\frac{1}{2m\mu}(\frac{\ell}{p}-\frac{s}{q})}},
  \quad t>0
\end{equation}
and
\begin{equation}\label{eq:(t,u0)-continuity-of-SV(t)}
  (0,\infty)\times M^{p,\ell}(\R^N) \ni (t,u_0) \to S_{\mu,V}(t) u_0 \in M^{q,s}(\R^{N})
  \ \text{ is continuous}.
\end{equation}
For $p=1$ all these  remain  true if we replace $M^{1,\ell}(\R^N)$ by
$\mathcal{M}^{\ell}(\R^N)$.

Moreover, if $p\in[p_0',\infty]$, or $q\in[p_0',\infty]$, then (\ref{eq:estimate-for-SV(t)})
holds with any $a$ satisfying
\begin{equation}\label{eq:specifying-a}
  a= c \| V\|_{M^{p_0,\ell_0}(\R^N)} ^{\frac{1}{1-\kappa_{0}}}
\end{equation}
for some positive constant $c=c(p,q,\ell,s)$.

\item
  {\bf (Analyticity)}
  For either

  \begin{enumerate}
  \item
    $p_0'\leq p \leq\infty$ and $0<\ell\leq\ell_0$, or

  \item
        $1\leq p\leq p_0$ and $0<\ell\leq \ell_0$ satisfying
    $\frac{\ell_0}{p_0} \leq \frac{\ell}{p}$,
  \end{enumerate}
the semigroup
$\{S_{\mu,V}(t)\}_{t\geq0}$ is analytic in $M^{p,\ell}(\R^N)$ with
sectorial generator with the additional restriction that $p\not=1$
when $\mu=1$.

\item
    {\bf (The perturbed equation)}
For $1< p \leq p_0$ and $0<\ell\leq\ell_0$ satisfying
$\frac{\ell_0}{p_0}\leq  \frac{\ell}{p}$,
we have that $u=S_{\mu,V}(\cdot) u_0$ with $u_0\in M^{p,\ell}(\R^N)$
satisfies, for $t>0$,
\begin{displaymath}
u_{t} + A_0^\mu u  = Vu  \quad \text{in
$M^{p,\ell}(\R^N)$}.
\end{displaymath}

\end{enumerate}

\end{theorem}
\begin{proof}
(i) ({\bf The perturbed semigroup}) We show that we can use Theorem
\ref{thm:linear-perturbed-semigroup} in the setting of Lemma
\ref{lem:abstract_setup_4_Morrey}, whose notations  we use all the
time below.

For convenience, if $\gamma^0\not=(0,0)$, in what follows we will denote
$\mathbb{J}_{\gamma^{0}}$ the set of $\gamma \in \mathbb{J}$ that
satisfy
\begin{equation} \label{eq:bound-on-gamma2overgamma1}
\frac{\gamma_{2}}{\gamma_{1}} \leq \frac{\gamma^{0}_{2}}{
    \gamma^{0}_{1}}  =  \frac{\ell_{0}}{2m\mu}
\end{equation}
(that is, a triangle   of elements in $\mathbb{J}$ with slopes less or equal that of $\gamma^0$),
while  if
$\gamma^{0}=(0,0)$ then we will denote $\mathbb{J}_{(0,0)}=
\mathbb{J}$.

Then, in the following steps we are going to prove that for all
$\gamma \in \mathbb{J}_{\gamma^{0}}$ we can apply Theorem
\ref{thm:linear-perturbed-semigroup}.

\noindent {\bf Step 1.}
Assume first  $\gamma^0\not=(0,0)$.
For $\alpha, \beta\in \mathbb{J}$ with $\beta= \alpha+ \gamma^{0}$, the multiplication operator belongs to the class of
admissible perturbations $\mathscr{P}_{\beta,R}$, as in Definition
\ref{def:perturbations_in_scale}, if and only if
$\alpha_1\leq 1-\gamma^0_1$ and
\begin{displaymath}
  \gamma_2^0=\kappa_{0}\mydef  \frac{\ell_0}{2m\mu p_0}< 1 , \quad
  \|V\|_{M^{p_{0}\ell_{0}}(\R^{N})} \leq R, \qquad
  \frac{\alpha_{2}}{\alpha_{1}} \leq \frac{\gamma^{0}_{2}}{
    \gamma^{0}_{1}} ,
\end{displaymath}
 that is, $\alpha \in \mathbb{J}_{\gamma^{0}}$.

To see this, observe that after (iii) in  Lemma
\ref{lem:abstract_setup_4_Morrey} we must have $\alpha_1\leq
1-\gamma^0_1$ and   in Definition \ref{def:perturbations_in_scale} we
require $0\leq
  d(\alpha,\beta)=\mathtt{r}(\alpha)-\mathtt{r}(\beta)
  =\gamma^{0}_{2} =  \frac{\ell_0}{2m\mu p_0}<1$.
  On the other hand, from
  (\ref{eq:smoothing_morrey_gamma_gamma'}) the condition   $\beta
  \stackrel{_{S_{\mu}(t)}}{\leadsto} \alpha$ reads
\begin{displaymath}
    \alpha_2\leq \beta_2 =\alpha_{2}+ \gamma^{0}_{2}\quad \text{ and } \quad
    \frac{\alpha_2}{\alpha_1} \leq \frac{\beta_2}{\beta_1} =
    \frac{\alpha_{2}+ \gamma^{0}_{2}}{\alpha_{1}+ \gamma^{0}_{1}} .
\end{displaymath}
The former is always satisfied and the latter one is equivalent to  $
\frac{\alpha_{2}}{\alpha_{1}} \leq \frac{\gamma^{0}_{2}}{
  \gamma^{0}_{1}}$,  that is, $\alpha \in \mathbb{J}_{\gamma^{0}}$.

\medskip
\noindent {\bf Step 2.}
With the restrictions in Step 1, we
have that  $\gamma\in \mathcal{E}_\alpha$, as in Theorem
\ref{thm:existence-linear} if and only if $\gamma\in\mathbb{J}$ and
\begin{displaymath}
  \alpha_{2} \leq \gamma_{2} < \alpha_{2}+1, \qquad \frac{\alpha_{2}}{\alpha_{1}} \leq \frac{\gamma_{2}}{
  \gamma_{1}}.
\end{displaymath}
To see this, note that the conditions
$\gamma\stackrel{S_\mu(t)}{\leadsto} \alpha$, with
(\ref{eq:smoothing_morrey_gamma_gamma'}),  and $\mathtt{r}(\gamma)
\in(\mathtt{r}(\alpha)-1, \mathtt{r}(\alpha)]$ give the restrictions
above.

\medskip
\noindent {\bf Step 3.}
With the restrictions in Step 1, we
have that  $\gamma' \in \mathcal{R}_\beta$, as in
Theorem \ref{thm:Xgamma-Xgamma'-estimates}, see
(\ref{eq:regularity_set_beta}),  if and only if $\gamma' \in
\mathbb{J}$ and
\begin{displaymath}
    \gamma'_2\leq \beta_2 =\alpha_{2}+ \gamma^{0}_{2} \quad \text{ and } \quad
    \frac{\gamma'_2}{\gamma'_1} \leq \frac{\beta_2}{\beta_1} =
    \frac{\alpha_{2}+ \gamma^{0}_{2}}{\alpha_{1}+ \gamma^{0}_{1}}  \quad \text{ and } \quad
    \gamma'_{2} > \alpha_{2} - j_{0}
\end{displaymath}
with $j_{0} = 1-\gamma^{0}_{2}>0$.  In particular $\gamma' \in
\mathbb{J}_{\gamma^{0}}$, that is $\mathcal{R}_\beta \subset
\mathbb{J}_{\gamma^{0}}$.

Actually,  the first two conditions
stem  from $\beta \stackrel{S_\mu(t)}{\leadsto} \gamma'$  and
Lemma \ref{lem:abstract_setup_4_Morrey} and the third one
from the condition $\mathtt{r}(\gamma') \in[\mathtt{r}(\beta), \mathtt{r}(\beta)+1)$.

Also observe that since from Step 1 we have
$\frac{\alpha_2}{\alpha_1}\leq \frac{\gamma_2^0}{\gamma_1^0}$ then
$\frac{\alpha_2 +\gamma_2^0}{\alpha_1
  +\gamma_1^0}=\frac{\gamma_2^0}{\gamma_1^0}\frac{\frac{\alpha_2}{\gamma_2^0}+1}{\frac{\alpha_1}{\gamma_1^0}+1}
\leq \frac{\gamma_2^0}{\gamma_1^0}$. Therefore $\beta\in \mathbb{J}_{\gamma^{0}}$, and in particular we have
$\frac{\gamma'_2}{\gamma'_1} \leq \frac{\gamma_2^0}{\gamma_1^0}$ and
then $\gamma' \in \mathbb{J}_{\gamma^{0}}$.

\medskip
\noindent {\bf Step 4.}
With the restrictions in Step 1, we
have that $\gamma\in \Sigma_{\alpha,\beta} =\mathcal{E}_\alpha\cap\mathcal{R}_{\beta}$, as in Theorem
\ref{thm:linear-perturbed-semigroup},  if and only if
\begin{equation}
  \label{eq:we-altogether-need}
  \begin{cases}
    \alpha_1+\gamma^0_1\leq 1
\\
\alpha_2\leq \gamma_2 \leq \alpha_2 +\gamma_2^0
    \\
    \frac{\alpha_2}{\alpha_1} \leq \frac{\gamma_2}{\gamma_1} \leq
    \frac{\alpha_2 +\gamma_2^0}{\alpha_1 +\gamma_1^0} .
  \end{cases}
\end{equation}
In particular $\gamma \in \mathbb{J}_{\gamma^{0}}$, that is,
$\Sigma_{\alpha,\beta} \subset \mathbb{J}_{\gamma^{0}}$.

This   follows immediately from Steps 1 to 3.

\medskip
\noindent {\bf Step 5.}
Given $\gamma \in \mathbb{J}_{\gamma^{0}}$,  we show that we can produce $\alpha \in
\mathbb{J}$ and $\beta= \alpha +\gamma^{0}\in \mathbb{J}$ such that
$\gamma \in \Sigma_{\alpha, \beta}$ as in
(\ref{eq:we-altogether-need}) (so such $\alpha$, $\beta$ will belong to $\mathbb{J}_{\gamma^0}$). Hence for such $\gamma$, Theorem
\ref{thm:linear-perturbed-semigroup} applies.

To see this, assume first $\gamma \not = (0,0)$ and  we  take $\alpha$ with the same slope
than $\gamma$, that is $ \frac{\alpha_2}{\alpha_1} =
\frac{\gamma_2}{\gamma_1}$ so the third condition in
(\ref{eq:we-altogether-need}) is met.

Now we can furthermore take $\alpha_{2}= \gamma_{2}$ (so the second condition in
(\ref{eq:we-altogether-need}) is met), and therefore
$\alpha=\gamma$, provided $\gamma_{1}\leq 1 -\gamma^{0}_{1}$, which
comes from the first condition in (\ref{eq:we-altogether-need}).

If, on the other hand, $1 -\gamma^{0}_{1}<\gamma_{1}$, then we must
choose $\alpha_{2}= \frac{\gamma_2}{\gamma_1}\alpha_{1}$ with
$0\leq \alpha_{1} \leq 1 -\gamma^{0}_{1}$ from the  first condition in
(\ref{eq:we-altogether-need}), such that the second  condition in
(\ref{eq:we-altogether-need}) is satisfied, that is
$\frac{\gamma_2}{\gamma_1}\alpha_{1}\leq \gamma_{2} \leq
\frac{\gamma_2}{\gamma_1}\alpha_{1} + \gamma^{0}_{2}$.

Then we claim that choosing $\alpha_{1}= 1 -\gamma^{0}_{1}$ achieves
that. For this notice that, since $\gamma$ satisfies
(\ref{eq:bound-on-gamma2overgamma1}),
\begin{displaymath}
   \frac{\gamma_2}{\gamma_1}(1 -\gamma^{0}_{1})+ \gamma^{0}_{2}  =  \frac{\gamma_2}{\gamma_1}(1 -\gamma^{0}_{1})+
\gamma^{0}_{1}  \frac{\gamma^{0}_2}{\gamma^{0}_1} \geq  \frac{\gamma_2}{\gamma_1}(1 -\gamma^{0}_{1})+
\gamma^{0}_{1}  \frac{\gamma_2}{\gamma_1}=  \frac{\gamma_2}{\gamma_1}
\end{displaymath}
and $\frac{\gamma_2}{\gamma_1}\geq \gamma_{2}$ because $\gamma_{1}\leq
1$. On the other hand, $\frac{\gamma_2}{\gamma_1}\alpha_{1}\leq
\gamma_{2}$ is satisfied because $\alpha_{1} = 1-\gamma^{0}_{1} <
\gamma_{1}$.

Finally, if $\gamma = (0,0)$ we
take $\alpha =\gamma = (0,0)$ and (\ref{eq:we-altogether-need})  is
satisfied.

\medskip
\noindent {\bf Step 6.} Assume now $\gamma^{0}=(0,0)$, that is $V\in
L^{\infty}(\R^{N})$.
Then in Step 1
above we have $\alpha \in \mathbb{J}$ and $\beta = \alpha$, while in
Steps 2 to 4   we just get $\Sigma_{\alpha, \beta}=\{\alpha\}$. In Step 5 for
$\gamma \in \mathbb{J}$, we take $\alpha=
\gamma$. Then we can apply Theorem
\ref{thm:linear-perturbed-semigroup} as well in this case.

\medskip
\noindent {\bf Step 7.}
Finally, as seen in Step 5 above, if $\gamma_{1}\leq 1
-\gamma^{0}_{1}$ then we can take $\alpha =\gamma$ and then the exponential estimates on the perturbed semigroup with
exponent (\ref{eq:specifying-omega}),  follow
from Proposition \ref{prop:exponential-bound-for-SV(t)-in-Xalpha}
since from (\ref{eq:estimates_Mpl-Mqs_abstract}) in Lemma
\ref{lem:abstract_setup_4_Morrey} we have
(\ref{eq:estimates_to_perform_perturbation}) with $a=0$ and then we
get (\ref{eq:esimate_after_perturbation}) with $a=0$. In terms of the
original parameters of the Morrey scale, the case $\gamma_{1}\leq 1
-\gamma^{0}_{1}$  corresponds to $p\geq p_{0}'$.

\medskip
\noindent (ii) ({\bf Smoothing}) We will use  Theorem
  \ref{thm:Xgamma-Xgamma'-estimates}, Corollary
  \ref{cor:continuity-of-SP(t)u0-with-respect-to-a-pair-of-argument},
  or the second part of Theorem \ref{thm:linear-perturbed-semigroup}
to get the smoothing.

\medskip
\noindent {\bf Step 8.}
Given  $\alpha \in
\mathbb{J}$ and $\beta= \alpha +\gamma^{0}\in \mathbb{J}$ to apply Theorem
  \ref{thm:Xgamma-Xgamma'-estimates}, Corollary
  \ref{cor:continuity-of-SP(t)u0-with-respect-to-a-pair-of-argument}
  (or the second part of Theorem \ref{thm:linear-perturbed-semigroup})
  we need
$\gamma\in \mathcal{E}_\alpha$ (or $\gamma\in\Sigma_{\alpha,\beta}$) and  $\gamma'\in
\mathcal{R}_{\beta}$ such that
$\gamma\stackrel{_{S(t)}}{\leadsto} \gamma'$. Hence, by Steps 2 and 3
above and  Lemma
\ref{lem:abstract_setup_4_Morrey}, we need
\begin{displaymath}
    \gamma'_2\leq \gamma_2 \quad \text{ and } \quad
    \frac{\gamma'_2}{\gamma'_1} \leq
    \frac{\gamma_2}{\gamma_1}
\end{displaymath}
and then we would get
$\gamma\stackrel{_{S_{\mu, V}(t)}}{\leadsto} \gamma'$, provided
(\ref{eq:we-altogether-need}) and
\begin{equation}\label{eq:gamma'_4_regularization_morrey}
  \begin{split}
  0&\leq \gamma'_2
  \\
  \alpha_{2} - j_{0}&<   \gamma'_2 \leq \gamma_2 \leq \alpha_{2}+ \gamma^{0}_{2} \\
     \frac{\gamma'_2}{\gamma'_1} & \leq
    \frac{\gamma_2}{\gamma_1} \leq  \frac{\alpha_{2}+
      \gamma^{0}_{2}}{\alpha_{1}+ \gamma^{0}_{1}} .
  \end{split}
\end{equation}

\medskip
\noindent {\bf Step 9.}
Given $\gamma \in \mathbb{J}_{\gamma^0}$ we want to produce  $\alpha \in
\mathbb{J}$ (and $\beta= \alpha +\gamma^{0}\in \mathbb{J}$) such that
$\gamma \in \Sigma_{\alpha, \beta}$ as in
(\ref{eq:we-altogether-need}) and such that the set of $\gamma'$ in
(\ref{eq:gamma'_4_regularization_morrey}) is as large as
possible. We remark that such $\alpha$, $\beta$, $\gamma$, $\gamma'$ will belong to $\mathbb{J}_{\gamma^0}$.

As in Step 3,  we  take $\alpha$ with the same slope
than $\gamma$, that is $ \frac{\alpha_2}{\alpha_1} =
\frac{\gamma_2}{\gamma_1}$,  and then we want to take the smallest
$\alpha_{2}$ possible in (\ref{eq:we-altogether-need})  (and hence in
(\ref{eq:gamma'_4_regularization_morrey})). Since
$\alpha_{2}=\frac{\gamma_2}{\gamma_1}\alpha_{1}$,  we minimize
$\alpha_{1}$ such that
\begin{displaymath}
\alpha_1 \leq 1 -   \gamma^0_1, \qquad
\frac{\gamma_2}{\gamma_1}\alpha_{1}  \leq \gamma_2 \leq
    \frac{\gamma_2}{\gamma_1}\alpha_{1} + \gamma^{0}_{2} = \alpha_2 + \gamma^{0}_{2}
\end{displaymath}
that is, we can take $\alpha_2$ up to $\alpha_{2}= \max\{0,\gamma_{2} -  \gamma^{0}_{2}\}$, and so
(\ref{eq:we-altogether-need}) is satisfied. Hence, since $j_{0} =
1-\gamma^{0}_{2}>0$ and $\max\{0,\gamma_{2} -  \gamma^{0}_{2}\}-j_0=\max\{-j_0,\gamma_{2} -  1\}$, the set of $\gamma'$
in (\ref{eq:gamma'_4_regularization_morrey}) is given by
\begin{equation}\label{eq:regularization_gamma_gamma'_Morrey}
0\leq \gamma'_2,  \qquad  \gamma_{2} - 1 <  \gamma'_2 \leq \gamma_2,  \qquad
     \frac{\gamma'_2}{\gamma'_1}  \leq
    \frac{\gamma_2}{\gamma_1} .
\end{equation}

Since for this $\gamma'$ we have $\gamma\stackrel{_{S_{\mu,
      V}(t)}}{\leadsto} \gamma'$, then Lemma
\ref{lem:estimates4semigroups_in_scales} and part (i) of the theorem
give the estimates (\ref{eq:estimate-for-SV(t)}) and
(\ref{eq:specifying-a}), whereas (\ref{eq:(t,u0)-continuity-of-SV(t)}) is from Corollary \ref{cor:continuity-of-SP(t)u0-with-respect-to-a-pair-of-argument}.

\medskip
\noindent {\bf Step 10.}
Now we prove that for any  $\gamma,\gamma' \in \mathbb{J}_{\gamma^0}$
such that
\begin{equation} \label{eq:smoothing_morrey_gamma_gamma'_4_perturbed}
   \gamma'_2 \leq \gamma_2, \qquad
     \frac{\gamma'_2}{\gamma'_1}  \leq
    \frac{\gamma_2}{\gamma_1}
\end{equation}
we have $\gamma\stackrel{_{S_{\mu, V}(t)}}{\leadsto} \gamma'$ with
  (\ref{eq:estimate-for-SV(t)}), (\ref{eq:(t,u0)-continuity-of-SV(t)}) and
(\ref{eq:specifying-a}).  For this we construct a finite sequence
$\gamma^{j} \in \mathbb{J}_{\gamma^{0}}$ for $j=1,\ldots, M$, such that $\gamma^{1}=
\gamma$, $\gamma^{M} = \gamma'$ and, given $j=1,\ldots, M-1$,  $\gamma^{j} \stackrel{_{S_{\mu,
      V}(t)}}{\leadsto} \gamma^{j+1}$ as in (\ref{eq:regularization_gamma_gamma'_Morrey}). Actually,  we choose
\begin{displaymath}
  \gamma^{j}_{2} - \frac{1}{2} \leq  \gamma^{j+1}_2 \leq
  \gamma^{j}_2, \qquad
     \frac{\gamma^{j+1}_2}{\gamma^{j+1}_1}  \leq
    \frac{\gamma^{j}_2}{\gamma^{j}_1}
\end{displaymath}
so (\ref{eq:regularization_gamma_gamma'_Morrey}) holds at each step of
the iteration. To see that this is possible, observe that if we take
points  $\tilde{\gamma}= \theta \gamma' + (1-\theta) \gamma$ with
$\theta \in (0,1)$ then
\begin{displaymath}
\gamma'_{2}< \tilde{\gamma}_{2}
   < \gamma_{2} , \qquad     \frac{\gamma'_2}{\gamma'_1} \leq     \frac{\tilde{\gamma}_2}{\tilde{\gamma}_1}  \leq
    \frac{\gamma_2}{\gamma_1} , \qquad \theta \in (0,1).
\end{displaymath}
Therefore, we chose $\theta_{j}\in (0,1)$ iteratively such that $
\gamma^{j}_{2} - \frac{1}{2} \leq   \gamma^{j+1}_2 \leq \gamma^{j}_2$
until we can take $\gamma^{M}_2 = \gamma'_{2}$ and thus $\gamma^{M} = \gamma'$.

To prove (\ref{eq:estimate-for-SV(t)}) and (\ref{eq:(t,u0)-continuity-of-SV(t)})
between $X^{\gamma}$ and $X^{\gamma'}$, observe that since they hold for each pair
$(\gamma^{j}, \gamma^{j+1})$, $j=1,\ldots, M-1$ (even if with
constants and exponents depending on $j$),  we use the semigroup
property $S_{\mu, V}(t) = S_{\mu, V}(\frac{t}{M}) \circ \cdots \circ
S_{\mu, V}(\frac{t}{M})$ and we get (\ref{eq:estimate-for-SV(t)}) and
(\ref{eq:(t,u0)-continuity-of-SV(t)}) for $(\gamma, \gamma')$.

Then using Lemma
\ref{lem:estimates4semigroups_in_scales} and (\ref{eq:estimate-for-SV(t)}) and part (i) of the theorem we conclude (\ref{eq:specifying-a}).

\medskip
\noindent (iii) {\bf (Analyticity)}
From part (i) in Lemma \ref{lem:abstract_setup_4_Morrey}  the unperturbed
semigroup is analytic in $X^{\gamma}$ for $\gamma \in
\mathbb{J}$ with the additional restriction that $\gamma_{1}<1$ if $\mu=1$ that we take into account, without mentioning it further.
For the perturbed semigroup, we now show we can apply either
Theorem \ref{thm:Xalpha-analiticity-of-linearly-perturbed-semigroup}
or Theorem
\ref{thm:Xbeta-analiticity-of-linearly-perturbed-semigroup}.

For $\gamma \in  \mathbb{J}_{\gamma^0}$, Theorem \ref{thm:Xalpha-analiticity-of-linearly-perturbed-semigroup}
applies provided that in the proof of part (i) above we can take
$\alpha =\gamma$. This can be done when
$\gamma_1\leq 1-\gamma_1^0$  (see Step 5 in the proof of (i)
above) which, from (\ref{eq:parameters_4_Morrey}) corresponds to
$p_0'\leq p$ and  $0<\ell\leq \ell_0$. This gives part (a) of the
statement.

On the other hand, Theorem
\ref{thm:Xbeta-analiticity-of-linearly-perturbed-semigroup} applies
provided that in the proof of part (i) above we can take $\beta=
\alpha + \gamma^{0} =
\gamma$, that is $\gamma\in \Big(\gamma^0+ \mathbb{J}_{\gamma^0}\Big)\cap
\mathbb{J}_{\gamma^0}$.  The intersection of these two triangles is
the triangle  $\gamma_{1}^{0} \leq \gamma_{1}\leq
  1$,   $\gamma_{2}^{0} \leq \gamma_{2}\leq
  \frac{N}{2m\mu}$,  and $\frac{\gamma_{2}}{\gamma_{1}} \leq
  \frac{\gamma_{2}^{0}}{\gamma_{1}^{0}}$ and  from (\ref{eq:parameters_4_Morrey}) and
  (\ref{eq:interpretation-of-Morrey_scale-2m}) this corresponds to

  \begin{displaymath}
    \frac{1}{p_0}\leq  \frac{1}{p} \leq 1, \qquad
    \frac{\ell_{0}}{2m\mu p_{0} }\leq \frac{\ell}{2m\mu p}
    \leq
  \frac{N}{2m\mu},   \qquad \frac{\ell}{2m\mu }
    \leq \frac{\ell_{0}}{2m\mu }
  \end{displaymath}
  i.e.
  \begin{displaymath}
   1\leq  p\leq p_{0},  \qquad
   \frac{\ell_{0}}{p_{0} } \leq \frac{\ell}{ p} , \qquad \ell \leq
   \ell_{0} .
  \end{displaymath}
This completes the proof of part (b) in the statement.

\medskip
\noindent (iv) {\bf (The perturbed equation)}
As we are in case (b) of part (iii) above, we have that Proposition
\ref{prop:LP-in-Xbeta} applies and we get the result, because from
Remark \ref{rem:related-to-{prop:1st-about-{eq:linear-e-eq}}}
the sectorial generator of the
unperturbed semigroup is $-A_0^\mu $.
\end{proof}

\begin{remark}
  \label{rem:summary_1perturbation_morrey}

  As a summary of the results above, observe that  for
(\ref{eq:in-Sec-6-linear-e-eq-with-potential}) we start with an
unperturbed semigroup   $\{S_{\mu}(t)\}_{t\geq0}$ (that is $V=0$) acting on the spaces in the  triangle
$\mathbb{J}$ in (\ref{eq:interpretation-of-Morrey_scale-2m}) with the smoothing in
(\ref{eq:smoothing_morrey_gamma_gamma'}).

Then we add a perturbation $V\in X^{\gamma^{0}}$ with
$\kappa_{0}\mydef     \frac{\ell_{0}}{2m\mu p_{0}}<1$ and we end up
with a perturbed semigroup   $\{S_{\mu,V}(t)\}_{t\geq0}$ acting on the spaces in the  triangle
$\mathbb{J}_{\gamma^{0}} \subset \mathbb{J}$ as in
(\ref{eq:bound-on-gamma2overgamma1}),
with the smoothing in
(\ref{eq:smoothing_morrey_gamma_gamma'_4_perturbed}).

Observe that (\ref{eq:smoothing_morrey_gamma_gamma'}) are
(\ref{eq:smoothing_morrey_gamma_gamma'_4_perturbed}) are identical
although the latter in the smaller triangle
$\mathbb{J}_{\gamma^{0}}$.

\end{remark}

As for the   continuous dependence on perturbations we have the
following result.

\begin{theorem}
  \label{thr:morrey_continuous_dependence_1_perturbation}

Assume $V,\tilde V\in M^{p_0,\ell_0}(\R^N)$ with
\begin{displaymath}
\kappa_{0}\mydef     \frac{\ell_{0}}{2m\mu p_{0}}<1
\end{displaymath}
and  $u_0, \tilde{u}_0\in M^{p,\ell}(\R^N)$ with $1\leq p\leq \infty$
and $0<\ell\leq \ell_0$. Assume also
  \begin{displaymath}
  \|V\|_{\mathcal{L}(M^{p_0,\ell_0}(\R^N))}, \|\tilde{V}\|_{\mathcal{L}(M^{p_0,\ell_0}(\R^N))}\leq R
  \end{displaymath}
  and
  \begin{displaymath}
\|u_0\|_{M^{p,\ell}(\R^N)}, \|\tilde{u}_0\|_{M^{p,\ell}(\R^N)}\leq
\mathscr{R} .
\end{displaymath}

Then for $1\leq q\leq \infty$  and $0<s\leq \ell$ such that
$\frac{s}{q}\leq \frac{\ell}{p}$ and  $T>0$, there exists
$C_0$, $C_1$ depending  on $p_0,\ell_0, p, \ell, q, s, R$ and
$T$ and  $C_0$ depending also  on $\mathscr{R}$,  such that
for  $t\in (0,T]$ we have
\begin{displaymath}
\|S_{\mu,V}(t)u_{0} - S_{\mu,\tilde{V}}(t)\tilde{u}_0\|_{M^{q,s}(\R^N)}
\leq \frac{C_{0}}{t^{\frac{1}{2m\mu}(\frac{\ell}{p}-\frac{s}{q})}}
\left( \|u_0-\tilde{u}_0 \|_{M^{p,\ell}(\R^N)}  + \|V-\tilde V\|_{\mathcal{L}(M^{p_0,\ell_0}(\R^N))} \right)
\end{displaymath}
and
\begin{displaymath}
\|S_{\mu,V}(t) - S_{\mu,\tilde{V}}(t)\|_{{\mathcal L}(M^{p,\ell}(\R^N), M^{q,s}(\R^N))}
\leq \frac{C_1}{t^{\frac{1}{2m\mu}(\frac{\ell}{p}-\frac{s}{q})}}
\|V-\tilde V\|_{\mathcal{L}(M^{p_0,\ell_0}(\R^N))} .
\end{displaymath}

\end{theorem}
\begin{proof}
We will use Theorem \ref{thm:Lipschitz-properties-of-S_P(t)} in the setting of Lemma
\ref{lem:abstract_setup_4_Morrey} and the notations in the proof of
Theorem \ref{thm:perturbation-by-a-potential}.

\smallskip

\noindent {\bf Step 1.}
For $\gamma\in \mathbb{J}_{\gamma^0}$, we  choose $\alpha$ satisfying
$\frac{\alpha_2}{\alpha_1}=\frac{\gamma_2}{\gamma_1}$ and
(\ref{eq:we-altogether-need}) and
$\beta=\alpha+\gamma^0$ (as in Step 5 of the proof of Theorem \ref{thm:perturbation-by-a-potential}). Hence we have $\gamma\in
\Sigma_{\alpha,\beta}\subset \mathcal{E}_\alpha$ and the perturbations
$P=V, \tilde{P}=\tilde{V}\in \mathscr{P}_{\beta,R}$.

Then for $\gamma'\in \mathcal{R}_\beta$, that is for $\gamma'$ as in  (\ref{eq:regularization_gamma_gamma'_Morrey}),
 from  Theorem
\ref{thm:Lipschitz-properties-of-S_P(t)} we have the estimates, for
$t\in (0,T]$,
\begin{equation}\label{eq:(*)-in-cd}
 \|S_{\mu,V}(t)u_0 - S_{\mu,\tilde{V}}(t)\tilde{u}_0\|_{X^{\gamma'}}
\leq \frac{M_0}{t^{\mathtt{r}(\gamma')-\mathtt{r}(\gamma)}} \left( \|u_0-\tilde{u}_0 \|_{X^\gamma} +\|V-\tilde V\|_{\mathcal{L}(X^{\gamma_0})}\right)
\end{equation}
and
\begin{equation}\label{eq:(**)-in-cd}
\|S_{\mu,V}(t) - S_{\mu,\tilde{V}}(t)\|_{{\mathcal L}(X^{\gamma}, X^{\gamma'})}
\leq \frac{M_1}{t^{\mathtt{r}(\gamma')-\mathtt{r}(\gamma)}} \|V-\tilde
V\|_{\mathcal{L}(X^{\gamma_0})} .
\end{equation}

\smallskip
\noindent {\bf Step 2.}
Now we show that
(\ref{eq:(*)-in-cd}) and (\ref{eq:(**)-in-cd})
hold for $\gamma'$ as in
(\ref{eq:smoothing_morrey_gamma_gamma'_4_perturbed}),
that is  satisfying $\gamma'_2 \leq \gamma_2$ and
$\frac{\gamma'_2}{\gamma'_1}  \leq \frac{\gamma_2}{\gamma_1}$. With
this and (\ref{eq:parameters_4_Morrey}),
(\ref{eq:interpretation-of-Morrey_scale-2m}), the theorem is proved.

So we consider $\gamma\in \mathbb{J}_{\gamma^0}$ and $\gamma'$
satisfying $\gamma'_2 \leq \gamma_2$ and $\frac{\gamma'_2}{\gamma'_1}
\leq \frac{\gamma_2}{\gamma_1}$.
As in  Step 10 in the proof of  part (ii) in Theorem \ref{thm:perturbation-by-a-potential},
we construct points
$\gamma^{j} \in \mathbb{J}_{\gamma^{0}}$ for $j=1,\ldots, M$, such that $\gamma^{1}=
\gamma$, $\gamma^{M} = \gamma'$ and $\gamma^{j}$ satisfy for $j=1,\ldots,M-1$ both $  \gamma^{j}_{2} - \frac{1}{2} \leq  \gamma^{j+1}_2 \leq
  \gamma^{j}_2$ and
     $\frac{\gamma^{j+1}_2}{\gamma^{j+1}_1}  \leq
    \frac{\gamma^{j}_2}{\gamma^{j}_1}$.

Due to Steps 1 and 2 above,
(\ref{eq:(*)-in-cd}) holds with $\gamma$ replaced by $\gamma^j$ and $\gamma'$ replaced by
$\gamma^{j+1}$, that is, for  $j=1,\ldots, M-1$ and  $t\in (0,T]$,
\begin{equation}\label{eq:(***)-in-cd}
 \|S_{\mu,V}(t)u_0 - S_{\mu,\tilde{V}}(t)\tilde{u}_0\|_{X^{\gamma^{j+1}}}
\leq \frac{c}{t^{\mathtt{r}(\gamma^{j+1})-\mathtt{r}(\gamma^j)}} \left( \|u_0-\tilde{u}_0 \|_{X^{\gamma^j}} +\|V-\tilde V\|_{\mathcal{L}(X^{\gamma^{0}})}\right)
\end{equation}
provided $\|u_0\|_{X^{\gamma^j}}, \|\tilde{u}_0 \|_{X^{\gamma^j}} \leq \mathscr{R}$.

Using (\ref{eq:(***)-in-cd}) and the semigroup property we will show
below that for  $j=1,\ldots, M-1$
\begin{equation}\label{eq:(****)-in-cd}
 \|S_{\mu,V}(t)u_0 - S_{\mu,\tilde{V}}(t)\tilde{u}_0\|_{X^{\gamma^{j+1}}}
\leq  \frac{c}{t^{\mathtt{r}(\gamma^{j+1})-\mathtt{r}(\gamma^1)}}
\left( \|u_0 - \tilde{u}_0\|_{X^{\gamma^1}} + \|V-\tilde V\|_{\mathcal{L}(X^{\gamma^{0}})}\right)
\end{equation}
for some constant $c$, which for $j=M-1$ gives the estimate in the statement.

First, (\ref{eq:(****)-in-cd}) for  $j=1$ follows from
(\ref{eq:(***)-in-cd}) with $j=1$.
In order to establish (\ref{eq:(****)-in-cd}) for $j=2$ we first observe that we have
\begin{equation}\label{eq:(*****)-in-cd}
\begin{split}
&\|S_{\mu,V}(t)u_0  - S_{\mu,\tilde{V}} (t)\tilde{u}_0\|_{X^{\gamma^3}}
\\
&
=
\frac{1}{\big(\frac{t}{2}\big)^{\mathtt{r}(\gamma^2)-\mathtt{r}(\gamma^1)}}
\Big\|
S_{\mu,V}\big(\frac{t}{2}\big)
\Big[\big(\frac{t}{2}\big)^{\mathtt{r}(\gamma^2)-\mathtt{r}(\gamma^1)}
S_{\mu,V}\big(\frac{t}{2}\big)u_0\Big]
- S_{\mu,\tilde{V}}\big(\frac{t}{2}\big)
\Big[\big(\frac{t}{2}\big)^{\mathtt{r}(\gamma^2)-\mathtt{r}(\gamma^1)}
S_{\mu,\tilde{V}}\big(\frac{t}{2}\big)\tilde{u}_0\Big]
\Big\|_{X^{\gamma^3}} .
\end{split}
\end{equation}
Using (\ref{eq:(****)-in-cd}) with $j=1$, $\tilde{u}_0=0$ and
$\tilde{V}=0$, and then with $u_0=0$ and $V=0$, we get  that
\begin{displaymath}
t^{\mathtt{r}(\gamma^2)-\mathtt{r}(\gamma^1)}
\|S_{\mu,V}(t)u_0\|_{X^{\gamma^{2}}}, \quad
t^{\mathtt{r}(\gamma^2)-\mathtt{r}(\gamma^1)}
\|S_{\mu,\tilde{V}}(t)\tilde{u}_0\|_{X^{\gamma^{2}}}
\end{displaymath}
are  bounded uniformly in  $t\in (0,T]$ provided
$\|u_0\|_{X^{\gamma^1}}, \|\tilde{u}_0\|_{X^{\gamma^1}} \leq
\mathscr{R}$ and $\|V\|_{\mathcal{L}(X^{\gamma^{0}})} ,
\|\tilde{V}\|_{\mathcal{L}(X^{\gamma^{0}})} \leq R$.
Hence we can use (\ref{eq:(***)-in-cd}) with $j=2$ to estimate the right hand side of (\ref{eq:(*****)-in-cd}) as
\begin{displaymath}
\begin{split}
\| &S_{\mu,V}(t) u_0 - S_{\mu,\tilde{V}}(t)\tilde{u}_0\|_{X^{\gamma^{3}}}
 \\
 &
 \leq
\frac{1}{\big(\frac{t}{2}\big)^{\mathtt{r}(\gamma^2)-\mathtt{r}(\gamma^1)}}
\Big[
 \frac{c}{\big(\frac{t}{2}\big)^{\mathtt{r}(\gamma^{3})-\mathtt{r}(\gamma^2)}} \Big(
 \Big\|\big(\frac{t}{2}\big)^{\mathtt{r}(\gamma^2)-\mathtt{r}(\gamma^1)} S_{\mu,V}\big(\frac{t}{2}\big)u_0
  -\big(\frac{t}{2}\big)^{\mathtt{r}(\gamma^2)-\mathtt{r}(\gamma^1)}
 S_{\mu,\tilde{V}}\big(\frac{t}{2}\big)\tilde{u}_0
 \Big\|_{X^{\gamma^{2}}}
 \\
 &
 \phantom{aaaaaaaaaaaaaaaaaaaaaaaaaaaaaaaaaaaaaaaa}
 +\|V-\tilde V\|_{\mathcal{L}(X^{\gamma^{0}})}
 \Big)
\Big].
\end{split}
\end{displaymath}
From this, after applying (\ref{eq:(****)-in-cd}) with $j=1$, we obtain
\begin{displaymath}
\|S_{\mu,V}(t)u_0 - S_{\mu,\tilde{V}}(t)\tilde{u}_0\|_{X^{\gamma^{3}}}
\leq \frac{c}{\left(\frac{t}{2}\right)^{\mathtt{r}(\gamma^{3})-\mathtt{r}(\gamma^1)}}
\big(\|u_0 - \tilde{u}_0\|_{X^{\gamma^1}}+\|V-\tilde V\|_{\mathcal{L}(X^{\gamma^{0}})}\big)
\end{displaymath}
that is,  (\ref{eq:(****)-in-cd}) for $j=2$. With a similar procedure
we conclude that (\ref{eq:(****)-in-cd}) holds for every
$j=1,\ldots,M-1$.

This proves the first inequality in the statement. The second follows
from taking $u_0=\tilde{u}_0$ of norm one.
\end{proof}

\subsection{Two perturbations}
\label{sec:two-perturbations}

  Assume now we have two perturbations given by the potentials

  \begin{equation}\label{eq:for-thm-with-2potentials}
     V^{i}\in M^{p_i,\ell_i}(\R^N) \quad \text{ for  $1\leq p_i\leq \infty$,
   \   $\ell_i\in(0,N]$} , \quad i=0,1.
  \end{equation}
  with
  \begin{equation}\label{eq:also-for-thm-with-2potentials}
 \kappa_{i}\mydef     \frac{\ell_{i}}{2m\mu p_{i}}<1 , \quad i=0,1 .
\end{equation}
Without loss of generality, we can assume
\begin{equation}\label{eq:one-more-for-thm-with-2potentials}
\ell_{0} \leq \ell_{1}.
\end{equation}

Then we have the following results concerning evolution problem
\begin{equation}\label{eq:in-Sec-6-linear-e-eq-with-2potentials}
  \begin{cases}
    u_{t} + A_0^\mu u = V^0(x)u + V^1(x)u, &  t>0, \ x\in \R^N ,
    \\
    u(0,x)=u_0(x), &x\in \R^N,
  \end{cases}
\end{equation}
that clearly generalizes to more
than two perturbations.

\begin{theorem}
  \label{thm:linear-e-eq-with-two-potentials}
Let $A_0$ be as in (\ref{eq:operator-A0}), $\mu\in(0,1]$  and
assume (\ref{eq:for-thm-with-2potentials}), (\ref{eq:also-for-thm-with-2potentials}) and (\ref{eq:one-more-for-thm-with-2potentials}).

\begin{enumerate}
\item
  {\bf (The perturbed semigroup)}
For
  $1\leq p \leq \infty$ and $0<\ell \leq \ell_0$,
    (\ref{eq:in-Sec-6-linear-e-eq-with-2potentials})
    defines a semigroup
  $\{S_{\mu,\{V^{0}, V^{1}\}}(t)\}_{t\geq0}$ in $M^{p,\ell}(\R^N)$ such that for
  $u_0\in M^{p,\ell}(\R^N)$, $u(t) \mydef   S_{\mu, \{V^{0}, V^{1}\}}(t)u_0 $
  satisfies
  \begin{displaymath}
    u(t) = S_{\mu}(t) u_{0} + \int_{0}^{t}  S_{\mu}(t-s) V^{0} u(s) \, ds
    + \int_{0}^{t}  S_{\mu}(t-s) V^{1} u(s) \, ds,
    \quad  t>0,
  \end{displaymath}
 \begin{displaymath}
    \lim_{t\to 0^+} \| u(t)- S_{\mu}(t)
    u_0\|_{M^{p,\ell}(\R^N)}= 0 .
  \end{displaymath}
  Also,
  \begin{equation}\label{eq:Mp,ell-estimate-of-S{mu,V0,V1}(t)}
    \|S_{\mu, \{V^{0}, V^{1}\}}(t) \|_{\mathcal{L}(M^{p,\ell}(\R^N))} \leq C e^{\omega
      t}, \quad t\geq 0
  \end{equation}
  for some constants $C$, $\omega$.
  For $p=1$ all the above  remains true if we replace $M^{1,\ell}(\R^N)$ by
  $\mathcal{M}^{\ell}(\R^N)$.

  Moreover, if $p\in[\max\{p_0',p_1'\},\infty]$ then (\ref{eq:Mp,ell-estimate-of-S{mu,V0,V1}(t)}) holds with any $\omega$ satisfying
  \begin{equation}\label{(diamonddiamonddiamonddiamond)-for-2}
    \omega = c\big(
    \| V^{0}\|_{M^{p_0,\ell_0}(\R^N)}^{\frac{1}{1-\kappa_{0}}}
    +
    \| V^{1}\|_{M^{p_1,\ell_1}(\R^N)}^{\frac{1}{1-\kappa_{1}}}
    \big)
  \end{equation}
  for some positive constant $c=c(p,\ell)$.

\item
{\bf (Smoothing properties)}
For
$1\leq p \leq \infty $, $0<\ell \leq \ell_0$, if   $1\leq q\leq \infty$ and
$0< s\leq \ell$ satisfy  $\frac{s}{q}\leq \frac{\ell}{p}$ we have,
for some constants $a$, $C$,
\begin{equation}\label{eq:(diamond)-sp-for-2}
  \| S_{\mu, \{V^{0}, V^{1}\}}(t)\|_{\mathcal{L}(M^{p,\ell}(\R^N), M^{q,s}(\R^N))}
  \leq
  \frac{Ce^{at}}{t^{\frac{1}{2m\mu}(\frac{\ell}{p}-\frac{s}{q})}},
  \quad t>0,
\end{equation}
and
\begin{equation}\label{eq:(diamonddiamond)-sp-for-2}
  (0,\infty)\times M^{p,\ell}(\R^N) \ni (t,u_0) \to S_{\mu, \{V^{0}, V^{1}\}}(t) u_0 \in M^{q,s}(\R^{N})
  \ \text{ is continuous}.
\end{equation}
For $p=1$ this remains true if we replace $M^{1,\ell}(\R^N)$ by
$\mathcal{M}^{\ell}(\R^N)$.

Moreover, if $p\in[\max\{p_0',p_1'\},\infty]$, or $q\in[\max\{p_0',p_1'\},\infty]$, then (\ref{eq:(diamond)-sp-for-2})
holds with any $a$ satisfying
\begin{equation}\label{eq:(diamonddiamonddiamond)-sp-for-2}
  a=
  c
  \big(
  \|V^{0}\|_{M^{p_0,\ell_0}(\R^N)}^{\frac{1}{1-\kappa_{0}}}
  +
  \|V^{1}\|_{M^{p_1,\ell_1}(\R^N)}^{\frac{1}{1-\kappa_{1}}}
  \big)
\end{equation}
  for some positive constant $c=c(p,\ell,q,s)$.

\item
 {\bf (Analyticity)}
 If either

 \begin{enumerate}
 \item
   $\max\{p_0', p_1'\}\leq p \leq\infty$ and $0<\ell\leq\ell_0$, or

\item
  if $1\leq p\leq \min\{p_0, p_1\}$ and $0<\ell\leq\ell_0$
   are such that
   $\max\{\frac{\ell_0}{p_0}, \frac{\ell_1}{p_1}\} \leq\frac{\ell}{p}$
 \end{enumerate}
 the semigroup $\{S_{\mu,\{V^0,V^1\}}(t)\}_{t\geq0}$ is analytic in
   $M^{p,\ell}(\R^N)$ with sectorial generator with the additional
   restriction that $p\not=1$ when $\mu=1$.

\item
   {\bf (The perturbed equation)}
For $1< p \leq \min\{p_0, p_1\}$ and $0<\ell\leq\ell_0$ satisfying
$\max\{\frac{\ell_0}{p_0}, \frac{\ell_1}{p_1}\}\leq  \frac{\ell}{p}$,
we have that $u=S_{\mu,\{V^{0}, V^{1}\}}(\cdot) u_0$ with $u_0\in M^{p,\ell}(\R^N)$
satisfies, for $t>0$,
\begin{displaymath}
u_{t} + A_0^\mu u  = V^0 u + V^1 u \quad \text{in
$M^{p,\ell}(\R^N)$}.
\end{displaymath}

\end{enumerate}
\end{theorem}
\begin{proof}
(i) ({\bf The perturbed semigroup})
Through (\ref{eq:parameters_4_Morrey}),  the perturbation potentials  correspond to Morrey
spaces $V^{i} \in X^{\gamma^{i}}$ with $\gamma^{i} \in
\mathbb{J}$. Since we have assumed  $\ell_{0} \leq \ell_{1}$ then   the slope of
$\gamma^{0}$ is smaller than that of $\gamma^{1}$ and therefore with the
definition (\ref{eq:bound-on-gamma2overgamma1}), the corresponding
triangles satisfy
$\mathbb{J}_{\gamma^0} \subset \mathbb{J}_{\gamma^1}$.

Now we do  sequential  perturbations, following  here  the notations
and  the proof of Theorem  \ref{thm:perturbation-by-a-potential}.

  \noindent {\bf Step 1.}
  We apply Theorem \ref{thm:perturbation-by-a-potential} with the
  perturbation $V^{1}$ so we get the perturbed semigroup
  $\{S_{\mu,V^{1}}(t)\}_{t\geq 0}$ defined in the spaces $X^{\gamma}$ with
  $\gamma \in  \mathbb{J}_{\gamma^{1}}$, that satisfies $ \lim_{t\to 0^+} \| S_{\mu,V^{1}}(t) - S_{\mu}(t)
    u_0\|_{X^{\gamma}}= 0$ for $u_{0}\in X^{\gamma}$.

Also, from Step 7 in the proof of part (i) in Theorem
\ref{thm:perturbation-by-a-potential}, we get that for  $\gamma \in
\mathbb{J}_{\gamma^1}$ such that $\gamma_{1} \leq 1-\gamma_{1}^{1}$
we have that
    $\|S_{\mu,V^{1}}(t) \|_{\mathcal{L}(X^{\gamma})} \leq C
    e^{\omega_{\gamma}^1 t}$ where     $\omega^{1}_{\gamma}=c_{1} \|
    V^{1}\|_{X^{\gamma^{1}}}^{\frac{1}{1-\kappa_{1}}}$ with
    $c_{1}=c_{1}(\gamma)$.

  \medskip
\noindent {\bf Step 2.} Since the perturbed semigroup
  $\{S_{\mu,V^{1}}(t)\}_{t\geq 0}$,  defined in  $X^{\gamma}$ with
  $\gamma \in  \mathbb{J}_{\gamma^{1}}$,  has the same properties in
  this scale  than
  the original unperturbed semigroup, although in the smaller triangle
  $\mathbb{J}_{\gamma^{1}} \subset \mathbb{J}$ (compare
  (\ref{eq:smoothing_morrey_gamma_gamma'}) and
  (\ref{eq:smoothing_morrey_gamma_gamma'_4_perturbed})), we can apply
  again Theorem \ref{thm:perturbation-by-a-potential} to this semigroup with the
  perturbation $V^{0}$, see Remark \ref{rem:summary_1perturbation_morrey}. So  we get the perturbed semigroup
  $\{(S_{\mu,V^{1}})_{V^{0}}(t)\}_{t\geq 0}$ defined in the spaces $X^{\gamma}$ with
  $\gamma \in  \mathbb{J}_{\gamma^{0}} \subset
  \mathbb{J}_{\gamma^{1}}$ and
      $ \lim_{t\to 0^+} \| (S_{\mu,V^{1}})_{V^{0}}(t) u_0 -
    S_{\mu,V^{1}}(t) u_0 \|_{X^{\gamma}}= 0$ for $u_{0}\in
    X^{\gamma}$. In particular,       $ \lim_{t\to 0^+} \|
    (S_{\mu,V^{1}})_{V^{0}}(t) u_0 -     S_{\mu}(t) u_0 \|_{X^{\gamma}}= 0$ for $u_{0}\in X^{\gamma}$.

  As in Step 1 above,
   we also get that for  $\gamma \in
\mathbb{J}_{\gamma^0}$ such that $\gamma_{1} \leq
\min\{1-\gamma_{1}^{0}, 1-\gamma_{1}^{1}\}$
we have that
 \begin{displaymath}
\|(S_{\mu,V^{1}})_{V^{0}} (t) \|_{\mathcal{L}(X^{\gamma})} \leq C
    e^{\omega t}
    \end{displaymath}
where     $\omega>\omega^{2}_{\gamma}=c_{1} \|
    V^{1}\|_{X^{\gamma^{1}}}^{\frac{1}{1-\kappa_{1}}} + c_{2} \|
    V^{0}\|_{X^{\gamma^{0}}}^{\frac{1}{1-\kappa_{0}}}$with
    $c_{i}=c_{i}(\gamma)$.  This gives
    (\ref{eq:Mp,ell-estimate-of-S{mu,V0,V1}(t)}) and
    (\ref{(diamonddiamonddiamonddiamond)-for-2}).

Using (\ref{eq:parameters_4_Morrey}),
(\ref{eq:interpretation-of-Morrey_scale-2m}) the range of $\gamma$
above correspond to $p\in[\max\{p_0',p_1'\},\infty]$  and $0<\ell\leq
\ell_0$.

     \medskip
\noindent {\bf Step 3.} Now we prove that actually
$\{(S_{\mu,V^{1}})_{V^{0}}(t)\}_{t\geq 0}$ coincides with
$\{(S_{\mu,P}(t)\}_{t\geq 0}$ with $P=\{V^{0}, V^{1}\}$, in the spaces $X^\gamma$ for $\gamma\in\mathbb{J}_{\gamma^0}$. For
this we employ Proposition \ref{prop:iterated-perturbations}.

For this, given $\gamma \in  \mathbb{J}_{\gamma^{0}}$ we must produce
$\alpha \in  \mathbb{J}_{\gamma^{0}}$ such that $\gamma \in
\mathcal{E}_{\alpha}$, $\beta_i=\alpha+\gamma^i\in \mathbb{J}$,   $\beta_{i} \stackrel{_{S(t)}}{\leadsto}
\alpha$ and
$0\leq d(\alpha,\beta_{i})=\mathtt{r}(\alpha)-\mathtt{r}(\beta_{i})<1$
for $i=0,1$ so the perturbations  $P_i=V^i\in\mathscr{P}_{\beta_i,R}$.

From Step 1 in the proof of part (i) in Theorem
\ref{thm:perturbation-by-a-potential}, the conditions for $\beta_i=\alpha+\gamma^i\in \mathbb{J}$ and $\beta_{i} \stackrel{_{S_\mu(t)}}{\leadsto}
\alpha$ for $i=0,1$  read
\begin{equation}\label{eq:the-1st-relation-needed-here}
  \alpha_1+\gamma_1^i \leq 1
\end{equation}
and
\begin{equation}\label{eq:the-2nd-relation-needed-here}
    \alpha_2\leq (\beta_2)_{i} =\alpha_{2}+ \gamma^{i}_{2}  \quad \text{ and } \quad
    \frac{\alpha_2}{\alpha_1} \leq \frac{(\beta_2)_{i} }{(\beta_1)_{i} } =
    \frac{\alpha_{2}+ \gamma^{i}_{2} }{\alpha_{1}+  \gamma^{i}_{1}}
  \end{equation}
  and are satisfied taking
  \begin{equation}\label{eq:the-3rd-relation-needed-here}
    \alpha=\begin{cases}\gamma & \text{ if } \gamma_1\leq \theta \\
    (\theta, \frac{\gamma_2}{\gamma_1} \theta) & \text{ if } \gamma_1>
    \theta
    \end{cases} \qquad \text{ where } \ \theta:=\min\{1-\gamma_1^0,
    1-\gamma_1^1\}.
\end{equation}
Indeed, for $\alpha$ as in (\ref{eq:the-3rd-relation-needed-here}) both
(\ref{eq:the-1st-relation-needed-here}) and the first condition  in
(\ref{eq:the-2nd-relation-needed-here}) are clearly satisfied, whereas
the second condition  in (\ref{eq:the-2nd-relation-needed-here}) is so,
  because $\gamma \in   \mathbb{J}_{\gamma^{0}}\subset
  \mathbb{J}_{\gamma^1}$.
  Also $0\leq  d(\alpha,\beta_{i}) <1$   holds because $\kappa_{i} <1$.

  From  Step 2 in the proof of part (i) in Theorem
\ref{thm:perturbation-by-a-potential} we see that
  $\gamma\in \mathcal{E}_\alpha$ if and only if
\begin{equation}\label{eq:the-4th-relation-needed-here}
  \alpha_{2} \leq \gamma_{2} < \alpha_{2}+1, \qquad \frac{\alpha_{2}}{\alpha_{1}} \leq \frac{\gamma_{2}}{
  \gamma_{1}}.
\end{equation}
For $\alpha$ as in (\ref{eq:the-3rd-relation-needed-here}) the first
and third inequalities  in (\ref{eq:the-4th-relation-needed-here}) are
clearly satisfied, whereas the second
inequality needs to be checked  only when $\gamma_1> \theta$ and then  $\alpha=(\theta,
\frac{\gamma_2}{\gamma_1} \theta)$ and in this case we need to justify
that
$\gamma_2<\frac{\gamma_2}{\gamma_1} \theta +1$. For this
observe that if $i$ is such that $\theta=1-\gamma_1^i$ then, using
that $\gamma_2^i=\kappa_{i}<1$, we have
\begin{displaymath}
\frac{\gamma_2}{\gamma_1} \theta +1
>
\frac{\gamma_2}{\gamma_1} (1-\gamma_1^i) +\kappa_i
=
\frac{\gamma_2}{\gamma_1} (1-\gamma_1^i)
+\frac{\gamma_2^i}{\gamma_1^i}\gamma_1^i
\geq
\frac{\gamma_2}{\gamma_1} (1-\gamma_1^i)
+\frac{\gamma_2}{\gamma_1}\gamma_1^i
= \frac{\gamma_2}{\gamma_1} \geq \gamma_2
\end{displaymath}
where we have used that $\gamma \in
\mathbb{J}_{\gamma^{0}}\subset \mathbb{J}_{\gamma^1}$ and $\gamma_1\leq
1$.

Therefore given   $\gamma$ in the triangle  $\mathbb{J}_{\gamma^{0}}$
we can always choose  $\alpha \in  \mathbb{J}_{\gamma^{0}}$ like in
(\ref{eq:the-3rd-relation-needed-here}) above, and so employ Proposition
\ref{prop:iterated-perturbations}.
Part (i) is thus proved.

\noindent
(ii) {\bf (Smoothing properties)}
We now observe that the proof of part (ii) in Theorem
\ref{thm:perturbation-by-a-potential} can be repeated line by line
perturbing  $\{S_{\mu,V^{1}}(t)\}_{t\geq 0}$, defined in the spaces of
the triangle $\mathbb{J}_{\gamma^{1}}$, with $V^{0}$ until
(\ref{eq:smoothing_morrey_gamma_gamma'_4_perturbed}) and
therefore, since from Step 3 we have $S_{\mu, \{V^{0},
  V^{1}\}}(t)=(S_{\mu,V^{1}})_{V^{0}}(t)$,  we get
\begin{displaymath}
\gamma\stackrel{_{S_{\mu, \{V^{0}, V^{1}\}}(t)}}{\leadsto} \gamma' \qquad \text{ for } \gamma,\gamma' \in \mathbb{J}_{\gamma^0} \ \text{ such that }
\gamma'_2 \leq \gamma_2,  \
     \frac{\gamma'_2}{\gamma'_1}  \leq
    \frac{\gamma_2}{\gamma_1}.
\end{displaymath}
Using Lemma
 \ref{lem:estimates4semigroups_in_scales}, this gives
(\ref{eq:(diamond)-sp-for-2}) and
(\ref{eq:(diamonddiamond)-sp-for-2}) and  (\ref{eq:(diamonddiamonddiamond)-sp-for-2}).
 This completes the proof of part (ii).

 \medskip
\noindent (iii)  {\bf (Analyticity)}
From Lemma \ref{lem:abstract_setup_4_Morrey}  the
unperturbed semigroup is analytic in $X^{\gamma}$ for $\gamma \in
\mathbb{J}$ with the additional restriction that $\gamma_{1}<1$ if $\mu=1$ that we take into account, without mentioning it further.
For the perturbed semigroup, we now apply
Theorem \ref{thm:Xalpha-analiticity-of-linearly-perturbed-semigroup}, which can be done provided that in Step 3 above we can take
$\alpha =\gamma$. This, in turn, can be done provided that
$\gamma_1\leq \min\{1-\gamma_1^0,
1-\gamma_1^1\}$. This gives the part (a) in the statement.

To prove part (b) we consider sequential perturbations. We apply first Theorem
\ref{thm:perturbation-by-a-potential}(iii)(b)
to get analyticity of
$\{S_{\mu,V^1}(t)\}_{t\geq0}$ in $X^\gamma$ for $\gamma$ satisfying
\begin{equation}\label{eq:first-for-gamma-in-analyticity(b)-for-2-perturbation}
\gamma\in\mathbb{J}_{\gamma^1}, \ \gamma_1^1\leq\gamma_1\leq 1 \ \text{ and } \gamma_2^1\leq\gamma_2 \leq \frac{N}{2m \mu}.
\end{equation}
Then we proceed similarly to conclude analyticity of
$\{(S_{\mu,V^1})_{V^0}(t)\}_{t\geq0}$ in $X^\gamma$ for $\gamma$ which besides (\ref{eq:first-for-gamma-in-analyticity(b)-for-2-perturbation}) satisfy also
\begin{equation}\label{eq:second-for-gamma-in-analyticity(b)-for-2-perturbation}
\gamma\in\mathbb{J}_{\gamma^0}, \ \gamma_1^0\leq \gamma_1 \leq 1\ \text{ and } \gamma_2^0<\gamma_2\leq \frac{N}{2m \mu}.
\end{equation}
Since $\mathbb{J}_{\gamma^0}\subset \mathbb{J}_{\gamma^1}$,
combining (\ref{eq:first-for-gamma-in-analyticity(b)-for-2-perturbation})
and (\ref{eq:second-for-gamma-in-analyticity(b)-for-2-perturbation})
gives the statement in (b).

\medskip

\noindent (iv) {\bf (The perturbed equation)}
Using sequential perturbations we first see via Theorem \ref{thm:perturbation-by-a-potential}(iv) that
$u=S_{\mu,V^1}(\cdot) u_0$ with $u_0\in X^\gamma$
and $\gamma$ in (\ref{eq:first-for-gamma-in-analyticity(b)-for-2-perturbation}) satisfying $\gamma_1<1$
solves, for $t>0$,
\begin{displaymath}
u_{t} + A_0^\mu u  = V^1 u  \quad \text{in
$X^\gamma$},
\end{displaymath}
and then
$u=(S_{\mu,V^1})_{V^0}(\cdot) u_0$ with $u_0\in X^\gamma$ and $\gamma$ satisfying in addition (\ref{eq:second-for-gamma-in-analyticity(b)-for-2-perturbation})
solves, for $t>0$,
\begin{displaymath}
u_{t} + (A_0^\mu -V^1) u  = V^0 u  \quad \text{in
$X^{\gamma}$} ,
\end{displaymath}
because for the considered parameters
the sectorial generator of $\{S_{\mu,V^1}(t)\}_{t\geq0}$ is $-A_0^\mu + V^1 $.
\end{proof}

\begin{remark}
  Observe that if say $V^{1}\in L^{\infty}(\R^{N})$ then we can set  $p_{1}=\infty$
  $\ell_{1}=N$ so the restrictions in Theorem
  \ref{thm:linear-e-eq-with-two-potentials} come from $V^{0}\in
  M^{p_{0}, \ell_{0}} (\R^{N})$.
\end{remark}

The following result concerns  continuous dependence with respect
to two perturbations.

\begin{theorem}
  Assume $V^{ i},\tilde{V}^{ i}\in M^{p_i,\ell_i}(\R^N)$ for
  $1\leq p_i\leq \infty$ and $\ell_i\in(0,N]$ which satisfy
  (\ref{eq:also-for-thm-with-2potentials}) and
  (\ref{eq:one-more-for-thm-with-2potentials}).
  Define the region of parameters $(p,\ell)$ for Morrey spaces as
\begin{equation} \label{eq:region_4_continuous_dependence_2perturbations}
1\leq p \leq \infty, \qquad     \ell \leq \ell_{0}, \qquad \ell
     \big( \frac{1}{p} - \frac{1}{ p'_{0}\vee p'_{1}} \big) \leq \frac{\ell_{0}}{p_{0}} \wedge
    \frac{\ell_{1}}{p_{1}} .
  \end{equation}

  If
  \begin{displaymath}
    \|V\|_{\mathcal{L}(M^{p_i,\ell_i}(\R^N))},
    \|\tilde{V}\|_{\mathcal{L}(M^{p_i,\ell_i}(\R^N))}\leq R \quad
    \text{ for } \ i=0,1,
  \end{displaymath}
  and
  \begin{displaymath}
    \|u_0\|_{M^{p,\ell}(\R^N)}, \|\tilde{u}_0\|_{M^{p,\ell}(\R^N)}\leq \mathscr{R}
  \end{displaymath}
  then for $T>0$, $1\leq q\leq \infty$ and $0<s\leq \ell$ satisfying
  $\frac{s}{q}\leq \frac{\ell}{p}$ and $(q,s)$ belong to the region
  (\ref{eq:region_4_continuous_dependence_2perturbations}),  we have in $(0,T]$
  \begin{displaymath}
    \begin{split}
      \| S_{\mu, \{V^0,V^1\}}(t)u_{0} &-
      S_{\mu,\{\tilde{V}^0,\tilde{V}^1\}}(t)\tilde{u}_0\|_{M^{q,s}(\R^N)}
      \\
      & \qquad \leq
      \frac{C_{0}}{t^{\frac{1}{2m\mu}(\frac{\ell}{p}-\frac{s}{q})}}
      \big( \|u_0-\tilde{u}_0 \|_{M^{p,\ell}(\R^N)} +
      \max_{i\in\{0,1\}}\|V^i-\tilde{V}^i\|_{\mathcal{L}(M^{p_i,\ell_i}(\R^N))}
      \big)
    \end{split}
  \end{displaymath}
  and
  \begin{displaymath}
    \begin{split}
      \|&S_{\mu, \{V^0,V^1\}}(t) -
      S_{\mu,\{\tilde{V}^0,\tilde{V}^1\}}(t)\|_{{\mathcal
          L}(M^{p,\ell}(\R^N), M^{q,s}(\R^N))} \leq
      \frac{C_1}{t^{\frac{1}{2m\mu}(\frac{\ell}{p}-\frac{s}{q})}}
      \max_{i\in\{0,1\}}\|V^i-\tilde{V}^i\|_{\mathcal{L}(M^{p_i,\ell_i}(\R^N))}
      ,
    \end{split}
  \end{displaymath}
  where $C_0$, $C_1$ depend on $p_i,\ell_i, p, \ell, q, s, R$ and
  $T$. Additionally $C_0$ depends on $\mathscr{R}$.
\end{theorem}

\begin{proof}
  For this result we will apply Theorem
  \ref{thm:Lipschitz-properties-of-S_P(t)}. For this we denote below
  $a\vee b= \max\{a,b\}$ and $a \wedge b= \min\{a,b\}$.

\medskip
\noindent {\bf Step 1.}
Let $\gamma \in \mathbb{J}_{\gamma^{0}}$.
From Step 3 in the proof of part (i) Theorem
\ref{thm:linear-e-eq-with-two-potentials}, see
(\ref{eq:the-1st-relation-needed-here}),
(\ref{eq:the-2nd-relation-needed-here}), (\ref{eq:the-4th-relation-needed-here}),
we have
$P=\{V^0,V^1\} \in \mathscr{P}_{\beta_0,\beta_1,R}$ and
$\tilde{P}=\{\tilde{V}^0,\tilde{V}^1\}\in
\mathscr{P}_{\beta_0,\beta_1,R}$,
provided that
\begin{equation}\label{eq:auxiliary1-for-cont-depend-for-2-perturbations}
\begin{cases}
\alpha_1+\gamma^i_1\leq 1
\\
\alpha_2\leq \alpha_2 +\gamma_2^i, \qquad \  i\in\{0,1\}
    \\
    \frac{\alpha_2}{\alpha_1} \leq
    \frac{\alpha_2 +\gamma_2^i}{\alpha_1 +\gamma_1^i}
  \end{cases}
\end{equation}
and $\gamma\in\mathcal{E}_\alpha$  provided
\begin{equation}\label{eq:auxiliary2-for-cont-depend-for-2-perturbations}
\begin{cases}
\alpha_2\leq \gamma_2 < \alpha_2 +1
    \\
    \frac{\alpha_2}{\alpha_1} \leq
    \frac{\gamma_2}{\gamma_1}.
  \end{cases}
\end{equation}
Notice the second condition in
(\ref{eq:auxiliary1-for-cont-depend-for-2-perturbations}) is always
satisfied and the third one is satisfied if $\alpha \in
\mathbb{J}_{\gamma^{0}}$. So we need to solve
(\ref{eq:auxiliary2-for-cont-depend-for-2-perturbations}) for $\alpha \in
\mathbb{J}_{\gamma^{0}}$ with $\alpha_{1}
\leq \theta := 1-(\gamma^{0}_{1} \vee \gamma^{1}_{1})$.  For this we take
$\alpha$ with the same slope than $\gamma$ as in
(\ref{eq:the-3rd-relation-needed-here}), that is,
\begin{equation} \label{eq:alpha_4_2perturbations}
    \alpha=\begin{cases}\gamma & \text{ if } \gamma_1\leq \theta \\
    (\theta, \frac{\gamma_2}{\gamma_1} \theta) & \text{ if } \gamma_1>
    \theta
    \end{cases}
\end{equation}
and all conditions are met, as
we showed below  (\ref{eq:the-4th-relation-needed-here}).

\medskip
\noindent {\bf Step 2.}
Since $\mathcal{R}_{\beta_0,\beta_1}= \mathcal{R}_{\beta_0}\cap \mathcal{R}_{\beta_1}$, from Step 3 in the proof of Theorem \ref{thm:perturbation-by-a-potential} we have $\gamma'\in \mathcal{R}_{\beta_0,\beta_1}$ provided that
\begin{equation}\label{eq:auxiliary3-for-cont-depend-for-2-perturbations}
  \begin{cases}
  0\leq \gamma'_2
  \\
  \alpha_{2} - j_{i}<   \gamma'_2 \leq \alpha_{2}+ \gamma^{i}_{2}
  \ \ \text{ where } \ j_i:=1-\gamma_2^i>0, \ i\in\{0,1\} \
    \\
     \frac{\gamma'_2}{\gamma'_1}  \leq
    \frac{\alpha_{2}+ \gamma^{i}_{2}}{\alpha_{1}+ \gamma^{i}_{1}},
  \end{cases}
  \end{equation}
while from  Lemma \ref{lem:abstract_setup_4_Morrey} we have $\gamma
\stackrel{_{S_\mu(t)}}{\leadsto} \gamma'$ provided that
\begin{equation}\label{eq:auxiliary4-for-cont-depend-for-2-perturbations}
  \begin{cases}
    \gamma'_2 &\leq \gamma_2
   \\
     \frac{\gamma'_2}{\gamma'_1} & \leq
    \frac{\gamma_2}{\gamma_1}
  \end{cases}
\end{equation}
and in particular $\gamma'\in \mathbb{J}_{\gamma^{0}}$.
  So, for $\gamma \in \mathbb{J}_{\gamma^{0}}$, $\alpha$ as in
  (\ref{eq:alpha_4_2perturbations}) and $\gamma'$ as above, we can
  apply Theorem
  \ref{thm:Lipschitz-properties-of-S_P(t)} to get the result.

  Now observe that the second condition in
  (\ref{eq:auxiliary4-for-cont-depend-for-2-perturbations}) and the
  choice of $\alpha$ implies the third condition in
  (\ref{eq:auxiliary3-for-cont-depend-for-2-perturbations}). So, given
  the choice of $\alpha$ in (\ref{eq:alpha_4_2perturbations}), the
  conditions  for $\gamma'$ are given by
  (\ref{eq:auxiliary4-for-cont-depend-for-2-perturbations}),
  $0<\gamma'_{2}$ and either
  \begin{displaymath}
    \gamma_{2}+ \gamma^{i}_{2}-1
    < \gamma_{2}' \leq \gamma_{2}, \quad \
    i\in\{0,1\} \ \qquad \text{if
      $\gamma_{1}\leq \theta$}
  \end{displaymath}
  or
  \begin{displaymath}
   \frac{\gamma_2}{\gamma_1} \theta   + \gamma^{i}_{2}-1  <
   \gamma_{2}' \leq   \frac{\gamma_2}{\gamma_1} \theta   +
   \gamma^{i}_{2} \quad \
    i\in\{0,1\} \ \qquad \text{if
      $\gamma_{1} > \theta$} .
  \end{displaymath}
These conditions can be recast as
  \begin{displaymath}
    \gamma_{2}+ (\gamma^{0}_{2} \vee \gamma^{1}_{2})-1
    < \gamma_{2}' \leq \gamma_{2},  \qquad \text{if
      $\gamma_{1}\leq \theta$}
  \end{displaymath}
  or
  \begin{displaymath}
   \frac{\gamma_2}{\gamma_1} \theta   + (\gamma^{0}_{2} \vee \gamma^{1}_{2})-1  <
   \gamma_{2}' \leq   \gamma_{2} \wedge \Big( \frac{\gamma_2}{\gamma_1} \theta   +
  (\gamma^{0}_{2} \wedge \gamma^{1}_{2})\Big) \qquad \text{if
      $\gamma_{1} > \theta$}
  \end{displaymath}
  and notice that actually $\frac{\gamma_2}{\gamma_1} \theta   +
  (\gamma^{0}_{2} \vee \gamma^{1}_{2})-1 < \gamma_{2}$, since
  $\frac{\gamma_2}{\gamma_1}\theta < \gamma_{2}$ if $\gamma_1>\theta$, and $ (\gamma^{0}_{2} \vee
  \gamma^{1}_{2})-1 <0$ so this second condition is non void.

\medskip
\noindent {\bf Step 3.}
Now, when $\gamma_{1} > \theta$, we restrict to the region in
$\mathbb{J}_{\gamma^{0}}$ such that
$\gamma_{2} \leq  \frac{\gamma_2}{\gamma_1} \theta   +
  (\gamma^{0}_{2} \wedge \gamma^{1}_{2})$,
that is
\begin{displaymath}
    \gamma_{2} \leq h(\gamma_{1})= \gamma^{0}_{2} \wedge
    \gamma^{1}_{2} + \frac{(\gamma^{0}_{2} \wedge \gamma^{1}_{2})
      \theta}{\gamma_{1}-\theta}, \quad \theta< \gamma_{1} \leq 1
\end{displaymath}
so $h$ is convex, decreasing $h(\theta)=\infty$ and $h(1)= \gamma^{0}_{2} \wedge
    \gamma^{1}_{2} + \frac{(\gamma^{0}_{2} \wedge
      \gamma^{1}_{2})}{\gamma^{0}_{1} \vee \gamma^{1}_{1}}
    (1-\gamma^{0}_{1} \vee \gamma^{1}_{1}) = \frac{\gamma^{0}_{2} \wedge
      \gamma^{1}_{2}}{\gamma^{0}_{1} \vee \gamma^{1}_{1}}$.

    Hence we define the subset $\mathbb{J}_{\gamma^{0}}^{*} = \{\gamma
    \in \mathbb{J}_{\gamma^{0}}, \ \gamma_{2} \leq h(\gamma_{1}), \
    \text{if $\theta< \gamma_{1} \leq 1$}\}$ and  prove that  for
    $\gamma,\gamma' \in \mathbb{J}_{\gamma^{0}}^{*}$  satisfying
    (\ref{eq:auxiliary4-for-cont-depend-for-2-perturbations})
    we can apply Theorem  \ref{thm:Lipschitz-properties-of-S_P(t)}.

First note that, with minor changes,  the bootstrap of estimates along segments that we
performed in Step 10 of the proof of Theorem
\ref{thm:perturbation-by-a-potential} and in Step 2 of the proof of
Theorem \ref{thm:linear-e-eq-with-two-potentials}, allows us in this
case to get the estimates in Theorem
\ref{thm:Lipschitz-properties-of-S_P(t)} if  $\gamma, \gamma' \in \mathbb{J}_{\gamma^{0}}^{*}$ satisfy
(\ref{eq:auxiliary4-for-cont-depend-for-2-perturbations}) and  the
segment joining them is contained  in $\mathbb{J}_{\gamma^{0}}^{*}$.

In particular it remains to consider the case  $\gamma_{1}<\gamma'_{1}$ and  the
segment crosses the graph of $h$ precisely in two points with
coordinates $\gamma_{1} \leq a < b \leq \gamma'_{1}$.

Then we can apply  Lemma \ref{lem:exterior_tangent} in $[\gamma_{1}-\delta,
b]$ with $c=\gamma_{1}$ and $\delta>0$ small,
to get a tangent to $h$ to the right through $\gamma$. Also we  can
apply  Lemma \ref{lem:exterior_tangent} in $[a, \gamma'_{1}+\delta]$ with $c=\gamma'_{1}$ and $\delta>0$ small,
to get a tangent to $h$ to the left through $\gamma'$.
Then these two tangents will cross at some point in
$\mathbb{J}_{\gamma^{0}}^{*}$. This point and $\gamma, \gamma'$
determine  two segments and along each one  we can do the
bootstrap. So the claim is proved.

\medskip
\noindent {\bf Step 4.}
It remains to describe the set  $\mathbb{J}_{\gamma^{0}}^{*} $ in
terms of the original parameters of the Morrey scale.  For this notice
that from (\ref{eq:parameters_4_Morrey}) we have that $\theta =
1-(\gamma^{0}_{1} \vee \gamma^{1}_{1}) = \frac{1}{p'_{0}\vee p'_{1}}$ and $\gamma^{0}_{2} \wedge
    \gamma^{1}_{2} =\frac{1}{2m\mu} \frac{\ell_{0}}{p_{0}} \wedge
    \frac{\ell_{1}}{p_{1}} $.

    Hence when $\gamma_{1} \leq  \theta$ the region reads
    \begin{displaymath}
      p\geq     p'_{0}\vee p'_{1}, \qquad \ell \leq \ell_{0}
  \end{displaymath}

On the other hand, when $\gamma_{1} > \theta$
    the region $\gamma_{2} \leq  \frac{\gamma_2}{\gamma_1} \theta   +
  (\gamma^{0}_{2} \wedge \gamma^{1}_{2})$ reads
  \begin{displaymath}
     p <  p'_{0}\vee p'_{1} , \qquad \ell \leq \ell_{0}, \qquad \ell
     \big( \frac{1}{p} - \frac{1}{ p'_{0}\vee p'_{1}} \big) \leq \frac{\ell_{0}}{p_{0}} \wedge
    \frac{\ell_{1}}{p_{1}} .
  \end{displaymath}
  Hence both conditions can be summarised as
  (\ref{eq:region_4_continuous_dependence_2perturbations}).
\end{proof}

\begin{remark}
  \begin{enumerate}
  \item
    If    $\gamma \in \mathbb{J}_{\gamma^{0}} \setminus
    \mathbb{J}_{\gamma^{0}}^{*}$ then there is a ``threshold'' in the
    possible jumps from $\gamma$ to $\gamma'$ as
    \begin{displaymath}
      \gamma'_{2} \leq \frac{\gamma_2}{\gamma_1} \theta   +
      (\gamma^{0}_{2} \wedge \gamma^{1}_{2}) < \gamma_{2} .
    \end{displaymath}
    In particular the estimates in the theorem can not be obtained for
    $\gamma'=\gamma$.

  \item
    Condition (\ref{eq:region_4_continuous_dependence_2perturbations})
    is satisfied if in particular
    \begin{displaymath}
      \frac{\ell}{p}\leq \min\big\{\frac{\ell_0}{p_0}, \frac{\ell_1}{p_1}\big\}.
    \end{displaymath}
    or if
      \begin{displaymath}
      p\geq     p'_{0}\vee p'_{1}, \qquad \ell \leq \ell_{0} .
  \end{displaymath}

In both cases the Theorem holds if both $(p,\ell)$ and $(q,s)$ satisfy
either one of them and $0<s\leq \ell$,
$\frac{s}{q}\leq \frac{\ell}{p}$.

\end{enumerate}
\end{remark}

\begin{lemma} [{\bf Exterior tangent lemma}]
  \label{lem:exterior_tangent}

  Let $f:[a,b] \to \R$ a $C^{2}$ convex function with $f''(x)>0$ in
  $[a,b]$ and let $c\in (a,b)$ and $d\in \R$. Then

  \begin{enumerate}
  \item (Tangent to the left)
There exist (a unique) $x_{*} \in [a,c)$ such that the tangent to $f$
through $x_{*}$ passes through $(c,d)$ if and only if $d \leq f(c)$
and $\frac{f(a)-d}{a-c} \geq f'(a)$.

\item
(Tangent to the right)
There exist (a unique) $x_{*} \in (c,d]$ such that the tangent to $f$
through $x_{*}$ passes through $(c,d)$ if and only if $d \leq f(c)$
and $\frac{f(b)-d}{b-c} \leq f'(b)$.

  \end{enumerate}
\end{lemma}
\begin{proof}
  For $x\in [a,b]$ the value of the tangent to $f$ at $x$ in  the point $c$ is
  given by $t(x)= f(x) + f'(x)(c-x)$ and $t'(x)= f''(x)(c-x)$.

  Therefore, since $f''>0$ we have that $t' >0$ in $[a,c)$ and $t'<0$
  in $(c,b]$. Hence, $t(x)=d$ for (a unique)  some $x$ if and only
  if either $t(a)\leq d\leq t(c)= f(c)$ or  $t(b) \leq d\leq
  t(c)=f(c)$ which gives the result.
\end{proof}

\bibliographystyle{plain}

\end{document}